\documentclass[11pt,reqno]{amsart}

\usepackage{amsfonts}
\usepackage{eurosym}
\usepackage{amssymb}
\usepackage{amsthm}
\usepackage{amsmath}
\usepackage{amsaddr}
\usepackage{bm}
\usepackage{cite}
\usepackage{mathrsfs}
\usepackage{xcolor}
\usepackage[OT1]{fontenc}
\usepackage[left=2.5cm, right=2.5cm, top=3cm, bottom=2.5cm]{geometry}
\usepackage{hyperref}
\usepackage{mathtools}
\hypersetup{colorlinks=true, linkcolor=blue, citecolor=red, urlcolor=blue}
\usepackage{comment}
\usepackage[T1]{fontenc}

\usepackage{times}

\usepackage{enumitem}
\usepackage{xpatch}
\makeatletter
\AtBeginDocument{\xpatchcmd{\@thm}{\thm@headpunct{.}}{\thm@headpunct{}}{}{}}
\makeatother

\allowdisplaybreaks
\newtheorem{theorem}{Theorem}[section]

\newtheorem{corollary}[theorem]{Corollary}
\newtheorem{definition}[theorem]{Definition}
\newtheorem{lemma}[theorem]{Lemma}
\newtheorem{proposition}[theorem]{Proposition}

\theoremstyle{remark}
\newtheorem{remark}[theorem]{Remark}

\numberwithin{equation}{section}

\usepackage{parskip}
\newcommand{\abs}[1]{\left| #1 \right|}

\newcommand{\norm}[1]{\| #1 \|}
\newcommand{\bignorm}[1]{\big\| #1 \big\|}

\newcommand{\ang}[2]{ \langle #1 , #2  \rangle}
\newcommand{\bigang}[2]{ \big< #1 , #2  \big>}

\newcommand{\scp}[2]{ \left( #1 , #2  \right)}
\newcommand{\bigscp}[2]{\big( #1 , #2 \big)}

\newcommand{\meano}[1]{{\langle #1 \rangle}_{\Omega}}
\newcommand{\meang}[1]{{\langle #1 \rangle}_{\Gamma}}
\newcommand{\mean}[2]{\textnormal{mean}\scp{#1}{#2}}

\newcommand{\R}{\mathbb R}
\newcommand{\N}{\mathbb N}
\newcommand{\n}{\mathbf{n}}

\newcommand{\intO}{\int_\Omega}

\newcommand{\intG}{\int_\Gamma}

\newcommand{\dtau}{\;\mathrm d\tau}

\newcommand{\dx}{\;\mathrm{d}x}

\newcommand{\ds}{\;\mathrm ds}

\newcommand{\dxs}{\;\mathrm{d}x\;\mathrm{d}s}
\newcommand{\dGs}{\;\mathrm{d}\Ga\;\mathrm{d}s}

\newcommand{\dG}{\;\mathrm d\Ga}

\newcommand{\ddt}{\frac{\mathrm d}{\mathrm dt}}

\newcommand{\del}{\partial}
\newcommand{\delt}{\partial_{t}}
\newcommand{\delth}{\partial_{t}^{h}}

\newcommand{\deln}{\partial_\n}

\newcommand{\Grad}{\nabla}
\newcommand{\Lap}{\Delta}
\newcommand{\Div}{\textnormal{div}}
\newcommand{\Gradg}{\nabla_\Ga}
\newcommand{\Lapg}{\Delta_\Ga}
\newcommand{\Divg}{\textnormal{div}_\Ga}

\newcommand{\emb}{\hookrightarrow}

\newcommand{\suchthat}{\;\ifnum\currentgrouptype=16 \middle\fi|\;}

\newcommand{\Om}{\Omega}
\newcommand{\Ga}{\Gamma}

\definecolor{rosso}{rgb}{0.8,0,0}
\definecolor{violet}{rgb}{0.65,0,0.65}
\definecolor{darkgreen}{rgb}{0,0.5,0}

\DeclareMathOperator*{\esssup}{ess\,sup}

\def \no#1#2#3 {{\bf #1} (#3), #2.}
\def \eds#1#2#3 {#1, #2, #3.}

\makeatletter
\def\@settitle{\begin{center}%
  \baselineskip14\p@\relax
    \huge
  \@title
  \end{center}%
}
\makeatother

\begin{document}

\title[Regularity of a convective bulk-surface Cahn--Hilliard model]
{\emph{Regularity of a bulk-surface
Cahn--Hilliard model \\ driven by Leray velocity fields}}
\author[Andrea Giorgini, Patrik Knopf \& Jonas Stange]{Andrea Giorgini, Patrik Knopf \& Jonas Stange}

\address{Politecnico di Milano\\
Dipartimento di Matematica\\
Via E. Bonardi 9, 20133 Milano, Italy\\
\href{mailto:andrea.giorgini@polimi.it}{andrea.giorgini@polimi.it}}

\address{Fakult\"{a}t f\"{u}r Mathematik\\ 
Universit\"{a}t Regensburg \\
93040 Regensburg, Germany\\
\href{mailto:patrik.knopf@ur.de}{patrik.knopf@ur.de}\\
\href{mailto:jonas.stange@ur.de}{jonas.stange@ur.de}
}


%

%


\begin{abstract}
We consider a convective bulk-surface Cahn--Hilliard system with dynamic boundary conditions and singular potentials. For this model, well-posedness results concerning weak and strong solutions have already been established in the literature. However, they require the prescribed velocity fields to belong to function spaces with high time regularity. In this paper, we prove that the well-posedness of weak solutions holds true under more general regularity assumptions on the velocity fields. Next, via an alternative proof for higher regularity, we show the well-posedness of strong solutions for velocity fields of Leray type, which is a more relevant assumption for physical applications. Our approach hinges upon a new well-posedness and regularity theory for a bulk-surface elliptic system with singular nonlinearities, which may be of independent interest.
\end{abstract}

\maketitle

\section{Introduction}
\label{SECT:INTRO}

The Cahn--Hilliard equation was originally introduced to model spinodal decomposition processes in binary alloys. Later, it turned out that the Cahn--Hilliard equation can also be used to describe later stages of the evolution of phase transition phenomena such as Ostwald ripening.
Nowadays,  it is the most popular diffuse interface model with diverse applications in materials science, life sciences and image processing. 
A crucial advantage of Cahn--Hilliard models compared with sharp-interface models (free boundary problems) is the purely Eularian formulation, which does not require any direct tracking of the interfaces separating the involved materials. As such, topological changes of the interfaces can be handled without imposing artificial conditions.

The Cahn--Hilliard equation describes the time evolution of a phase-field $\phi$, with associated chemical potential $\mu$. The function $\phi$ is a conserved order parameter, which typically represents the concentration difference of the two components of a binary mixture.
In the standard formulation of the Cahn--Hilliard system in a bounded domain, homogeneous Neumann boundary conditions for both $\phi$ and $\mu$ are classically imposed in the study of the initial-boundary value problem. However, as already been observed by mathematicians and physicists (see, e.g., \cite{Binder1991,Fischer1997,Fischer1998,Kenzler2001}), these boundary conditions entail crucial limitations when a precise description of the dynamics close to the boundary is required.
The main drawback is that the homogeneous Neumann boundary condition on the phase-field enforces the diffuse interfaces between the two components to intersect the boundary of the domain at a perfect right angle. In most applications, this is very unrealistic as the contact angle is expected to deviate from ninety degrees, and to even change dynamically over the course of time. Furthermore, the homogeneous Neumann boundary condition on the chemical potential prohibits any transfer of material between the bulk and the boundary. In some applications this is actually desired as the bulk mass is supposed to be conserved. On the other hand, when absorption processes or chemical reactions at the boundary are considered, an exchange of material between bulk and boundary needs to be included in the formulation.

Due to these limitations of the classical Cahn--Hilliard problem, and also to describe short-range interactions between bulk and surface more effectively, 
several Cahn--Hilliard models with dynamic boundary conditions have been proposed and analyzed in the literature. In particular, an evolutionary equation of Cahn--Hilliard type on the boundary has become very popular in recent years (see, e.g., \cite{Goldstein2011,Liu2019,Knopf2021a,Knopf2020}), which also account for a transfer of material between bulk and boundary. These models are sometimes also referred as bulk-surface Cahn--Hilliard systems.
In this paper, we study the following convective bulk-surface Cahn--Hilliard model:%
\begin{subequations}\label{EQ:SYSTEM}
    \begin{align}
        \label{EQ:SYSTEM:1}
        &\delt\phi + \Div(\phi\boldsymbol{v}) = \Div(m_\Om(\phi)\Grad\mu) && \text{in} \ Q, \\
        \label{EQ:SYSTEM:2}
        &\mu = -\Lap\phi + F'(\phi)   && \text{in} \ Q, \\
        \label{EQ:SYSTEM:3}
        &\delt\psi + \Divg(\psi\boldsymbol{w}) = \Divg(m_\Ga(\psi)\Gradg\theta) - \beta m_\Om(\phi)\deln\mu && \text{on} \ \Sigma, \\
        \label{EQ:SYSTEM:4}
        &\theta = - \Lapg\psi + G'(\psi) + \alpha\deln\phi && \text{on} \ \Sigma, \\
        \label{EQ:SYSTEM:5}
        &\begin{cases} K\deln\phi = \alpha\psi - \phi &\text{if} \ K\in [0,\infty), \\
        \deln\phi = 0 &\text{if} \ K = \infty
        \end{cases} && \text{on} \ \Sigma, \\
        \label{EQ:SYSTEM:6}
        &\begin{cases} 
        L m_\Om(\phi)\deln\mu = \beta\theta - \mu &\text{if} \  L\in[0,\infty), \\
        m_\Om(\phi)\deln\mu = 0 &\text{if} \ L=\infty
        \end{cases} &&\text{on} \ \Sigma, \\
        \label{EQ:SYSTEM:7}
        &\phi\vert_{t=0} = \phi_0 &&\text{in} \ \Om, \\
        \label{EQ:SYSTEM:8}
        &\psi\vert_{t=0} = \psi_0 &&\text{on} \ \Ga.
    \end{align}
\end{subequations}
This model was first proposed in \cite{Knopf2024} and further analyzed in \cite{Knopf2024a}. It offers a unified framework for various Cahn–Hilliard-type dynamic boundary conditions previously introduced in the literature.

In System~\eqref{EQ:SYSTEM}, $\Om\subset\R^d$ with $d\in\{2,3\}$ is a bounded domain with boundary $\Ga\coloneqq\partial\Om$ and outward unit normal vector field $\n$. Moreover, $T>0$ denotes a prescribed final time and we write $Q\coloneqq \Om\times(0,T)$ and $\Sigma\coloneqq\Ga\times(0,T)$. 
The symbols $\Gradg$, $\Divg$ and $\Lapg$ denote the tangential surface gradient, the surface divergence, and the Laplace-Beltrami operator on $\Ga$, respectively. 

System \eqref{EQ:SYSTEM} is referred to as a \textit{convective} Cahn--Hilliard model as the dynamics are influenced by prescribed velocity fields $\boldsymbol{v}:Q\rightarrow\R^d$ and $\boldsymbol{w}:\Sigma\rightarrow\R^d$.  
As in \cite{Knopf2024}, we assume that $\Div(\boldsymbol{v}) = 0$ in $Q$ and that
$\Divg(\boldsymbol{w}) = 0$, $\boldsymbol{v} \cdot \n = 0$ and $\boldsymbol{w} \cdot \n = 0$ on $\Sigma$.

The functions $\phi:Q\rightarrow\R$ and $\mu:Q\rightarrow\R$ denote the phase-field and the chemical potential in the bulk, whereas $\psi:\Sigma\rightarrow\R$ and $\theta:\Sigma\rightarrow\R$ represent the phase-field and the chemical potential on the boundary, respectively. The functions $m_\Om(\phi)$ and $m_\Ga(\psi)$ represent the mobilities of the mixtures in the bulk and on the boundary, respectively. They model the spatial locations and intensity to which the diffusion processes take place.
The phase-fields $\phi$ and $\psi$ are coupled by the boundary condition \eqref{EQ:SYSTEM:5}, and the chemical potentials $\mu$ and $\theta$ are coupled by the boundary condition \eqref{EQ:SYSTEM:6}. Here, the parameters $K,L\in [0,\infty]$ are used to distinguish different types of these coupling conditions.

The functions $F^\prime$ and $G^\prime$ are the derivatives of double-well potentials $F$ and $G$, respectively.
A physically motivated example of such a double-well potential, especially in applications related to material science, is the \textit{Flory--Huggins potential}, which is also referred to as the \textit{logarithmic potential}. It is defined as
\begin{align}\label{DEF:W:LOG}
	W_{\textup{log}}(s) \coloneqq \frac{\Theta}{2}\big[(1+s)\ln(1+s) + (1-s)\ln(1-s)\big] -\frac{\Theta_c}{2}s^2, \quad s\in(-1,1).
\end{align}
Here, we assume $0 < \Theta < \Theta_c$, where $\Theta$ denotes the temperature of the mixture and $\Theta_c$ represents the critical temperature below which phase separation processes occur. We point out that $W_{\textup{log}}$ can be continuously extended onto the interval $[-1,1]$. However, since $W_{\textup{log}}'(s) \to \pm\infty$ as $s\to \pm 1$, $W_{\textup{log}}$ is classified as a singular potential. In this paper, we consider a general class of singular potentials (see Subsection~\ref{SUBSEC:ASS}) such that, for example, the choice $F=G=W_{\textup{log}}$ is admissible.

The free energy functional associated with System~\eqref{EQ:SYSTEM} reads as
\begin{align}
	\label{INTRO:ENERGY}
	\begin{split}
		E_K(\phi,\psi) &= \intO \frac12\abs{\Grad\phi}^2 + F(\phi) \dx + \intG \frac12\abs{\Gradg\psi}^2 + G(\psi) \dG \\
		&\quad + \sigma(K)\intG\frac{1}{2}\abs{\alpha\psi - \phi}^2\dG,
	\end{split}
\end{align}
where, to account for the different cases corresponding to the choice of $K$, the function $\sigma$ is defined 
\begin{align*}
	\sigma(K) \coloneqq
	\begin{cases}
		K^{-1}, &\text{if } K\in (0,\infty), \\
		0, &\text{if } K\in\{0,\infty\}.
	\end{cases}
\end{align*}
We observe that sufficiently regular solutions of System~\eqref{EQ:SYSTEM} satisfy the \textit{mass conservation law}
\begin{align}
	\label{INTRO:MASS}
	\begin{dcases}
		\beta\intO \phi(t)\dx + \intG \psi(t)\dG = \beta\intO \phi_0 \dx + \intG \psi_0\dG, &\textnormal{if } L\in[0,\infty), \\
		\intO\phi(t)\dx = \intO\phi_0\dx \quad\textnormal{and}\quad \intG\psi(t)\dG = \intG\psi_0\dG, &\textnormal{if } L = \infty
	\end{dcases}
\end{align}
for all $t\in[0,T]$ and the \textit{energy identity}
\begin{align}
	\label{INTRO:ENERGY:ID}
	\begin{split}
		\ddt E_K(\phi,\psi)  &= \intO\phi\boldsymbol{v}\cdot\Grad\mu\dx + \intG \psi\boldsymbol{w}\cdot\Gradg\theta\dG 
		- \intO m_\Om(\phi)\abs{\Grad\mu}^2\dx\\
		&\qquad
		- \intG m_\Ga(\psi)\abs{\Gradg\theta}^2\dG
		- \sigma(L) \intG (\beta\theta-\mu)^2\dG 
	\end{split}
\end{align}
on $(0,T)$. In the non-convective case $(\boldsymbol{v},\boldsymbol{w})\equiv (\boldsymbol{0},\boldsymbol{0})$, the right-hand side of \eqref{INTRO:ENERGY:ID} is non-positive as the first two terms vanish. 

\textbf{State of the art and related literature.}
The \textit{non-convective} case of System~\eqref{EQ:SYSTEM}, corresponding to $\boldsymbol{v}\equiv 0$ and $\boldsymbol{w}\equiv 0$, has been the subject of extensive investigation in the literature.
We refer, for example, to 
\cite{Garcke2020,Garcke2022,Colli2020,Colli2022,Colli2022a,Fukao2021,Miranville2020,Lv2024,Lv2024a,Lv2024b} 
for the mathematical analysis and to
\cite{Metzger2021,Metzger2023,Harder2022,Meng2023,Bao2021,Bao2021a,Bullerjahn2024,Bullerjahn2025}
for the numerical analysis and simulations. A related nonlocal bulk-surface Cahn--Hilliard model was studied in \cite{Knopf2021b,Lv2025}.

The general \textit{convective} version of \eqref{EQ:SYSTEM} (i.e., with non-trivial velocity fields) was recently analyzed for in \cite{Knopf2024} for regular potentials, and in \cite{Knopf2024a} for singular potentials. In fact, System~\eqref{EQ:SYSTEM} appears as a subsystem in models for two-phase flows with bulk-surface interaction. A Navier--Stokes--Cahn--Hilliard model with dynamic boundary conditions based on \eqref{EQ:SYSTEM} was derived in \cite{Giorgini2023} and analyzed in \cite{Giorgini2023,Gal2024}, and a related Cahn--Hilliard--Brinkman model with dynamic boundary conditions was investigated in \cite{Colli2024}.

For a general overview on Cahn--Hilliard type models with classical or dynamic boundary conditions, we refer to the review paper \cite{Wu2022} as well as the book \cite{Miranville-Book}.

\textbf{Goals and novelties of this paper.}
The well-posedness of weak and strong solutions of System~\eqref{EQ:SYSTEM} as well as further regularity properties were already established in \cite{Knopf2024a}. These results require the velocity fields to satisfy restrictive regularity assumptions, which might not be satisfied if the velocity fields are not simply prescribed, but are determined by additional coupled evolution equations.
This motivates us to establish the following main results. 

\begin{enumerate}[label = \textbf{(\Roman*)}, leftmargin=*]
    \item\label{MR1} 
    \textbf{Existence of a weak solution and further regularity.} 
    In Theorem~\ref{THEOREM:EOWS} we improve the weak well-posedness result in \cite[Theorem~3.4]{Knopf2024a}. In contrast to the previous result in \cite{Knopf2024a}, where the velocity fields are such that
    \begin{equation}
    \label{INTRO:VEL:WEAK:OLD}
        \boldsymbol{v}\in L^2(0,T;\mathbf{L}^{3}_\Div(\Om))
        \quad\text{and}\quad
        \boldsymbol{w}\in L^2(0,T;\mathbf{L}^{2+\omega}_\tau(\Ga))
        \quad\text{for some $\omega > 0$},
    \end{equation}
    while our new result requires the velocity fields to satisfy the weaker condition
    \begin{equation}
    \label{INTRO:VEL:WEAK}
        \boldsymbol{v}\in L^2(0,T;\mathbf{L}^{2}_\Div(\Om))
        \quad\text{and}\quad
        \boldsymbol{w}\in L^2(0,T;\mathbf{L}^{2}_\tau(\Ga)).
    \end{equation}
    The definition of the spaces $\mathbf{L}^p_\Div(\Om)$ and $\mathbf{L}^{p}_\tau(\Ga)$ 
    for any $p\in [2,\infty)$ can be found in Subsection~\ref{SUBSECT:NOTFS}.

    Moreover, Theorem~\ref{THM:TIMEREG} provides an additional regularity result for weak solutions, which improves \cite[Theorem~3.4]{Knopf2024a} by providing a higher time regularity for the phase-fields $\phi$ and $\psi$. 
    At first glance, this seems to be only a very minor improvement. However, especially in Navier--Stokes--Cahn--Hilliard models (e.g., the one from \cite{Giorgini2023}), where \eqref{EQ:SYSTEM} appears as a subsystem, these enhanced time regularities become crucial in the mathematical analysis as they help to estimate the nonlinear terms.

    \item\label{MR2}  
    \textbf{Continuous dependence on the initial data and the velocity fields.}
    In Theorem~\ref{THEOREM:UNIQUE:SING}, we further prove the uniqueness and the continuous dependence of the weak solution on the initial data and the velocity fields along with a concrete stability estimate. This result improves the result of \cite[Theorem~3.6]{Knopf2024a} since one of the velocity fields only needs to satisfy \eqref{INTRO:VEL:WEAK} instead of \eqref{INTRO:VEL:WEAK:OLD}, and our new stability estimate merely uses the differences of the velocity fields in the $L^2(0,T;\mathbf{L}^2(\Omega)\times \mathbf{L}^2(\Gamma))$-norm instead of the $L^2(0,T;\mathbf{L}^3(\Omega)\times \mathbf{L}^{2+\omega}(\Gamma))$-norm. 

    \item\label{MR3}  
    \textbf{Existence of a strong solution and further regularity.}
    In \cite[Theorem~3.7]{Knopf2024a}, the existence of a strong solution to System~\eqref{EQ:SYSTEM} was established for velocity fields, which satisfy
    \begin{align}
        \label{INTRO:VEL:STRONG:OLD}
        \left\{
        \begin{aligned}
        \boldsymbol{v} &\in H^1(0,T;\mathbf{L}^{6/5}(\Om))\cap L^2(0,T;\mathbf{L}_\Div^3(\Om))\cap L^\infty(0,T;\mathbf{L}^2(\Om)),
        \\
        \boldsymbol{w}&\in H^1(0,T;\mathbf{L}^{1+\omega}(\Ga))\cap L^2(0,T;\mathbf{L}_\Div^{2}(\Ga))\cap L^\infty(0,T;\mathbf{L}^2(\Ga))
        \end{aligned}
        \right.
    \end{align}
    for some $\omega > 0$. For $K\in[0,\infty)$ and $L\in(0,\infty]$, we show in Theorem~\ref{Theorem:HighReg} that the same result holds true if the velocity fields have \textit{Leray type regularity}, that is
    \begin{align}
        \label{INTRO:VEL:STRONG}
        (\boldsymbol{v}, \boldsymbol{w})\in L^\infty(0,T;\mathcal{L}_{\Div}^2)\cap L^2(0,T;\mathcal{H}^1).
    \end{align}
    In the two-dimensional case (i.e., $d=2$), the result remains correct for $L=0$ and $K\in[0,\infty)$, see Remark~\ref{Rem:HighReg}\ref{Rem:HighReg:a}.    
    The crucial advantage compared to \cite[Theorem~3.7]{Knopf2024a} is that \eqref{INTRO:VEL:STRONG} is a more natural assumption on the velocity fields than \eqref{INTRO:VEL:STRONG:OLD}. 
    In particular, if the velocity fields are not just prescribed but determined by additional evolution equations (e.g., Navier--Stokes equations), Leray type regularity is exactly what the classical energy balance provides.
    In this sense, the new regularity result provided by Theorem~\ref{Theorem:HighReg} is more useful in the mathematical analysis for coupled systems, e.g., for the convergence to stationary states as in \cite{AGG2024global}.
\end{enumerate}

To prove our main result \ref{MR3} and the additional regularities in \ref{MR1}, we first establish useful auxiliary results. These results are not only relevant to the present paper but may also be of interest for the further analysis of Cahn--Hilliard-type or Allen--Cahn-type systems with dynamic boundary conditions in future research.
\begin{enumerate}[label = \textbf{(\roman*)}, leftmargin=*]
    \item \label{AR1}
    In Section~\ref{SECT:SUBDIFF}, we characterize the subdifferential of the singular functional
    \begin{align*}
    \widetilde{E}_K(u,v) = 
        \intO \frac12\abs{\Grad u}^2 + F_1(u)\dx + \intG \frac12\abs{\Gradg v}^2 + G_1(v)\dG 
        + \displaystyle\sigma(K) \intG \frac12 (\alpha v - u)^2 \dG,
    \end{align*}
    where $F_1$ and $G_1$ can be chosen as the convex, singular parts of the potentials $F$ and $G$ introduced above. In particular, $\widetilde{E}_K$ coincides with $E_K$ introduced in \eqref{INTRO:ENERGY} except for the remaining regular parts $F_2 = F - F_1$ and $G_2 = G - G_1$ of the potentials $F$ and $G$. In Proposition~\ref{App:Proposition:Subdiff}, we show that the subdifferential $\partial\widetilde{E}_K$ of $\widetilde{E}_K$ is a maximal monotone operator and we provide an explicit representation formula for $\partial\widetilde{E}_K$. 

    \item \label{AR2}
    In Section~\ref{SECT:BSE}, we investigate the following bulk-surface elliptic system, which involves the singular nonlinearities $F_1$ and $G_1$:
    \begin{subequations}
    \label{INTRO:BSE}
        \begin{align}
            \label{INTRO:BSE:1}
            -\Lap u + F_1^\prime(u) &= f &\text{in $\Omega$},\\
            \label{INTRO:BSE:2}
            -\Lapg v + G_1^\prime(v) + \alpha \deln u &= g &\text{on $\Gamma$},\\
            \label{INTRO:BSE:3}
            K\deln u &= \alpha v - u 
            &\text{on $\Gamma$}.
        \end{align}        
    \end{subequations}
    The results of Proposition~\ref{App:Proposition:Subdiff} outlined in \ref{AR1} can be used to prove the existence of a unique strong solution to \eqref{INTRO:BSE}. In the remainder of Section~\ref{SECT:BSE} we further prove several regularity results with corresponding estimates for strong solutions of \eqref{INTRO:BSE}. These results are used to establish the higher regularities claimed in our main results \ref{MR1} and \ref{MR3}.
\end{enumerate}

\noindent
\section{Functional framework, preliminaries and assumptions}
\subsection{Notation and Function Spaces.} \label{SUBSECT:NOTFS}
For any Banach space $X$, we denote its norm by $\norm{\cdot}_X$, its dual space by $X^\prime$, and the associated duality pairing of elements $\Lambda \in X^\prime$ and $\zeta\in X$ by $\ang{\Lambda}{\zeta}_X$. The space $L^p(I;X)$, $1\leq p \leq +\infty$, denotes the set of all strongly measurable $p$-integrable functions mapping from any interval $I\subset\R$ into $X$, or, if $p = +\infty$, essentially bounded functions. Moreover, the space $W^{1,p}(I;X)$ consists of all functions $f\in L^p(I;X)$ such that $\delt f\in L^p(I;X)$, where $\delt f$ denotes the vector-valued distributional derivative of $f$.

Throughout this paper, we prescribe $\Omega\subset \mathbb{R}^d$ with $d=2,3$ to be a bounded domain of class $C^2$, whose boundary is denoted by $\Gamma$. 

For any real numbers $s\geq 0$ and $p\in[1,\infty]$, we introduce the spaces
\begin{align*}
    \mathcal{L}^p \coloneqq L^p(\Om)\times L^p(\Ga), \quad\text{and}\quad \mathcal{W}^{s,p} \coloneqq W^{s,p}(\Om)\times W^{s,p}(\Ga),
\end{align*}
provided that the boundary $\Ga$ is sufficiently regular. We abbreviate $\mathcal{H}^s \coloneqq \mathcal{W}^{s,2}$ and identify $\mathcal{L}^2$ with $\mathcal{H}^0$. Note that $\mathcal{H}^s$ is a Hilbert space with respect to the inner product
\begin{align*}
    \bigscp{\scp{\phi}{\psi}}{\scp{\zeta}{\xi}}_{\mathcal{H}^s} \coloneqq \scp{\phi}{\zeta}_{H^s(\Om)} + \scp{\psi}{\xi}_{H^s(\Ga)} \qquad\text{for all~} \scp{\phi}{\psi},\scp{\zeta}{\xi}\in\mathcal{H}^s
\end{align*}
and its induced norm $\norm{\cdot}_{\mathcal{H}^s} \coloneqq \scp{\cdot}{\cdot}_{\mathcal{H}^s}^{\frac12}$. We recall that the duality pairing can be expressed as
\begin{align*}
    \ang{\scp{\phi}{\psi}}{\scp{\zeta}{\xi}}_{\mathcal{H}^s} \coloneqq \scp{\phi}{\zeta}_{L^2(\Om)} + \scp{\psi}{\xi}_{L^2(\Ga)}
\end{align*}
for all $(\zeta,\xi)\in\mathcal{H}^s$ if $(\phi,\psi)\in\mathcal{L}^2$.

Moreover, for any $2\leq p <\infty$, we introduce the spaces
\begin{align*}
    \mathbf{L}^p_\Div(\Om) &\coloneqq 
    \{\boldsymbol{v}\in\mathbf{L}^p(\Om) : 
    \Div\;\boldsymbol{v} = 0 \ \text{in~} \Om, \ 
    \boldsymbol{v}\cdot\mathbf{n} = 0 \ \text{on~} \Ga\},
    \\
    \mathbf{L}^p_\tau(\Ga)&\coloneqq
    \{\boldsymbol{w}\in\mathbf{L}^p(\Ga) : 
    \boldsymbol{w}\cdot\n = 0 \ \text{on~}\Ga\},
    \\
    \mathbf{L}^p_\Div(\Ga)&\coloneqq
    \{\boldsymbol{w}\in\mathbf{L}^p_\tau(\Ga) : 
    \Divg\;\boldsymbol{w} = 0 \ \text{on~}\Ga\},
    \\
    \mathcal{L}^p_\Div &\coloneqq 
    \mathbf{L}^p_\Div(\Om) \times \mathbf{L}^p_\Div(\Ga).
\end{align*}
Note that in the definitions of $\mathbf{L}^p_\Div(\Om)$ and $\mathbf{L}^p_\Div(\Ga)$, the expressions $\Div\;\boldsymbol{v}$ in $\Om$ and $\Divg\;\boldsymbol{w}$ on $\Ga$ are to be understood in the sense of distributions. For any $\boldsymbol{v}\in \mathbf{L}^p(\Om)$ ($p\ge 2$) with $\Div\;\boldsymbol{v} = 0$ in $\Om$, it holds that $\boldsymbol{v}\cdot\mathbf{n} \in H^{-1/2}(\Ga)$, and therefore, the relation $\boldsymbol{v}\cdot\mathbf{n} = 0$ on $\Ga$ is well-defined.

For $L\in[0,\infty]$ and $\beta\in\R$, we introduce the linear subspaces
\begin{align*}
    \mathcal{H}_{L,\beta}^1 \coloneqq
    \begin{cases}
        \mathcal{H}^1, &\text{if } L \in (0,\infty] , \\
     \displaystyle   \{(\phi,\psi)\in\mathcal{H}^1 : \phi = \beta\psi \text{ a.e.~on } \Ga\}, &\text{if } L=0.
    \end{cases}
\end{align*}
The spaces $\mathcal{H}_{L,\beta}^1$ are Hilbert spaces endowed with the inner product $\scp{\cdot}{\cdot}_{\mathcal{H}_{L,\beta}^1} \coloneqq \scp{\cdot}{\cdot}_{\mathcal{H}^1}$ and its induced norm.  Moreover, we define the product
\begin{align*}
    \ang{\scp{\phi}{\psi}}{\scp{\zeta}{\xi}}_{\mathcal{H}_{L,\beta}^1} \coloneqq \scp{\phi}{\zeta}_{L^2(\Om)} + \scp{\psi}{\xi}_{L^2(\Ga)}
\end{align*}
for all $\scp{\phi}{\psi}, \scp{\zeta}{\xi}\in\mathcal{L}^2$. By means of the Riesz representation theorem, this product can be extended to a duality pairing on $(\mathcal{H}_{L,\beta}^1)^\prime\times\mathcal{H}_{L,\beta}^1$, which will also be denoted as $\ang{\cdot}{\cdot}_{\mathcal{H}_{L,\beta}^1}$.

For $(\phi,\psi)\in(\mathcal{H}^1_{L,\beta})^\prime$, we define the generalized bulk-surface mean
\begin{align*}
    \mean{\phi}{\psi} \coloneqq \frac{\ang{\scp{\phi}{\psi}}{\scp{\beta}{1}}_{\mathcal{H}^1_{L,\beta}}}{\beta^2\abs{\Om} + \abs{\Ga}},
\end{align*}
which reduces to
\begin{align*}
    \mean{\phi}{\psi} = \frac{\beta\abs{\Om}\meano{\phi} + \abs{\Ga}\meang{\psi}}{\beta^2\abs{\Om} + \abs{\Ga}}
\end{align*}
if $\phi\in L^2(\Om)$ and $\psi\in L^2(\Ga)$, where
\begin{align*}
\meano{\phi}=\frac{1}{\abs{\Omega}} \int_{\Omega} \phi \, \dx,
\quad
\meang{\psi}=\frac{1}{\abs{\Gamma}} \int_{\Gamma} \psi \, \dG.
\end{align*}
We then define the closed linear subspaces
\begin{align*}
    \mathcal{V}_{L,\beta}^1 &\coloneqq \begin{cases} 
    \{\scp{\phi}{\psi}\in\mathcal{H}^1_{L,\beta} : \mean{\phi}{\psi} = 0 \}, &\text{if~} L\in[0,\infty), \\
    \{\scp{\phi}{\psi}\in\mathcal{H}^1: \meano{\phi} = \meang{\psi} = 0 \}, &\text{if~}L=\infty.
    \end{cases} 
\end{align*}
Note that these subspaces are Hilbert spaces with respect to the inner product $\scp{\cdot}{\cdot}_{\mathcal{H}^1}$.

Next, we set
\begin{align*}
    \sigma(L) \coloneqq
    \begin{cases}
        L^{-1}, &\text{if } L\in(0,\infty), \\
            0, &\text{if } L\in\{0,\infty\},
    \end{cases}
\end{align*}
and we introduce a bilinear form on $\mathcal{H}^1\times\mathcal{H}^1$ by defining
\begin{align*}
    \bigscp{\scp{\phi}{\psi}}{\scp{\zeta}{\xi}}_{L,\beta} \coloneqq &\intO\Grad\phi\cdot\Grad\zeta \dx + \intG\Gradg\psi\cdot\Gradg\xi \dG + \sigma(L)\intG (\beta\psi-\phi)(\beta\xi-\zeta)\dG
\end{align*}
for all $ \scp{\phi}{\psi}, \scp{\zeta}{\xi}\in\mathcal{H}^1$. Moreover, we set 
\begin{align*}
    \norm{\scp{\phi}{\psi}}_{L,\beta} \coloneqq \bigscp{\scp{\phi}{\psi}}{\scp{\phi}{\psi}}_{L,\beta}^{\frac12}
\end{align*}
for all $\scp{\phi}{\psi}\in\mathcal{H}^1$. The bilinear form $\scp{\cdot}{\cdot}_{L,\beta}$ defines an inner product on $\mathcal{V}^1_{L,\beta}$, and $\norm{\cdot}_{L,\beta}$ defines a norm on $\mathcal{V}^1_{L,\beta}$, that is equivalent to the norm $\norm{\cdot}_{\mathcal{H}^1}$ (see \cite[Corollary A.2]{Knopf2021}). Hence, the space $\mathcal{V}^1_{L,\beta}$ endowed with $\scp{\cdot}{\cdot}_{L,\beta}$ is a Hilbert space.

Next, we define the spaces
\begin{align*}
    \mathcal{V}_{L,\beta}^{-1} \coloneqq \begin{cases} 
    \{\scp{\phi}{\psi}\in(\mathcal{H}^1_{L,\beta})^\prime : \beta\abs{\Om}\meano{\phi} + \abs{\Ga}\meang{\psi} = 0 \}, &\text{if~} L\in[0,\infty), \\
    \{\scp{\phi}{\psi}\in(\mathcal{H}^1)^\prime: \meano{\phi} = \meang{\psi} = 0 \}, &\text{if~}L=\infty.
    \end{cases}
\end{align*}
Using the Lax--Milgram theorem, one can show that for any $(\phi,\psi)\in\mathcal{V}^{-1}_{L,\beta}$, there exists a unique weak solution $\mathcal{S}_{L,\beta}(\phi,\psi) = \big(\mathcal{S}_{L,\beta}^\Om(\phi,\psi),\mathcal{S}_{L,\beta}^\Ga(\phi,\psi)\big)\in\mathcal{V}^1_{L,\beta}$ to the following elliptic problem with bulk-surface coupling
\begin{alignat*}{2}
    -\Lap\mathcal{S}_{L,\beta}^\Om(\phi,\psi) &= -\phi &&\qquad\text{in~}\Om, \\
    -\Lapg\mathcal{S}_{L,\beta}^\Ga(\phi,\psi) + \beta\deln\mathcal{S}_{L,\beta}^\Om(\phi,\psi) &= -\psi &&\qquad\text{on~}\Ga, \\
    L\deln\mathcal{S}_{L,\beta}^\Om(\phi,\psi) &= \beta\mathcal{S}_{L,\beta}^\Ga(\phi,\psi) - \mathcal{S}_{L,\beta}^\Om(f,g) &&\qquad\text{on~}\Ga,
\end{alignat*}
in the sense that is satisfies the weak formulation
\begin{align*}
    \big(\mathcal{S}_{L,\beta}(\phi,\psi)),(\zeta,\xi)\big)_{L,\beta} = -\bigang{(\phi,\psi)}{(\zeta,\xi)}_{\mathcal{H}^1_{L,\beta}}
\end{align*}
for all test functions $(\zeta,\xi)\in\mathcal{H}^1_{L,\beta}$. Consequently, there exists a constant $C \geq 0$, depending only on $\Om, L$ and $\beta$ such that
\begin{align*}
    \norm{\mathcal{S}_{L,\beta}(\phi,\psi)}_{L,\beta}\leq C\norm{(\phi,\psi)}_{(\mathcal{H}^1_{L,\beta})^\prime}
\end{align*}
for all $(\phi,\psi)\in\mathcal{V}^{-1}_{L,\beta}$. We can thus define a solution operator
\begin{align*}
    \mathcal{S}_{L,\beta}:\mathcal{V}^{-1}_{L,\beta}\rightarrow\mathcal{V}^1_{L,\beta}, \quad (\phi,\psi)\mapsto \mathcal{S}_{L,\beta}(\phi,\psi) = \big(\mathcal{S}_{L,\beta}^\Om(\phi,\psi),\mathcal{S}_{L,\beta}^\Ga(\phi,\psi)\big)
\end{align*}
as well as an inner product and its induced norm on $\mathcal{V}^{-1}_{L,\beta}$ by
\begin{align*}
    \big((\phi,\psi),(\zeta,\xi)\big)_{L,\beta,\ast} &\coloneqq \big(\mathcal{S}_{L,\beta}(\phi,\psi),\mathcal{S}_{L,\beta}(\zeta,\xi)\big)_{L,\beta}, \\
    \norm{(\phi,\psi)}_{L,\beta,\ast} &\coloneqq \big((\phi,\psi),(\phi,\psi)\big)_{L,\beta,\ast}^{1/2},
\end{align*}
for $(\phi,\psi), (\zeta,\xi)\in\mathcal{V}^{-1}_{L,\beta}$. This is a norm on $\mathcal{V}^{-1}_{L,\beta}$, which is equivalent to $\norm{\cdot}_{(\mathcal{H}^1_{L,\beta})^\prime}$, see, e.g., \cite[Theorem~3.3 and Corollary~3.5]{Knopf2021} for a proof in the cases $L\in(0,\infty)$. In other cases, the proof can be carried out analogously.
\subsection{Important tools}

Throughout this paper, we will frequently use the bulk-surface Poincar\'{e} inequality, which has been established in \cite[Lemma A.1]{Knopf2021}:
\begin{lemma}
    \label{Prelim:Poincare}
    Let $K\in[0,\infty)$ and $\alpha,\beta\in\R$ such that $\alpha\beta\abs{\Om} + \abs{\Ga}\neq 0$. Then, there exists a constant $C_P > 0$, depending only on $K,\alpha,\beta$ and $\Om$ such that
    \begin{align*}
        \norm{(\zeta,\xi)}_{\mathcal{L}^2} \leq C_P \norm{(\zeta,\xi)}_{K,\alpha}
    \end{align*}
    for all pairs $(\zeta,\xi)\in\mathcal{H}^1_{K,\alpha}$ satisfying $\mean{\zeta}{\xi} = 0$.
\end{lemma}

Furthermore, we recall the following trace interpolation inequality (see, e.g., \cite{Necas2012}):
\begin{lemma}\label{Prelim:Lemma:Interpol}
    Let $\Om$ be a Lipschitz domain in $\R^d$, $d=2,3$, with compact boundary. Then, there exists a positive constant $C$ such that
    \begin{align*}
        \norm{u}_{L^2(\partial\Om)} \leq C\norm{u}_{L^2(\Om)}^{\frac12}
        \norm{u}_{H^1(\Om)}^{\frac12}, \quad\text{for all~}u\in H^1(\Om).
    \end{align*}
\end{lemma}

\medskip

\subsection{Main Assumptions.}
\label{SUBSEC:ASS}
The constants $\alpha,\beta\in\R$ appearing in System~\eqref{EQ:SYSTEM} are supposed to satisfy $\alpha\beta\abs{\Om} + \abs{\Ga} \neq 0$. We further assume that $-1 \leq \alpha \leq 1$.
We require that the mobility functions $m_\Om,m_\Ga\in C(\R)$ satisfy
\begin{align}
    \label{Assumption:Mobility}
    0 < m_\Om^\ast \leq m_\Om(s) \leq M_\Om^\ast \quad\text{and~}\quad 0 < m_\Ga^\ast \leq m_\Ga(s) \leq M_\Ga^\ast \quad\text{for all~}s\in\R
\end{align}
for some positive constants $m_\Om^\ast, M_\Om^\ast, m_\Ga^\ast, M_\Ga^\ast > 0$. The singular potentials $F$ and $G$ have the form
\begin{align*}
    F(s) = F_1(s) + F_2(s) \quad\text{and~}\quad G(s) = G_1(s) + G_2(s), \quad\text{for all~} s\in[-1,1],
\end{align*}
where $F_1,G_1\in C([-1,1])\cap C^2(-1,1)$ with
\begin{align}\label{Assumption:Pot:Sing}
    \lim_{s\searrow -1} F_1^\prime(s) = \lim_{s\searrow -1} G_1^\prime(s) = -\infty \quad\text{and~}\quad \lim_{s\nearrow 1} F_1^\prime(s) = \lim_{s\nearrow 1} G_1^\prime(s) = +\infty.
\end{align}
There further exist constants $\Theta_\Om, \Theta_\Ga > 0$ such that
\begin{align}\label{Assumption:Pot:Convexity}
    F_1^{\prime\prime}(s) \geq \Theta_\Om \quad\text{and~}\quad G_1^{\prime\prime}(s) \geq \Theta_\Ga \qquad\text{for all~}s\in(-1,1).
\end{align}
We extend $F_1$ and $G_1$ by defining $F_1(s) = G_1(s) = +\infty$ for $s\not\in[-1,1]$ and, without loss of generality, we assume  that $F_1(0) = G_1(0) = 0$ and $F_1^\prime(0) = G_1^\prime(0) = 0$. 
Lastly, we assume that the singular part of the boundary potential dominates the singular part of the bulk potential in the following sense: there exist $\kappa_1,\kappa_2 > 0$ such that
\begin{align}
    \label{Domination:Prime}
    \abs{F_1^\prime(\alpha s)} \leq \kappa_1\abs{G_1^\prime(s)} + \kappa_2 \qquad\text{for all~}s\in(-1,1),
\end{align}
where $\alpha$ is given in \eqref{EQ:SYSTEM}.
Finally, $F_2, G_2\in C^1(\R)$, and their derivatives are globally Lipschitz continuous.

\section{Main results.} 

\subsection{Weak solutions}

We start by introducing the notion of weak solutions.

\begin{definition}\label{DEF:SING:WS}
    Let $K,L\in[0,\infty]$, $(\boldsymbol{v},\boldsymbol{w})\in L^2(0,T;\mathbf{L}^2_\Div(\Om)\times\mathbf{L}^2_\tau(\Ga))$ be given velocity fields, and let $\scp{\phi_0}{\psi_0}\in\mathcal{H}^1_{K,\alpha}$ be an initial datum satisfying
    \begin{subequations}\label{cond:init}
    \begin{align}\label{cond:init:int}
        \norm{\phi_0}_{L^\infty(\Om)} \leq 1, \qquad \norm{\psi_0}_{L^\infty(\Ga)} \leq 1.
    \end{align}
    In addition, we assume that
    \begin{align}\label{cond:init:mean:L}
       &\beta\,\mean{\phi_0}{\psi_0}\in (-1,1), \quad \mean{\phi_0}{\psi_0}\in(-1,1), \quad \text{if~ }L\in[0,\infty),
\end{align}
and 
\begin{align}
 \label{cond:init:mean:inf}
      &  \meano{\phi_0}\in(-1,1), \quad \meang{\psi_0}\in(-1,1), \quad \text{if~ }L=\infty.
    \end{align}
    \end{subequations}
    The set $(\phi,\psi,\mu,\theta)$ is called a weak solution of System~\eqref{EQ:SYSTEM} on $[0,T]$ if the following properties hold:
    \begin{enumerate}[label=\textnormal{(\roman*)}, ref=\thetheorem(\roman*), topsep=0ex, leftmargin=*, itemsep=1.5ex]
        \item The functions $\phi, \psi, \mu$ and $\theta$ satisfy
        \begin{subequations}
        \label{REG:SING}
            \begin{align}
                &\scp{\phi}{\psi} \in C([0,T];\mathcal{L}^2)\cap H^1(0,T;(\mathcal{H}_{L,\beta}^1)^\prime)\cap 
                L^\infty(0,T;\mathcal{H}_{K,\alpha}^1), \label{REGPP:SING}\\
                &\scp{\mu}{\theta}\in L^2(0,T;\mathcal{H}_{L,\beta}^1) \label{REGMT:SING}, \\
                &\scp{F^\prime(\phi)}{G^\prime(\psi)}\in L^2(0,T;\mathcal{L}^2) \label{REGLC:SING},
            \end{align}
        \end{subequations}
        and it holds 
        \begin{equation}
            \label{PROP:CONF}
            \abs{\phi} < 1 \quad\text{a.e.~in $Q$}
            \quad\text{and}\quad
            \abs{\psi} < 1 \quad\text{a.e.~on $\Sigma$}.
        \end{equation}
    
    \item \label{DEF:SING:WS:IC} The initial conditions are satisfied in the following sense:
    \begin{align*}
        \phi\vert_{t=0} = \phi_0 \quad\text{a.e.~in } \Omega, \quad\text{and} \quad\psi\vert_{t=0} = \psi_0 \quad\text{a.e.~on }\Gamma.
    \end{align*}
    \item \label{DEF:SING:WS:WF} The variational formulation
    \begin{subequations}\label{SING:WF}
        \begin{align}
        &\begin{aligned}
            &\ang{\scp{\delt\phi}{\delt\psi}}{\scp{\zeta}{\xi}}_{\mathcal{H}_{L,\beta}^1} - \intO \phi\boldsymbol{v}\cdot\Grad\zeta\dx - \intG \psi\boldsymbol{w}\cdot\Gradg\xi\dG 
            \\
            &= - \intO m_\Om(\phi)\Grad\mu\cdot\Grad\zeta\dx - \intG  m_\Ga(\psi)\Gradg\theta\cdot\Gradg\xi\dG \label{WF:PP:SING}\\
            &\quad - \sigma(L)\intG(\beta\theta-\mu)(\beta\xi - \zeta)\dG, 
        \end{aligned}
            \\
        &\begin{aligned}
            &\intO \mu\,\eta\dx + \intG\theta\,\vartheta\dG
            \\
            &=  \intO\Grad\phi\cdot\Grad\eta + F^\prime(\phi)\eta\dx
            + \intG \Gradg\psi\cdot\Gradg\vartheta + G^\prime(\psi)\vartheta\dG\label{WF:MT:SING}
            \\
            &\quad + \sigma(K)\intG(\alpha\psi-\phi)(\alpha\vartheta - \eta) \dG, 
        \end{aligned}
        \end{align}
    \end{subequations}
    holds a.e.~on $[0,T]$ for all $\scp{\zeta}{\xi}\in\mathcal{H}_{L,\beta}^1, \scp{\eta}{\vartheta}\in\mathcal{H}_{K,\alpha}^1$.
    \item \label{DEF:SING:WS:MCL} The functions $\phi$ and $\psi$ satisfy the mass conservation law
    \begin{align}\label{MCL:SING}
        \begin{dcases}
            \beta\intO \phi(t)\dx + \intG \psi(t)\dG = \beta\intO \phi_0 \dx + \intG \psi_0\dG, &\textnormal{if } L\in[0,\infty), \\
            \intO\phi(t)\dx = \intO\phi_0\dx \quad\textnormal{and}\quad \intG\psi(t)\dG = \intG\psi_0\dG, &\textnormal{if } L = \infty
        \end{dcases}
    \end{align}
    for all $t\in[0,T]$.
    \item \label{DEF:SING:WS:WEDL} The energy inequality
    \begin{align}\label{WEDL:SING}
        \begin{split}
            &E_K(\phi(t),\psi(t)) + \int_0^t\intO m_\Om(\phi)\abs{\Grad\mu}^2\dxs + \int_0^t\intG m_\Ga(\psi)\abs{\Gradg\theta}^2\dGs \\
            &\quad + \sigma(L) \int_0^t\intG (\beta\theta-\mu)^2\dGs \\
            &
            \leq E_K(\phi_0,\psi_0) + \int_0^t\intO\phi\boldsymbol{v}\cdot\Grad\mu\dxs +\int_0^t\intG \psi\boldsymbol{w}\cdot\Gradg\theta\dGs 
        \end{split}
    \end{align}
    holds for all $t\in[0,T]$.
    \end{enumerate}
\end{definition}

\medskip

For prescribed velocity fields such that
\begin{equation}
    \label{ASS:VEL:WEAK:OLD}
    \boldsymbol{v}\in L^2(0,T;\mathbf{L}^{3}_\Div(\Om))
    \quad\text{and}\quad
    \boldsymbol{w}\in L^2(0,T;\mathbf{L}^{2+\omega}_\tau(\Ga))
    \quad\text{for some $\omega > 0$}
\end{equation}
the existence of a global-in-time weak solution has been proven in \cite[Theorem 3.4]{Knopf2024a}. In our first contribution, we show that this result holds true under the weaker assumptions
\begin{equation}
    \label{ASS:VEL:WEAK}
    \boldsymbol{v}\in L^2(0,T;\mathbf{L}^{2}_\Div(\Om))
    \quad\text{and}\quad
    \boldsymbol{w}\in L^2(0,T;\mathbf{L}^{2}_\tau(\Ga)).
\end{equation}

\medskip

\begin{theorem}\label{THEOREM:EOWS}
    Let $K,L\in[0,\infty]$, let $\scp{\phi_0}{\psi_0}\in\mathcal{H}^1_{K,\alpha}$ be an arbitrary initial datum satisfying \eqref{cond:init}, and let $(\boldsymbol{v},\boldsymbol{w})\in L^2(0,T;L^2(0,T;\mathbf{L}^2_\Div(\Om)\times\mathbf{L}^2_\tau(\Ga)))$ be given velocity fields. 
    Then, System~\eqref{EQ:SYSTEM} admits at least one weak solution $(\phi,\psi,\mu,\theta)$ in the sense of Definition~\ref{DEF:SING:WS},
    which has the additional regularities
    \begin{align*}
        \scp{\phi}{\psi}\in L^2(0,T;\mathcal{W}^{2,6}) \quad\text{and~}\quad (F^\prime(\phi), G^\prime(\psi))\in L^2(0,T;\mathcal{L}^6),
    \end{align*}
    and satisfies the equations
    \begin{align}
        &\mu = -\Lap\phi + F^\prime(\phi) &&\text{a.e.~in } Q, \label{Eq:mu:strong}\\
        &\theta = -\Lapg\psi + G^\prime(\psi) + \alpha\deln\phi &&\text{a.e.~on } \Sigma, \label{Eq:theta:strong}\\
        & \begin{cases} 
            K\deln\phi = \alpha\psi - \phi &\text{if} \ K\in [0,\infty), \\
            \deln\phi = 0 &\text{if} \ K = \infty
        \end{cases} &&  \text{a.e.~on } \Sigma .\label{Eq:bd:strong}
    \end{align}
    Furthermore, if $L \in(0,\infty]$, any weak solutions satisfies the energy equality, i.e., \eqref{WEDL:SING} holds with an equality.
\end{theorem}

The proof of this theorem can be found in Section~\ref{SECT:EXWS}.

Via regularity theory for a bulk-surface elliptic system with singular nonlinearities (see Section~\ref{SECT:BSE}), we can further improve the time regularities of the phase-fields $\phi$ and $\psi$ for any weak solution.

\begin{theorem}
    \label{THM:TIMEREG}
    Let $(\phi,\psi,\mu,\theta)$ be a weak solution to System~\eqref{EQ:SYSTEM} in the sense of Definition~\ref{DEF:SING:WS}. Then, if $K\in (0,\infty)$, it holds $(\phi, \psi) \in L^4(0,T;\mathcal{H}^2)$, while in the case $K = 0$, $(\phi, \psi)\in L^3(0,T;\mathcal{H}^2)$.
\end{theorem}

The proof of this theorem is presented in Section~\ref{SECT:EXWS}.

\begin{remark}
    If $K = \infty$ and $(\phi,\psi,\mu,\theta)$ is a weak solution to \eqref{EQ:SYSTEM} in the sense of Definition~\ref{DEF:SING:WS}, it also holds that $(\phi,\psi)\in L^4(0,T;\mathcal{H}^2)$. In this case, we have $\deln\phi = 0$ almost everywhere on $\Sigma$, and the equations \eqref{EQ:SYSTEM:2} and \eqref{EQ:SYSTEM:3} are uncoupled. Consequently, the argumentation is even easier and we simply refer to \cite[Lemma~A.3]{Conti2020}.
\end{remark}

Under the regularity assumption \eqref{ASS:VEL:WEAK:OLD} on the velocity fields and provided that the mobility functions are constant, it was further shown in \cite[Theorem 3.6]{Knopf2024a} that the weak solution is unique and depends continuously on the velocity fields in the $L^2(0,T;\mathbf{L}^3(\Om)\times\mathbf{L}^{2+\omega}(\Ga))$-norm and the initial data. The next theorem shows that this results can be strengthened by providing a continuous dependence, where the differences of the velocity fields are merely measured in the $L^2(0,T;\mathcal{L}^2)$-norm.

\begin{theorem}\label{THEOREM:UNIQUE:SING}
    Assume that the mobility functions $m_\Om$ and $m_\Ga$ are constant, let $K, L\in[0,\infty]$, let $(\phi_0^1,\psi_0^1)$, $(\phi_0^2,\psi_0^2)\in\mathcal{H}^1_{K,\alpha}$ be two pairs of initial data, which satisfy
    \begin{align}\label{initial-data-mean-value}
        \begin{cases}
            \mean{\phi_0^1}{\psi_0^1} = \mean{\phi_0^2}{\psi_0^2} \quad&\text{if~} L\in[0,\infty), \\
            \meano{\phi_0^1} = \meano{\phi_0^2} \quad\text{and}\quad \meang{\psi_0^1} = \meang{\psi_0^2} &\text{if~} L=\infty,
        \end{cases}
    \end{align}
    as well as \eqref{cond:init}, and let 
    \begin{align*}
        (\boldsymbol{v}_1,\boldsymbol{w}_1) 
        \in L^2(0,T;\mathcal{L}^2_\Div)
        \quad\text{and}\quad
        (\boldsymbol{v}_2,\boldsymbol{w}_2) 
        \in L^2(0,T;\mathbf{L}^3_\Div(\Om)\times\mathbf{L}^{2+\omega}_\tau(\Ga))
    \end{align*}
    for some $\omega > 0$ be prescribed velocity fields.
    Moreover, let $(\phi_1,\psi_1,\mu_1,\theta_1)$ and $(\phi_2,\psi_2,\mu_2,\theta_2)$ be weak solutions in the sense of Definition~\ref{DEF:SING:WS} corresponding to $(\phi_0^1, \psi_0^1, \boldsymbol{v}_1, \boldsymbol{w}_1)$ and $(\phi_0^2, \psi_0^2, \boldsymbol{v}_2, \boldsymbol{w}_2)$, respectively. 
    Then, the continuous dependence estimate
    \begin{align}\label{EST:continuous-dependence:sing}
        \begin{split}
            &\bignorm{\bigscp{\phi_1(t) - \phi_2(t)}{\psi_1(t) - \psi_2(t)}}_{L,\beta,\ast}^2 + \int_0^t \bignorm{\bigscp{\phi_1(\tau) - \phi_2(\tau)}{\psi_1(\tau) - \psi_2(\tau)}}_{K,\alpha}^2 \dtau
            \\[0.4em]
            &\leq C_0\bignorm{\bigscp{\phi_0^1 - \phi_0^2}{\psi_0^1 - \psi_0^2}}_{L,\beta,\ast}^2 \exp\left(C_0\int_0^t 1 + \norm{\scp{\boldsymbol{v}_2(\tau)}{\boldsymbol{w}_2(\tau)}}_{\mathbf{L}^3(\Om)\times\mathbf{L}^{2+\omega}(\Ga)}^2 \dtau\right) 
            \\
            &\quad+ C_0\int_0^t 
            \bignorm{\bigscp{\boldsymbol{v}_1(s) - \boldsymbol{v}_2(s)}{\boldsymbol{w}_1(s) - \boldsymbol{w}_2(s)}}_{\mathcal{L}^2}^2 \\
            &\qquad\times\exp\left(C_0\int_s^t 1 + \norm{\scp{\boldsymbol{v}_2(\tau)}{\boldsymbol{w}_2(\tau)}}_{\mathbf{L}^3(\Om)\times\mathbf{L}^{2+\omega}(\Ga)}^2\dtau\right)\ds
        \end{split}
    \end{align}
    holds for almost all $t\in[0,T]$. Here, the constants denoted by $C_0$ depend only on $\Om, K, L, \alpha, \beta$ and the Lipschitz constants of $F_2^\prime$ and $G_2^\prime$.
    In particular, this entails that the weak solution found in Theorem~\ref{THEOREM:EOWS} is unique.
\end{theorem}

The proof of this theorem can be found in Section~\ref{SECT:EXWS}.

\subsection{Strong solutions}

We now deal with the well-posedness of strong solutions.
The global existence was obtained in \cite{Knopf2024a}, where the velocity fields were required to satisfy the conditions
\begin{align}
    \label{ASS:VEL:STRONG:OLD}
    \left\{
    \begin{aligned}
    \boldsymbol{v} &\in H^1(0,T;\mathbf{L}^{6/5}(\Om))\cap L^2(0,T;\mathbf{L}_\Div^3(\Om))\cap L^\infty(0,T;\mathbf{L}^2(\Om)),
    \\
    \boldsymbol{w}&\in H^1(0,T;\mathbf{L}^{1+\omega}(\Ga))\cap L^2(0,T;\mathbf{L}_\Div^{2}(\Ga))\cap L^\infty(0,T;\mathbf{L}^2(\Ga))
    \end{aligned}
    \right.
\end{align}
for some $\omega > 0$. In the following theorem, we show that, if $K\in[0,\infty)$ and $L\in(0,\infty]$, well-posedness of strong solutions is obtained by assuming that the velocity fields are of Leray type, that is
\begin{align}
    \label{ASS:VEL:STRONG}
    (\boldsymbol{v}, \boldsymbol{w})\in L^\infty(0,T;\mathcal{L}_{\Div}^2)\cap L^2(0,T;\mathcal{H}^1).
\end{align}
This means that, in comparison with \eqref{ASS:VEL:STRONG:OLD}, we require more spatial regularity but less time regularity. The advantage of condition \eqref{ASS:VEL:STRONG} is that these Leray type regularities are exactly what can be expected from a weak solution of a Navier--Stokes equation in the bulk or on the surface, respectively. Therefore, compared to the result from \cite{Knopf2024a}, our new result seems more useful in the investigation of possible bulk-surface Navier--Stokes--Cahn--Hilliard models. Our result reads as follows.

\begin{theorem}
    \label{Theorem:HighReg}
    Let $K\in[0,\infty)$ and $L\in(0,\infty]$. We assume that $\Omega$ is of class $C^3$, the pair $(\phi_0,\psi_0)\in\mathcal{H}^1_{K,\alpha}$ satisfies \eqref{cond:init} and the mobility functions $m_\Om$ and $m_\Ga$ are constant. Moreover, let $(\boldsymbol{v}, \boldsymbol{w})\in L^\infty(0,T;\mathcal{L}_{\Div}^2)\cap L^2(0,T;\mathcal{H}^1)$ be given velocity fields. If $K = 0$, we further assume that $\boldsymbol{v}\vert_\Ga = \boldsymbol{w}$ a.e.~on $\Ga$.
    We further assume that the following compatibility condition holds:
    \begin{enumerate}[label=\textnormal{\bfseries(C)},topsep=0ex,leftmargin=*]
        \item \label{cond:MT:0} There exists $\scp{\mu_0}{\theta_0}\in\mathcal{H}^1_{L,\beta}$ such that for all $\scp{\eta}{\vartheta}\in\mathcal{H}^1_{K,\alpha}$ it holds
        \begin{align*}
        \begin{aligned}
            &\intO\mu_0\eta\dx + \intG\theta_0\vartheta\dG 
            \\
            &= \intO\Grad\phi_0\cdot\Grad\eta + F^\prime(\phi_0)\eta\dx + \intG\Gradg\psi_0\cdot\Gradg\vartheta + G^\prime(\psi_0)\vartheta\dG 
            \\
            &\quad + \sigma(K)\intG(\alpha\psi_0 - \phi_0)(\alpha\vartheta - \eta)\dG.
        \end{aligned}
        \end{align*}
    \end{enumerate}
    Then, the corresponding unique weak solution $(\phi,\psi,\mu,\theta)$ of System~\eqref{EQ:SYSTEM} obtained by Theorem~\ref{THEOREM:EOWS} has the following regularity:
    \begin{subequations}
        \begin{align*}
            (\delt\phi,\delt\psi) &\in L^\infty(0,T;(\mathcal{H}^1_{L,\beta})^\prime) \cap L^2(0,T;\mathcal{H}^1), \\
            (\phi,\psi)&\in L^\infty(0,T;\mathcal{W}^{2,6})\cap \Big(C(\overline{Q})\times C(\overline\Sigma)\Big), \\
            (\mu,\theta)&\in L^\infty(0,T;\mathcal{H}^1_{L,\beta})\cap L^2(0,T;\mathcal{H}^3), \\
            (F^\prime(\phi), G^\prime(\psi))&\in L^2(0,T;\mathcal{L}^\infty)\cap L^\infty(0,T;\mathcal{L}^6).
        \end{align*}
    \end{subequations}
    In particular, $(\phi,\psi,\mu,\theta)$ is the unique strong solution of System~\eqref{EQ:SYSTEM}, i.e., all equations of System~\eqref{EQ:SYSTEM} are fulfilled almost everywhere. Moreover, the following estimates hold
    \begin{align}\label{HighReg:Est:MT:Lb:sup}
        \begin{split}
            \esssup_{t\in(0,T)}\norm{(\mu(t),\theta(t))}_{L,\beta}^2 &\leq C_0\bigg(1 + \norm{(\mu_0,\theta_0)}_{L,\beta}^2 + \int_0^T\norm{(\boldsymbol{v},\boldsymbol{w})}_{\mathcal{H}^1}^2\ds\bigg) \\
            &\quad\times\exp\bigg(C_0\int_0^T\norm{(\boldsymbol{v},\boldsymbol{w})}_{\mathcal{H}^1}\ds\bigg)
        \end{split}
    \end{align}
    and
    \begin{align}\label{HighReg:Est:PP:Ka:L^2}
        \begin{split}
            &\int_0^T \norm{(\delt\phi,\delt\psi)}_{K,\alpha}^2 
            + \norm{(\Grad\mu,\Gradg\theta)}_{\mathcal{H}^2}^2 \ds
            \\
            &\quad
            \leq C_0\bigg(1 + \norm{(\mu_0,\theta_0)}_{L,\beta}^2 + \int_0^T\norm{(\boldsymbol{v},\boldsymbol{w})}_{\mathcal{H}^1}^2\ds\bigg) 
            \\
            &\qquad
            \times\bigg(1 + \bigg(\int_0^T\norm{(\boldsymbol{v},\boldsymbol{w})}_{\mathcal{H}^1}^2\ds\bigg)\exp\left(C_0\int_0^T\norm{(\boldsymbol{v},\boldsymbol{w})}_{\mathcal{H}^1}^2\ds\right)\bigg).
        \end{split}
    \end{align}
   Here, the constants denoted by $C_0$ depends only on $\Om, K, L, \alpha, \beta, F, G, \norm{(\mu_0,\theta_0)}_{L,\beta}$ and the initial energy $E_K(\phi_0,\psi_0)$.
\end{theorem}

The proof of this theorem will be presented in Section~\ref{SECT:HIGHREG}.

\medskip

\begin{remark} \label{Rem:HighReg}
    \begin{enumerate}[label=\textnormal{(\alph*)},topsep=0ex,leftmargin=*]
    \item\label{Rem:HighReg:a} 
    The results of Theorem~\ref{Theorem:HighReg} hold true in the case 
    $K = L = 0$
    if the regular parts of the potentials $F_2$ and $G_2$ satisfy the compatibility condition
    \begin{align*}
        F_2(\alpha s) = \alpha\beta G_2(s) \qquad\text{for all~}s\in[-1,1],
    \end{align*}
    with $\alpha,\beta\neq 0$. It seems that this compatibility condition cannot be avoided in our proof
    since in the underlying approximation scheme, the functions $\mu_{0,\lambda}$ (see \eqref{Def:Approx:M:0:lambda}) and $\theta_{0,\lambda}$ (see \eqref{Def:Approx:T:0:lambda}) need to satisfy the trace relation $\mu_{0,\lambda} = \beta\theta_{0,\lambda}$ almost everywhere on $\Gamma$ if $L = 0$.
    
    In the two-dimensional case (i.e., $d=2$) the statement of Theorem~\ref{Theorem:HighReg} remains correct also for $K\in [0,\infty)$ and $L=0$. In this case, both $\phi$ and $\psi$ satisfy a strict separation property (see~\cite[Theorem~3.11]{Knopf2024a}) and therefore, the proof can be carried out analogously, but without approximating the potentials and the initial data.
    
    \item\label{Rem:HighReg:b} The compatibility condition \ref{cond:MT:0} is, for example, fulfilled if initial data $(\phi_0,\psi_0)\in \mathcal{H}^3$ with 
    \begin{equation*}
        ( F'(\phi_0) , G'(\psi_0) ) \in \mathcal{H}^1
        \qquad\text{and}\qquad
        K\deln\phi_0 = \alpha\psi_0 - \phi_0 \quad\text{on $\Gamma$}
    \end{equation*}
    are prescribed. In this case, we simply set
    \begin{align*}
        (\mu_0,\theta_0) 
        &\coloneqq 
        (-\Lap\phi_0 + F^\prime(\phi_0) , 
        - \Lapg\psi_0 + G^\prime(\psi_0) + \alpha\deln\phi_0 )
        \in \mathcal{H}^1.
    \end{align*}
    \end{enumerate}
\end{remark}

\section{The subdifferential of the convex part of the free energy}
\label{SECT:SUBDIFF}

In this section, let $K\in [0,\infty)$ be arbitrary. 
To prove the regularity results stated in Theorem~\ref{THM:TIMEREG} and Theorem~\ref{Theorem:HighReg},
it makes sense to first gain more information about the energy functional $E_K$.
We recall that the potentials can be decomposed as $F=F_1+F_2$ and $G=G_1+G_2$, where $F_1$ and $G_1$ are the singular convex parts and $F_2$ and $G_2$ can be interpreted as smooth perturbations. 
To this end, we consider the functional $\widetilde{E}_K: \mathcal{L}^2 \to (-\infty,\infty]$ with
\begin{align*}
    \widetilde{E}_K(u,v) = 
    \begin{cases} 
    \displaystyle\intO \frac12\abs{\Grad u}^2 + F_1(u)\dx + \intG \frac12\abs{\Gradg v}^2 + G_1(v)\dG 
    &\smash{\raisebox{-1.6ex}
    {\text{for~}$(u,v)\in\mathrm{Dom}\,\tilde E_K,$}} 
    \\
    \quad + \displaystyle\sigma(K) \intG \frac12 (\alpha v - u)^2 \dG, 
    \\
    +\infty, &\text{else},
    \end{cases}
\end{align*}
where
\begin{equation*}
    \mathrm{Dom}\,\widetilde{E}_K = \big\{(u,v)\in\mathcal{H}^1_{K,\alpha}: \abs{u}\leq 1 
    \text{~a.e.~on~}\Om, \ \abs{v}\leq 1 \text{~a.e.~on~}\Ga\big\}.
\end{equation*}
Here, $F_1,G_1:[-1,1]\rightarrow\R$ satisfy the assumptions \eqref{Assumption:Pot:Sing}-\eqref{Domination:Prime}. We first observe that the functional $\widetilde E_K$ has the following properties.

\medskip

\begin{lemma}\label{Lemma:ConvEnergy:Prop}
    The functional $\widetilde E_K$ as defined above is proper, lower semi-continuous and convex.
\end{lemma}

\begin{proof}
    It is obvious that $\widetilde E_K$ is proper and convex since $F_1$ and $G_1$ are convex. The lower semi-continuity of $\widetilde E_K$ can be shown by proceeding analogously to \cite[Lemma~4.1]{Abels2007}. 
\end{proof}

\medskip

The main goal of this section is to investigate the subdifferential 
$\del \widetilde{E}_K$ of $\widetilde{E}_K$. It can be formulated as an operator
$\del \widetilde{E}_K:\mathcal{L}^2\rightarrow \mathcal{P}(\mathcal{L}^2)$, where
$\mathcal{P}(\mathcal{L}^2)$ denotes the power set of $\mathcal{L}^2$.
For any point $(u,v)\in\mathrm{Dom}\,\widetilde{E}_K$, 
it holds $(\zeta,\xi)\in\del\widetilde E_K(u,v)$ if and only if
\begin{align*}
 \bigscp{(\zeta,\xi)}{(u^\prime,v^\prime) - (u,v)}_{\mathcal{L}^2} \leq \widetilde E_K(u^\prime,v^\prime) - \widetilde E_K(u,v) \qquad\text{for all~}(u^\prime,v^\prime)\in\mathcal{L}^2.
\end{align*}
Moreover, the essential domain of $\del \widetilde{E}_K$ is given by
\begin{equation*}
    D(\del \widetilde{E}_K) 
    = \big\{ (u,v) \in \mathrm{Dom}\,\widetilde{E}_K : 
    \del \widetilde{E}_K(u,v) \neq \emptyset \big\}.
\end{equation*}
A characterization of the subdifferential $\del \widetilde{E}_K$ and its essential domain are provided by the following proposition.

\medskip

\begin{proposition}\label{App:Proposition:Subdiff}
    The subdifferential $\partial \widetilde{E}_K$ of $\widetilde{E}_K$ is a maximal monotone operator on $\mathcal{L}^2$, whose essential domain is given by
    \begin{align*}
        D(\partial \widetilde{E}_K) = \{(u, v)\in\mathcal{H}^2: (F_1^\prime(u), G_1^\prime(v))\in\mathcal{L}^2, \ K\deln u = \alpha v - u \ \text{on~}\Ga\}.
    \end{align*}
    For any $(u,v)\in D(\partial \widetilde{E}_K)$, it holds
    \begin{align}
    \label{EQ:DELE}
        \partial \widetilde{E}_K(u, v) = (-\Lap u + F_1^\prime(u), -\Lapg v + G_1^\prime(u) + \alpha\deln u).
    \end{align}
\end{proposition}

\medskip

Before proving this proposition, we briefly recall the Moreau--Yosida regularization, which will be applied on the singular parts $F_1$ and $G_1$, respectively. For any $\lambda\in (0,\infty)$, we define $F_{1,\lambda}, G_{1,\lambda}:\R\rightarrow[0,\infty)$ as
\begin{align}\label{APX:MY}
    F_{1,\lambda}(r) = \inf_{s\in\R}\left\{\frac{1}{2\lambda}\abs{r-s}^2 + F_1(s)\right\}, \qquad G_{1,\lambda}(r) = \inf_{s\in\R}\left\{\frac{1}{2\lambda}\abs{r-s}^2 + G_1(s)\right\}
\end{align}
for all $r\in\R$, respectively. Then, the derivatives are given as
\begin{align}\label{APX:Y}
    F_{1,\lambda}^\prime(r) = \frac1\lambda\Big( r - (I + \lambda F_1^\prime)^{-1}(r)\Big), \qquad G_{1,\lambda}^\prime(r) = \frac1\lambda\Big(r - (I + \lambda G_1^\prime)^{-1}(r)\Big)
\end{align}
for all $r\in\R$, respectively. 
For any $\lambda>0$, $F_\lambda$ has the following important properties:
\begin{enumerate}[label=\textbf{(M\arabic*)},topsep=0ex,leftmargin=*]
    \item \label{Yosida:Reg} For every $\lambda > 0$, $F_{1,\lambda}\in C^{2,1}_{\textup{loc}}(\R)$ with $F_{1,\lambda}(0) = F_{1,\lambda}^\prime(0) = 0$;
    \item \label{Yosida:Convexity}  $F_{1,\lambda}$ is convex with
    \begin{align*}
        F_{1,\lambda}^{\prime\prime}(r) \geq \frac{\Theta_\Om}{1 + \Theta_\Om} \quad\text{for all~}r\in\R;
    \end{align*}
    \item \label{Yosida:Lipschitz}  For every $\lambda > 0$, $F_{1,\lambda}^\prime$ is Lipschitz continuous on $\R$ with constant $\frac1\lambda$.
    \item \label{Yosida:Growth} There exists $\overline\lambda\in(0,1)$ and $C > 0$ such that
    \begin{align*}
        F_{1,\lambda}(r) \geq \frac{1}{4\overline\lambda}r^2 - C \qquad\text{for all~}r\in\R\text{~and~}\lambda\in(0,\overline\lambda);
    \end{align*}
    \item \label{Yosida:Convergence}  
    For every $\lambda > 0$, it holds $F_{1,\lambda}(r)\le  F_1(r)$ for all $r\in[-1,1]$.
    Moreover, as $\lambda \rightarrow 0$, we have $F_{1,\lambda}(r)\rightarrow F_1(r)$ for all $r\in[-1,1]$ and $\vert F_{1,\lambda}^\prime(r)\vert \rightarrow \abs{F_1^\prime(r)}$ for $r\in(-1,1)$.
\end{enumerate}
Analogous properties hold true for $G_{1,\lambda}$ instead of $F_{1,\lambda}$.
Additionally, as a consequence of the domination property \eqref{Domination:Prime}, 
we obtain an analogous estimate for the corresponding Yosida regularizations.
Namely, for any $\lambda>0$, it holds that
\begin{align}
    \label{DOMINATION:YOSIDA}
    \abs{F_{1,\lambda}^\prime(\alpha r)} \leq \kappa_1\abs{G_{1,\lambda}^\prime(r)} + \kappa_2 \qquad\text{for all~}r\in\R
\end{align}
(see, e.g., \cite[Lemma 4.4]{Calatroni2013}), where $\kappa_1, \kappa_2$ are the same constants as in \eqref{Domination:Prime}.

\medskip

We now proceed with the proof of the above proposition.

\begin{proof}[Proof of Proposition~\ref{App:Proposition:Subdiff}]
We follow the general scheme as in \cite[Proposition~2.9]{Barbu2010} and consider the operator $\mathcal{A}:\mathcal{L}^2\rightarrow\mathcal{L}^2$ that is defined as
\begin{align*}
    \mathcal{A}(u,v) &= (-\Lap u + F_1^\prime(u), -\Lapg v + G_1^\prime(v) + \alpha\deln u) \quad\text{for all~} (u,v)\in D(\mathcal{A}), 
\end{align*} 
where
\begin{align*}   
    D(\mathcal{A}) &= \{(u,v)\in\mathcal{H}^2: (F_1^\prime(u), G_1^\prime(v))\in\mathcal{L}^2, \ K\deln u = \alpha v - u \ \text{on~}\Ga\}.
\end{align*}

We first show $\mathcal{A}\subset \partial \widetilde{E}_K$. Recalling the convexity properties of $F_1$ and $G_1$ (especially \eqref{Assumption:Pot:Convexity}) and using integration by parts, a straightforward computation yields \pagebreak[3]
\begin{align*} 
    &\big(\mathcal{A}(u,v),(u,v) - (\zeta, \xi)\big)_{\mathcal{L}^2} 
    \nonumber\\[0.5ex]
    &= \intO (-\Lap u + F_1^\prime(u))(u - \zeta) + \intG (-\Lapg v + G_1^\prime(v) + \alpha\deln u)(v - \xi)\dG 
    \nonumber\\[0.5ex]
    &= \intO \Grad u\cdot(\Grad u - \Grad\zeta) + F_1^\prime(u)(u - \zeta)\dx  
    +  \sigma(K)\intG (\alpha v - u)\big((\alpha v - u) - (\alpha\xi - \zeta)\big)\dG 
    \nonumber\\
    &\quad + \intG \Gradg v \cdot(\Gradg v - \Gradg\xi) + G_1^\prime(v)(v - \xi)\dG
    \\[0.5ex]
    &\geq \intO \frac12\abs{\Grad u}^2 + F_1(u) \dx + \intG \frac12\abs{\Gradg v}^2 + G_1(v) \dG + \sigma(K) \intG \frac12 (\alpha v - u)^2\dG
    \nonumber\\
    &\quad - \intO \frac12\abs{\Grad\zeta}^2 + F_1(\zeta)\dx - \intG \frac12\abs{\Gradg\xi}^2 + G_1(\xi)\dG - \sigma(K) \intG \frac12 (\alpha\xi - \zeta)^2\dG
    \nonumber\\[0.5ex]
    &= \widetilde{E}_K(u,v) 
    - \widetilde{E}_K(\zeta, \xi) 
    \nonumber
\end{align*}
for all $(u,v)\in D(\mathcal{A}), (\zeta, \xi)\in\mathcal{H}^1_{K,\alpha}$. 
This directly implies $\mathcal{A}(u,v) \in \partial\widetilde{E}_K$ for all $(u,v) \in D(\mathcal{A})$, which means that $\mathcal{A}\subset \partial \widetilde{E}_K$.

In order to show that $\mathcal{A} = \partial \widetilde{E}_K$, it suffices to prove that $\mathcal{A}$ is a maximal monotone operator in $\mathcal{L}^2$, i.e., $R(I + \mathcal{A}) = \mathcal{L}^2$. 
To this end, we fix $(f,g)\in\mathcal{L}^2$ and consider the following system:
\begin{subequations}
\label{EQ:I+A}
\begin{alignat}{2}
    \label{EQ:I+A:1}
    u - \Lap u + F_1^\prime(u) &= f &&\qquad\text{in~}\Om, 
    \\
    \label{EQ:I+A:2}
    v - \Lapg v + G_1^\prime(v) + \alpha\deln u &= g &&\qquad\text{on~}\Ga, 
    \\
    \label{EQ:I+A:3}
    K\deln u &= \alpha v - u &&\qquad\text{on~}\Ga. 
\end{alignat}
\end{subequations}
We approximate \eqref{EQ:I+A} by
\begin{subequations}
\label{EQ:I+A:lambda}
\begin{alignat}{2}
    \label{EQ:I+A:lambda:1}
    u - \Lap u + F_\lambda^\prime(u) &= f &&\qquad\text{in~}\Om, 
    \\
    \label{EQ:I+A:lambda:2}
    v - \Lapg v + G_\lambda^\prime(v) + \alpha\deln u &= g &&\qquad\text{on~}\Ga, 
    \\
    \label{EQ:I+A:lambda:3}
    K\deln u &= \alpha v - u &&\qquad\text{on~}\Ga 
\end{alignat}
\end{subequations}
for any $\lambda>0$, where $F_{1,\lambda}$ and $G_{1,\lambda}$ are the Moreau--Yosida regularizations of the potentials $F_1$ and $G_1$ that were introduced above.
We start by showing that, for any $\lambda>0$, System~\eqref{EQ:I+A:lambda} has a unique solution $(u_\lambda, v_\lambda)\in\mathcal{H}^2$. For this purpose, we fix an arbitrary $\lambda>0$ and consider the operator $T_\lambda:\mathcal{L}^2\rightarrow\mathcal{L}^2$, $(u,v)\mapsto (\overline{u}, \overline{v})$, where $(\overline{u}, \overline{v})\in\mathcal{H}^1_{K,\alpha}$ is the unique weak solution of
\begin{subequations}
\label{EQ:I+A:lambda:AUX}
\begin{alignat}{2}
    \label{EQ:I+A:lambda:AUX:1}
    \lambda\overline{u} - \lambda\Lap\overline{u} + \overline{u} &= \lambda f + (1 + \lambda F_1^\prime)^{-1}(u) 
    &&\qquad\text{in~}\Om, 
    \\
    \label{EQ:I+A:lambda:AUX:2}
    \lambda\overline{v} - \lambda\Lapg\overline{v} + \lambda\alpha\deln\overline{u} + \overline{v} &= \lambda g + (1 + \lambda G_1^\prime)^{-1}(v) 
    &&\qquad\text{on~}\Ga, 
    \\
    \label{EQ:I+A:lambda:AUX:3}
    K\deln\overline{u} &= \alpha\overline{v} - \overline{u} &&\qquad\text{on~}\Ga. 
\end{alignat}
\end{subequations}
Note that the existence of this unique weak solution $(\overline{u}, \overline{v})$ follows from the Lax--Milgram lemma (see also \cite{Knopf2021}). 
Let now $(u_1, v_1), (u_2, v_2) \in\mathcal{L}^2$ be arbitrary, and let $(\overline{u}_1, \overline{v}_1)$ and $(\overline{u}_2, \overline{v}_2)$ denote the corresponding solutions of System~\eqref{EQ:I+A:lambda:AUX}, respectively.
Using integration by parts, we readily infer that
\begin{align*}
    \begin{split}
        &\norm{(\overline{u}_1,\overline{v}_1) - (\overline{u}_2,\overline{v}_2)}_{\mathcal{H}^1}^2 + \norm{(\alpha\overline{v}_1 - \overline{u}_1) - (\alpha\overline{v}_2 - \overline{u}_2)}_{L^2(\Ga)}^2 + \frac1\lambda \norm{(\overline{u}_1,\overline{v}_1) - (\overline{u}_2,\overline{v}_2)}_{\mathcal{L}^2}^2 \\
        &= \frac1\lambda\intO \big((1+\lambda F_1^\prime)^{-1}(u_1) - (1+\lambda F_1^\prime)^{-1}(u_2)\big)(\overline{u}_1 - \overline{u}_2)\dx \\
        &\qquad + \frac1\lambda\intG \big((1+\lambda G_1^\prime)^{-1})(v_1) - (1+\lambda G_1^\prime)^{-1}(v_2)\big)(\overline{v}_1 - \overline{v}_2)\dG.
    \end{split}
\end{align*}
Invoking the estimate
\begin{align*}
    \abs{(1 + \lambda F_1^\prime)^{-1}(x) - (1 + \lambda F_1^\prime)^{-1}(y)} \leq \abs{x - y} \quad\text{for all~}x,y\in\R, 
\end{align*}
we find 
\begin{align*}
    \norm{T_\lambda(u_1,v_1) - T_\lambda(u_2,v_2)}_{\mathcal{H}^1}^2 + \frac{1}{\lambda} \norm{T_\lambda(u_1,v_1) - T_\lambda(u_2,v_2)}_{\mathcal{L}^2}^2 \leq \frac1\lambda \norm{(u_1,v_1) - (u_2, v_2)}_{\mathcal{L}^2}^2.
\end{align*}
Consequently, we have
\begin{align*}
    \norm{T_\lambda(u_1,v_1) - T_\lambda(u_2,v_2)}_{\mathcal{L}^2} \leq \frac{1}{\sqrt{1+\lambda}} \norm{(u_1,v_1) - (u_2,v_2)}_{\mathcal{L}^2}.
\end{align*}
This means that the map $T_\lambda$ is a contraction. Therefore, applying the Banach fixed-point theorem, we conclude the existence of a pair $(u,v)\in\mathcal{L}^2$ such that $T_\lambda(u,v) = (u,v)$. 
Recalling \eqref{APX:Y}, this proves that System~\eqref{EQ:I+A:lambda} has a unique weak solution $(u_\lambda, v_\lambda)\in\mathcal{H}^1$, which means that $(u_\lambda, v_\lambda)$ satisfies
\begin{align}\label{WF:I+A:lambda}
    \begin{split}
        &\intO u_\lambda\zeta + \Grad u_\lambda\cdot\Grad\zeta + F_{1,\lambda}^\prime(u_\lambda)\zeta\dx + \intG v_\lambda\xi + \Gradg v_\lambda\cdot\Gradg\xi + G_{1,\lambda}^\prime(v_\lambda)\xi\dG  \\
        & \quad + \sigma(K) \intG (\alpha v_\lambda - u_\lambda)(\alpha\xi - \zeta)\dG = \intO f\zeta \dx + \intG g \xi \dG
    \end{split}
\end{align}
for all $(\zeta,\xi)\in\mathcal{H}^1_{K,\alpha}$. Now, testing \eqref{WF:I+A:lambda} with $(u_\lambda, v_\lambda)$ yields
\begin{align*}
    &\intO \abs{u_\lambda}^2 + \abs{\Grad u_\lambda}^2 + F^\prime_{1,\lambda}(u_\lambda)u_\lambda \dx + \intG \abs{v_\lambda}^2 + \abs{\Gradg v_\lambda}^2 + G^\prime_{1,\lambda}(v_\lambda)v_\lambda\dG \\
    &
    \quad + \sigma(K) \intG (\alpha v_\lambda - u_\lambda)^2 \dG 
    = \intO fu_\lambda\dx + \intG gv_\lambda\dG.
\end{align*}
Since $F^\prime_{1,\lambda}(s)s \geq F_{1,\lambda}(s)$ as well as $G^\prime_{1,\lambda}(s)s \geq G_{1,\lambda}(s)$ for all $s\in\R$, we infer
\begin{align}
\label{EST:H1:lambda}
    \begin{split}
        \norm{(u_\lambda,v_\lambda)}_{\mathcal{H}^1} + \norm{(F_{1,\lambda}(u_\lambda),G_{1,\lambda}(v_\lambda))}_{\mathcal{L}^1} + \sigma(K)\norm{\alpha v_\lambda - u_\lambda}_{L^2(\Ga)} \leq \norm{(f,g)}_{\mathcal{L}^2}.
    \end{split}
\end{align}
In the following, the letter $C$ will denote generic positive constants that are independent of $\lambda$.
Our next goal is to show that $(u_\lambda, v_\lambda)$ belongs to $\mathcal{H}^2$ and fulfills the estimate
\begin{align}
    \label{EST:UNI:LAMBDA}
    \norm{(u_\lambda, v_\lambda)}_{\mathcal{H}^2} + \norm{(F_{1,\lambda}^\prime(u_\lambda), G_{1,\lambda}^\prime(v_\lambda))}_{\mathcal{L}^2} \leq C\big(1 + \norm{(f,g)}_{\mathcal{L}^2}\big), 
\end{align}
which is uniform with respect to $\lambda$.
For this purpose, we will consider the cases $K\in(0,\infty)$ and $K=0$ separately.

\textit{The case $K\in(0,\infty)$.} In this case, we may use $(F^\prime_{1,\lambda}(u_\lambda), G^\prime_{1,\lambda}(v_\lambda))$ as a test function in \eqref{WF:I+A:lambda}. We obtain
\begin{align*}
    \begin{split}
        &\norm{F^\prime_{1,\lambda}(u_\lambda)}_{L^2(\Om)}^2 + \norm{G^\prime_{1,\lambda}(v_\lambda)}_{L^2(\Ga)}^2 + \intO F^\prime_{1,\lambda}(u_\lambda)u_\lambda\dx + \intG G^\prime_{1,\lambda}(v_\lambda)v_\lambda\dG \\
        &\quad + \intO F^{\prime\prime}_{1,\lambda}(u_\lambda)\abs{\Grad u_\lambda}^2 \dx + \intG G^{\prime\prime}_{1,\lambda}(v_\lambda)\abs{\Gradg v_\lambda}^2\dG + \frac1K \intG (\alpha v_\lambda - u_\lambda)(\alpha G^\prime_{1,\lambda}(v_\lambda) - F^\prime_{1,\lambda}(u_\lambda))\dG \\
        &= \intO f F^\prime_{1,\lambda}(u_\lambda)\dx + \intG g G^\prime_{1,\lambda}(v_\lambda)\dG.
    \end{split}
\end{align*}
Since $F^{\prime\prime}_{1,\lambda}$ and $G^{\prime\prime}_{1,\lambda}$ are non-negative, we can argue similarly as above to infer
\begin{align}\label{EST:prime:lambda}
    \frac34\norm{(F^\prime_{1,\lambda}(u_\lambda), G^\prime_{1,\lambda}(v_\lambda))}_{\mathcal{L}^2} \leq \norm{(f,g)}_{\mathcal{L}^2} - \frac1K\intG (\alpha v_\lambda - u_\lambda)(\alpha G^\prime_{1,\lambda}(v_\lambda) - F^\prime_{1,\lambda}(u_\lambda))\dG.
\end{align}
For the last term on the right-hand side, we use the monotonicity of $F^\prime_{1,\lambda}$ as well as the domination property \eqref{DOMINATION:YOSIDA} to deduce that
\begin{align*}
    & - \frac1K\intG (\alpha v_\lambda - u_\lambda)(\alpha G^\prime_{1,\lambda}(v_\lambda) - F^\prime_{1,\lambda}(u_\lambda))\dG \\
    &= - \frac1K\intG (\alpha v_\lambda - u_\lambda)(F^\prime_{1,\lambda}(\alpha v_\lambda) - F^\prime_{1,\lambda}(u_\lambda))\dG - \frac1K\intG (\alpha v_\lambda - u_\lambda)(\alpha G^\prime_{1,\lambda}(v_\lambda) - F^\prime_{1,\lambda}(\alpha v_\lambda))\dG \\
    &\leq - \frac1K\intG (\alpha v_\lambda - u_\lambda)(\alpha G^\prime_{1,\lambda}(v_\lambda) - F^\prime_{1,\lambda}(\alpha v_\lambda))\dG \\
    &\leq \frac14\norm{G^\prime_{1,\lambda}(v_\lambda)}_{L^2(\Ga)}^2 + C(\norm{u_\lambda}_{H^1(\Om)}^2 + \norm{v_\lambda}_{L^2(\Ga)}^2).
\end{align*}
Using this inequality to estimate the right-hand side of \eqref{EST:prime:lambda}, we conclude that
\begin{align*}
    \norm{(F^\prime_{1,\lambda}(u_\lambda), G^\prime_{1,\lambda}(v_\lambda))}_{\mathcal{L}^2} \leq C.
\end{align*}
Finally, by regularity theory for elliptic systems with bulk-surface coupling (see \cite[Theorem 3.3]{Knopf2021}), we infer
\begin{align*}
    \norm{(u_\lambda,v_\lambda)}_{\mathcal{H}^2} \leq C\big(1 + \norm{(f - u_\lambda - F^\prime_{1,\lambda}(u_\lambda), g - v_\lambda - G^\prime_{1,\lambda}(v_\lambda))}_{\mathcal{L}^2}\big) \leq C\big(1 + \norm{(f,g)}_{\mathcal{L}^2}\big).
\end{align*}

\textit{The case $K=0$.} In this case, we have to argue a bit differently, since we cannot use $ (\zeta, \xi)= (F_{1,\lambda}^\prime(u_\lambda), G_{1,\lambda}^\prime(v_\lambda))$ in \eqref{WF:I+A:lambda} as a test function as it does not satisfy the trace relation. 
On the other hand, by applying elliptic systems with bulk-surface coupling (see \cite[Theorem 3.3]{Knopf2021}), we infer that $(u_\lambda,v_\lambda) \in \mathcal{H}^2$ and therefore, the equations \eqref{EQ:I+A:lambda} are satisfied almost everywhere in $\Om$ and on $\Ga$, respectively.
This allows us to multiply \eqref{EQ:I+A:lambda:1} by $F_{1,\lambda}^\prime(u_\lambda)$, and \eqref{EQ:I+A:lambda:2} by $G_{1,\lambda}^\prime(v_\lambda)$, and integrate over $\Omega$ and $\Gamma$, respectively. Via integration by parts, we then deduce that
\begin{align}\label{EST:prime:lambda:0}
    \begin{split}
        &\norm{F^\prime_{1,\lambda}(u_\lambda)}_{L^2(\Om)}^2 + \norm{G^\prime_{1,\lambda}(v_\lambda)}_{L^2(\Ga)}^2 + \intO F^\prime_{1,\lambda}(u_\lambda)u_\lambda\dx + \intG G^\prime_{1,\lambda}(v_\lambda)v_\lambda\dG \\
        &\quad + \intO F^{\prime\prime}_{1,\lambda}(u_\lambda)\abs{\Grad u_\lambda}^2 \dx + \intG G^{\prime\prime}_{1,\lambda}(v_\lambda)\abs{\Gradg v_\lambda}^2\dG + \intG \alpha\deln u_\lambda(\alpha G^\prime_{1,\lambda}(v_\lambda) - F^\prime_{1,\lambda}(u_\lambda))\dG \\
        &= \intO f F^\prime_{1,\lambda}(u_\lambda)\dx + \intG g G^\prime_{1,\lambda}(v_\lambda)\dG.
    \end{split}
\end{align}
Using the trace relation $u_\lambda = \alpha v_\lambda$ on $\Ga$ and the domination property \eqref{DOMINATION:YOSIDA}, we argue analogously as above to obtain
\begin{align}\label{EST:FG:lambda:0}
    \norm{(F_{1,\lambda}^\prime(u_\lambda), G_{1,\lambda}^\prime(v_\lambda))}_{\mathcal{L}^2} \leq C\big( 1 + \norm{\deln u_\lambda}_{L^2(\Ga)}\big).
\end{align}
Thus, using the regularity estimate for the bulk-surface elliptic problem (see \cite[Theorem 3.3]{Knopf2021}), we infer that
\begin{align}\label{EST:H^2+deln:lambda}
    \norm{(u_\lambda, v_\lambda)}_{\mathcal{H}^2} \leq C\big(1 + \norm{(f,g)}_{\mathcal{L}^2} + \norm{\deln u_\lambda}_{L^2(\Ga)}\big).
\end{align}
In order to close the estimate, we are left to control the term $\norm{\deln u_\lambda}_{L^2(\Ga)}$. Using the trace theorem (see, e.g., \cite[Chapter 2, Theorem 2.24]{Brezzi1987}), we deduce that
\begin{align*}
    \norm{\deln u_{\lambda}}_{L^2(\Ga)} \leq C\norm{u_\lambda}_{H^{7/4}(\Om)}.
\end{align*}
Thus, using \eqref{EST:H1:lambda} as well as the compact embeddings $H^2(\Om)\emb H^{7/4}(\Om) \emb L^2(\Om)$ along with Ehrling's lemma, we infer that for any $\varepsilon>0$, it holds
\begin{align}\label{EST:deln}
    \norm{\deln u_\lambda}_{L^2(\Ga)} \leq  \varepsilon  \norm{(u_\lambda, v_\lambda)}_{\mathcal{H}^2} + C_\varepsilon
    \norm{(f,g)}_{\mathcal{L}^2}.
\end{align}
for some constant $C_\varepsilon>0$ that may depend on $\varepsilon$.
Eventually, recalling \eqref{EST:FG:lambda:0} and choosing a suitable value of $\varepsilon$, we deduce from \eqref{EST:H^2+deln:lambda}-\eqref{EST:deln} that
\begin{align}
    \label{EST:UNI}
    \norm{(u_\lambda, v_\lambda)}_{\mathcal{H}^2} + \norm{(F_{1,\lambda}^\prime(u_\lambda), G_{1,\lambda}^\prime(v_\lambda))}_{\mathcal{L}^2} \leq C\big(1 + \norm{(f,g)}_{\mathcal{L}^2}\big).
\end{align}
Consequently, according to the Banach--Alaoglu theorem and the Aubin--Lions--Simon lemma, there exist functions $u, v, \chi$ and $\chi_\Ga$ such that
\begin{alignat*}{3}
     \scp{u_\lambda}{v_\lambda} &\rightarrow \scp{u}{v} \qquad&&\text{weakly in~} \mathcal{H}^2 \ \text{~and strongly in~}\mathcal{H}^1_{K,\alpha}, \\
     \scp{F^\prime_{1,\lambda}(u_\lambda)}{G^\prime_{1,\lambda}(v_\lambda)} &\rightarrow \scp{\chi}{\chi_\Ga} &&\text{weakly in~} \mathcal{L}^2. 
\end{alignat*}
From these convergences, we readily infer that $(u,v)$ satisfies
\begin{alignat}{2}\label{EQ:I+A:chi}
    u - \Lap u + \chi &= f &&\qquad\text{a.e.~in~}\Om, \nonumber \\
    v - \Lapg v + \chi_\Ga + \alpha\deln u &= g &&\qquad\text{a.e.~on~}\Ga, \\
    K\deln u &= \alpha v - u &&\qquad\text{a.e.~on~}\Ga. \nonumber
\end{alignat}
Similarly as in the proof of Theorem~\ref{THEOREM:EOWS}, the identities $F_1^\prime(u) = \chi$ and $G_1^\prime(v) = \chi_\Ga$ follow by means of the theory of maximal of monotone operators (see, e.g., \cite[Proposition 1.1, p.42]{Barbu2010} and also\cite[Section 5.2]{Garcke2017}) together with the convergences
\begin{align*}
    \lim_{\lambda\rightarrow 0}\intO F_{1,\lambda}^\prime(u_\lambda)u_\lambda\dx &= \intO \chi u\dx, \\
    \lim_{\lambda\rightarrow 0}\intG G_{1,\lambda}^\prime(v_\lambda)v_\lambda\dG &= \intG \chi_\Ga v\dG.
\end{align*}
In addition, this shows that $(F_{1,\lambda}^\prime(u_\lambda), G_{1,\lambda}^\prime(v_\lambda)) \in \mathcal{L}^2$ and from \eqref{EST:UNI}, we conclude that
\begin{align*}
    \norm{(u, v)}_{\mathcal{H}^2} + \norm{(F_1^\prime(u), G_1^\prime(v))}_{\mathcal{L}^2} 
    \leq C\big(1 + \norm{(f,g)}_{\mathcal{L}^2}\big)
\end{align*}
by means of weak lower semicontinuity.

We have thus shown that $(u,v)$ is a solution of \eqref{EQ:I+A}, and because $(f,g)\in\mathcal{L}^2$ was arbitrary, we infer that $\mathcal{A} = \partial \widetilde{E}_K$. Hence, the proof is complete.
\end{proof}

\section{A bulk-surface elliptic system with singular nonlinearities}
\label{SECT:BSE}

In this section, we establish well-posedness and regularity results for a bulk-surface elliptic system with singular nonlinearities. These regularity results will be the key ingredient in the proofs of Theorem~\ref{THM:TIMEREG} and Theorem~\ref{Theorem:HighReg}. The elliptic system reads as follows:
\begin{subequations}
\label{BSE}
    \begin{align}
        \label{BSE:1}
        -\Lap u + F_1^\prime(u) &= f &\text{in $\Omega$},\\
        \label{BSE:2}
        -\Lapg v + G_1^\prime(v) + \alpha \deln u &= g &\text{on $\Gamma$},\\
        \label{BSE:3}
        K\deln u &= \alpha v - u 
        &\text{on $\Gamma$}.
    \end{align}        
\end{subequations}
Here and in the remainder of this section, we consider $K\in[0,\infty)$, and $F_1$ and $G_1$ are assumed to satisfy the assumptions \eqref{Assumption:Pot:Sing}-\eqref{Domination:Prime}. We point out that results similar to those in this section can also be obtained for regular potentials $F_1,G_1:\R\to \R$ that have suitable polynomial growth (see also Remark~\ref{REM:YOS}).

We now present a well-posedness result for \eqref{BSE} along with additional regularity results in the framework of Sobolev spaces. By means of the theory developed in Section~\ref{SECT:SUBDIFF}, we immediately infer the existence of a solution to System~\eqref{BSE}. This is stated by the following proposition.

\begin{proposition}\label{App:Prop:Existence}
    Let $(f,g)\in\mathcal{L}^2$.
    Then, there exists a solution $(u,v)\in\mathcal{H}^2\cap\mathcal{H}^1_{K,\alpha}$ of System~\eqref{BSE} such that 
    $(F_1^\prime(u), G_1^\prime(v))\in\mathcal{L}^2$ and the equations \eqref{BSE:1}-\eqref{BSE:3} are satisfied almost everywhere. 
    In particular, it holds $\abs{u}< 1$ a.e.~in $\Omega$ and $\abs{v}< 1$ a.e.~on $\Gamma$.
    Moreover, there exists a positive constant $C$ depending only on $\Om$, $K$, $\alpha$, $\Theta_\Om$, $\Theta_\Ga$, $\kappa_1$ and $\kappa_2$ such that
    \begin{align}\label{App:Est:H^2+L^2}
        \norm{(u,v)}_{\mathcal{H}^2} + \norm{(F_1^\prime(u), G_1^\prime(v))}_{\mathcal{L}^2} \leq C\big(1 + \norm{(f,g)}_{\mathcal{L}^2}\big).
    \end{align}
\end{proposition}

\begin{proof}
  Invoking formula \eqref{EQ:DELE} from Proposition~\ref{App:Proposition:Subdiff}, recalling \eqref{Assumption:Pot:Convexity}, and using integration by parts, it is straightforward to check that
  \begin{align*}
      (\partial \widetilde{E}_K(u, v) - \partial \widetilde{E}_K(\zeta, \xi), (u,v) - (\zeta,\xi))_{\mathcal{L}^2} \geq \Theta_* \norm{(u,v) - (\zeta,\xi)}_{\mathcal{L}^2}^2
  \end{align*}
  for all $(u, v), (\zeta, \xi)\in D(\partial \widetilde{E}_K)$, where $\Theta_* = \min\{\Theta_\Om,\Theta_\Ga\}$.
  This means that $\partial \widetilde{E}_K$ is strongly monotone and thus surjective (i.e., $R(\partial \widetilde{E}_K) = \mathcal{L}^2$). Hence, the existence of a strong solution is a direct consequence. Moreover, estimate \eqref{App:Est:H^2+L^2} can be verified by proceeding analogously to the proof of Proposition~\ref{App:Proposition:Subdiff} (especially the derivation of \eqref{EST:UNI}) together with regularity theory for elliptic systems with bulk-surface coupling (see, e.g., \cite{Knopf2021}).
\end{proof}

Next, in the following proposition, we prove a stability and uniqueness result for weak solutions of System~\eqref{BSE}. In particular, this proves that the solution of \eqref{BSE}, which exists according to Proposition~\ref{App:Prop:Existence}, is unique.

\begin{proposition}\label{App:Prop:ContDep}
    Assume that $(f_1,g_1), (f_2,g_2)\in\mathcal{L}^2$, and let $(u_1, v_1), (u_2, v_2) \in\mathcal{H}^2\cap \mathcal{H}^1_{K,\alpha}$ be corresponding solutions of System~\eqref{BSE}, respectively. Then there exists a constant $C > 0$ depending only on $\Om$, $K$, $\alpha$, $\Theta_\Om$ and $\Theta_\Ga$ such that
    \begin{align}\label{EST:CONTDEP}
        \norm{(u_1 - u_2, v_1 - v_2)}_{\mathcal{H}^1} \leq C \norm{(f_1 - f_2, g_1 - g_2)}_{\mathcal{L}^2}.
    \end{align}
    In particular, this entails that the solution of System~\eqref{BSE} is unique. 
\end{proposition}

\begin{proof}
    For brevity, we define $(\tilde{u}, \tilde{v}) = (u_1 - u_2, v_1 - v_2)$ and $(\tilde{f}, \tilde{g}) = (f_1 - f_2, g_1 - g_2)$. Taking the difference of the corresponding weak formulations, we deduce that
    \begin{align*}
        &\intO \Grad \tilde{u}\cdot\Grad\zeta\dx + \intO (F_1^\prime(u_1) - F_1^\prime(u_2))\zeta\dx + \intG \Gradg \tilde{v}\cdot\Gradg\xi \dG + \intG (G_1^\prime(v_1) - G_1^\prime(v_2))\xi\dG \\
        &\qquad + \sigma(K) \intG (\alpha \tilde{v} - \tilde{u})(\alpha\xi - \zeta) \dG  = \intO \tilde{f}\zeta\dx + \intG \tilde{g}\xi\dG
    \end{align*}
    holds for all $(\zeta,\xi)\in\mathcal{H}^1_{K,\alpha}$. Using $(\zeta, \xi) = (\tilde{u}, \tilde{v}) \in\mathcal{H}^1_{K,\alpha}$ as a test function, and exploiting \eqref{Assumption:Pot:Convexity}, we easily obtain
    \begin{align*}
        \norm{(\tilde{u}, \tilde{v})}_{K,\alpha}^2 + \Theta_\ast\norm{(\tilde{u},\tilde{v})}_{\mathcal{L}^2}^2\leq \intO \tilde{f}\tilde{u}\dx + \intG \tilde{g}\tilde{v}\dG,
    \end{align*}
    where, as above, $\Theta_\ast = \min\{\Theta_\Om,\Theta_\Ga\}$.
    Using Hölder's and Young's inequalities, we readily infer \eqref{EST:CONTDEP}. 
\end{proof}

\medskip

Now we deal with some regularity properties, which will play an essential role in the sequel.

\begin{proposition}\label{App:Proposition:Pot:L^p}
    Let $(f,g)\in\mathcal{L}^p$ with $2\leq p \leq \infty$ and let $(u,v)\in\mathcal{H}^2\cap\mathcal{H}^1_{K,\alpha}$ be the corresponding solution of System~\eqref{BSE}. 
    Then $(u,v)\in\mathcal{W}^{2,p}$ and $(F_1^\prime(u), G_1^\prime(v)) \in \mathcal{L}^p$ and there exists a constant $C = C(p) > 0$, such that
    \begin{align}\label{App:Est:Pot:L^p}
        \norm{(u,v)}_{\mathcal{W}^{2,p}} + \norm{(F_1^\prime(u), G_1^\prime(v))}_{\mathcal{L}^p} \leq C\big( 1 + \norm{(f,g)}_{\mathcal{L}^p} 
        + \gamma(K) \norm{\deln u}_{L^p(\Ga)} \big),
    \end{align}
    where $\gamma(K) = \mathbf{1}_{\{0\}}(K)$ denotes the characteristic function of the set $\{0\}$. Moreover, if $(f,g)\in\mathcal{L}^\infty$,
    there exists $\delta\in(0,1]$, depending on $\norm{(f,g)}_{\mathcal{L}^\infty}$ and $\norm{\deln u}_{L^\infty(\Ga)}$, such that
    \begin{align}
        \label{BSE:SEPPROP}
        \abs{u(x)} \leq 1 - \delta \quad\text{for all~}x\in\overline{\Om}, \qquad 
        \abs{v(z)} \leq 1 - \delta \quad\text{for all~}z\in \Ga.
    \end{align}
\end{proposition}

\begin{proof}
    In this proof, we assume without loss of generality that $\alpha\neq 0$. The case $\alpha = 0$ can be handled similarly and the computations are even easier. In addition, we can assume that $p \in(2,\infty]$ as the case $p = 2$ is already covered by Proposition~\ref{App:Prop:Existence}.
    For $k\in\N$, we introduce the truncation 
    \begin{align*}
        h_k:\R\rightarrow\R,
        \quad
        h_k(s) = \begin{cases}
            -1 + \frac1k, & s < -1 + \frac1k, \\
            s, & s\in \big[-1 + \frac1k, 1 - \frac1k \big], \\
            1 - \frac1k, & s > 1 - \frac1k,
        \end{cases}
    \end{align*}
    which is globally Lipschitz continuous.
    Then, we define $u_k = \alpha h_k\circ (\alpha^{-1}u)$, $v_k = h_k \circ v$. Since $(u,v)\in\mathcal{H}^1$, a well-known result on the composition of Lipschitz and Sobolev functions entails that $(u_k,v_k)\in\mathcal{H}^1$ along with 
    \begin{align*}
        \Grad u_k = \Grad u \,\mathbf{1}_{[-1+1/k,1-1/k]}(u), \qquad \Gradg v_k = \Gradg v \,\mathbf{1}_{[-1+1/k,1-1/k]}(v)
    \end{align*}
    for any $k\in\N$.
    
    First, we consider the case $p\in(2,\infty)$. 
    To verify the assertion, we now need to handle the cases $K\in(0,\infty)$ and $K = 0$ separately. In the following, the letter $C$ always denotes a non-negative constant independent of $k$, whose value may change from line to line.

    \textit{The case $p\in(2,\infty)$ and $K\in(0,\infty)$.} 
    We multiply \eqref{BSE:1} with $F_1^\prime(u_k) \abs{F_1^\prime(u_k)}^{p-2}\in H^1(\Om)$ and integrate over $\Om$ and perform an integration by parts. Moreover, we multiply \eqref{BSE:2} with $G_1^\prime(v_k)\abs{G_1^\prime(v_k)}^{p-2}\in H^1(\Ga)$, integrate over $\Ga$ and integrate by parts. Adding the resulting equations leads to
    \begin{align}\label{App:Est:Pot:Test}
        \begin{split}
            &(p-1)\intO \Grad u \cdot\Grad u_k \abs{F_1^\prime(u_k)}^{p-2} F_1^{\prime\prime}(u_k)\dx + (p-1)\intG \Gradg v\cdot\Gradg v_k \abs{G_1^\prime(v_k)}^{p-2} G_1^{\prime\prime}(v_k)\dG \\
            &\qquad + \intO F_1^\prime(u) F_1^\prime(u_k) \abs{F_1^\prime(u_k)}^{p-2}\dx + \intG G_1^\prime(v)G_1^\prime(v_k)\abs{G_1^\prime(v_k)}^{p-2}\dG \\
            &\quad = \intO f F_1^\prime(u_k)\abs{F_1^\prime(u_k)}^{p-2}\dx + \intG gG_1^\prime(v_k)\abs{G_1^\prime(v_k)}^{p-2} \\
            &\qquad - \sigma(K) \intG (\alpha v - u)\big(\alpha G_1^\prime(v_k)\abs{G_1^\prime(v_k)}^{p-2} - F_1^\prime(u_k)\abs{F_1^\prime(u_k)}^{p-2}\big)\dG.
        \end{split}
    \end{align}
    We first notice that in view of the convexity of $F_1$ and $G_1$, the first two terms on the left-hand side of \eqref{App:Est:Pot:Test} are non-negative. Thanks to \eqref{Assumption:Pot:Convexity}, $F_1^\prime(s)$ and $F_1^\prime(\alpha h_k(\alpha^{-1} s))$, as well $G_1^\prime(s)$ and $G_1^\prime(h_k(s))$, have the same sign for all $s\in(-1,1)$, which yields
    \begin{alignat*}{2}
        F_1^\prime(u_k)^2 &\leq F_1^\prime(u) F_1^\prime(u_k)  &&\qquad\text{a.e.~in~}\Om, \\
        G_1^\prime(v_k)^2 &\leq G_1^\prime(v) G_1^\prime(v_k) &&\qquad\text{a.e.~on~}\Ga.
    \end{alignat*}
    Consequently, it holds
    \begin{align}\label{App:Est:Pot:Test:1}
        \begin{split}
        &\intO \abs{F_1^\prime(u_k)}^p\dx + \intG \abs{G_1^\prime(v_k)}^p \dG
        \\
        &\quad \leq \intO F_1^\prime(u)F_1^\prime(u_k)\abs{F_1^\prime(u_k)}^{p-2}\dx + \intG G_1^\prime(v)G_1^\prime(v_k)\abs{G_1^\prime(v_k)}^{p-2}\dG.
        \end{split}
    \end{align}
    For the last term on the right-hand side of \eqref{App:Est:Pot:Test}, we exploit the domination property \eqref{Domination:Prime} as well as the monotonicity of the function $\R\ni s\mapsto F_1^\prime(\alpha h_k(s))\abs{F_1^\prime(\alpha h_k(s))}^{p-2}$ to infer
    \begin{align}\label{App:Est:Pot:Test:2}
        \begin{split}
            &- \sigma(K) \intG (\alpha v - u)\big(\alpha G_1^\prime(v_k)\abs{G_1^\prime(v_k)}^{p-2} - F_1^\prime(u_k)\abs{F_1^\prime(u_k)}^{p-2}\big)\dG \\
            &\quad = - \sigma(K) \intG (\alpha v - u)\big(F_1^\prime(\alpha v_k)\abs{F_1^\prime(\alpha v_k)}^{p-2} - F_1^\prime(u_k)\abs{F_1^\prime(u_k)}^{p-2}\big)\dG \\
            &\qquad - \sigma(K) \intG (\alpha v - u)\big(\alpha G_1^\prime(v_k)\abs{G_1^\prime(v_k)}^{p-2} - F_1^\prime(\alpha v_k)\abs{F_1^\prime(\alpha v_k)}^{p-2}\big)\dG \\
            &\quad \leq 2\sigma(K) \Big(\intG \abs{G_1^\prime(v_k)}^{p-1}\dG + 2^{p-2}\kappa_1^{p-2}\intG \abs{G_1^\prime(v_k)}^{p-1}\dG + 2^{p-2}\kappa_2^{p-1}\abs{\Ga}\Big).
        \end{split}
    \end{align}
    Here, we additionally made use of Minkowski's inequality
    \begin{align*}
        \abs{a+b}^{p-1} \leq 2^{p-2}\big(\abs{a}^{p-1} + \abs{b}^{p-1}\big) \qquad\text{for any~}a,b\in\R.
    \end{align*}
    Combining \eqref{App:Est:Pot:Test} with \eqref{App:Est:Pot:Test:1} and \eqref{App:Est:Pot:Test:2}, and using Hölder's and Young's inequality, we get
    \begin{align}
        \label{EST:FPGP:LP}
        \norm{(F_1^\prime(u_k),G_1^\prime(v_k))}_{\mathcal{L}^p} \leq C\big(1 + \norm{(f,g)}_{\mathcal{L}^p}\big)
    \end{align}
    for some constant $C = C(p) > 0$ independent of $k$. For more details, we refer to the proof of \cite[Proposition 6.1]{Knopf2024a} since similar estimates have been carried out there. Finally, by means of Fatou's lemma, we end up with
    \begin{align*}
        \norm{(F_1^\prime(u), G_1^\prime(v))}_{\mathcal{L}^p} \leq C\big( 1 + \norm{(f,g)}_{\mathcal{L}^p}\big),
    \end{align*}
    and and application of elliptic regularity theory for bulk-surface systems readily implies \eqref{App:Est:Pot:L^p}.

    \textit{The case $p\in(2,\infty)$ and $K = 0$.} First, we consider $2 < p \leq 4$. Performing the same testing procedure as in the case $K\in(0,\infty)$ and arguing again with the convexity of $F_1$ and $G_1$, we deduce
    \begin{align}\label{App:Est:Pot:Test:0}
        \begin{split}
            &\intO \abs{F_1^\prime(u_k)}^p\dx + \intG \abs{G_1^\prime(v_k)}^p\dG \\
            &\quad \leq \intO f F_1^\prime(u_k)\abs{F_1^\prime(u_k)}^{p-2}\dx + \intG gG_1^\prime(v_k)\abs{G_1^\prime(v_k)}^{p-2} \\
            &\qquad - \intG \deln u\big(\alpha G_1^\prime(v_k)\abs{G_1^\prime(v_k)}^{p-2} - F_1^\prime(u_k)\abs{F_1^\prime(u_k)}^{p-2}\big)\dG.
        \end{split}
    \end{align}
    We note that this equation differs from \eqref{App:Est:Pot:Test} only in the last term on the right-hand side, which now contains the normal derivative $\deln u$. Therefore, it has to be treated differently. First, due the regularity $u\in H^2(\Om)$, the trace embedding $H^1(\Om)\emb L^4(\Ga)$ implies that $\deln u\in L^4(\Ga)$. Then, exploiting the trace relation $u_k = \alpha v_k$ a.e.~on $\Ga$ in combination with the domination property \eqref{Domination:Prime}, we obtain 
    \begin{align*}
        &-\intG \deln u \big(\alpha G_1^\prime(v_k)\abs{G_1^\prime(v_k)}^{p-2} - F_1^\prime(u_k)\abs{F_1^\prime(u_k)}^{p-2}\big)\dG \\
        &\quad = -\intG \deln u\big(\alpha G_1^\prime(v_k)\abs{G_1^\prime(v_k)}^{p-2} - F_1^\prime(\alpha v_k)\abs{F_1^\prime(\alpha v_k)}^{p-2}\big)\dG \\
        &\quad\leq \intG \abs{\deln u}\big((1 + 2^{p-2}\kappa_1^{p-1})\abs{G_1^\prime(v_k)}^{p-1} + 2^{p-2}\kappa_2^{p-1}\big)\dG.
    \end{align*}
    Hence, in a similar manner to the case $K\in(0,\infty)$ we can derive the estimate
    \begin{align*}
        \norm{(F_1^\prime(u_k), G_1^\prime(v_k))}_{\mathcal{L}^p} \leq C\big( 1 + \norm{(f,g)}_{\mathcal{L}^p} + \norm{\deln u}_{L^p(\Ga)}\big),
    \end{align*}
    where the constant $C = C(p) > 0$ is independent of $k$.
    This proves the desired claim for $p\in(2,4]$ by means of Fatou's lemma.
    To establish the desired result also in the case $p\in(4,\infty)$, we employ elliptic regularity theory (see \cite[Proposition A.1]{Knopf2024a}) to infer $(u,v)\in \mathcal{W}^{2,4}$. Then, using again the trace theorem, we even have $\deln u\in W^{3/4,4}(\Ga)\emb L^\infty(\Ga)$. This allows us to choose $p\in(4,\infty)$ in our previous calculations, which verifies 
    \begin{align*}
        \norm{(F_1^\prime(u),G_1^\prime(v))}_{\mathcal{L}^p} \leq C\big(1 + \norm{(f,g)}_{\mathcal{L}^p} + \norm{\deln u}_{L^p(\Ga)}\big)
    \end{align*}
    in all cases $p\in(2,\infty)$. Applying once more \cite[Proposition A.1]{Knopf2024a}, we readily deduce \eqref{App:Est:Pot:L^p}.

    In summary, we conclude that \eqref{App:Est:Pot:L^p} is verified for all choices $K\in[0,\infty)$ as long as $p\in [2,\infty)$.

    We still have to prove \eqref{App:Est:Pot:L^p} also for $p = \infty$. To this end, we now assume that $(f,g)\in\mathcal{L}^\infty$. 
    Then, our previous calculations show that $(F_1^\prime(u), G_1^\prime(v))\in \mathcal{L}^r$ for any $r\in(2,\infty)$, and thus, again by elliptic regularity theory, we have $(u,v)\in \mathcal{W}^{2,r}$ for any $r\in(2,\infty)$. In particular, this entails that $\deln u\in L^\infty(\Ga)$. First we perform some preliminary computations. Fix an arbitrary $r\in(2,\infty)$. Testing equation \eqref{BSE:2} 
    with $G_1^\prime(v_k)\abs{G_1^\prime(v_k)}^{r-2}$ and proceeding similarly as in the case $p\in(2,\infty)$ above, we see that
    \begin{align*}
        \norm{G_1^\prime(v_k)}_{L^r(\Ga)}^r 
        \leq \big(1 + \norm{g}_{L^r(\Ga)} + \gamma(K)\norm{\deln u}_{L^r(\Ga)}\big)\norm{G_1^\prime(v_k)}_{L^r(\Ga)}^{r-1}.
    \end{align*}
    Since $g,\deln u\in L^\infty(\Gamma)$, this directly implies
    \begin{align}\label{App:Est:Pot:G:p}
        \norm{G_1^\prime(v_k)}_{L^r(\Ga)} 
        \leq (1+\abs{\Ga})\big(1 + \norm{g}_{L^\infty(\Ga)} 
        + \gamma(K)\norm{\deln u}_{L^\infty(\Ga)}\big) 
        <\infty.
    \end{align}
    We thus conclude that
    \begin{align}\label{App:Est:Pot:G:p:infty}
        \norm{G_1^\prime(v_k)}_{L^\infty(\Ga)} 
        \leq (1+\abs{\Ga}) \big( 1 + \norm{(f,g)}_{\mathcal{L}^\infty}
        + \gamma(K)\norm{\deln u}_{L^\infty(\Ga)} \big) < \infty.
    \end{align}
    Since $\abs{v_k} < 1$ and $\abs{v} < 1$ a.e.~on $\Ga$, we have
    \begin{align*}
        G_1^\prime(v_k) \rightarrow G_1^\prime(v) \qquad\text{a.e.~on~}\Ga \ \text{as~}k\rightarrow\infty.
    \end{align*}
    In combination with \eqref{App:Est:Pot:G:p:infty} this proves that
    \begin{align}
        \label{EST:GPR}
        \norm{G_1^\prime(v)}_{L^\infty(\Ga)} 
        \leq (1+\abs{\Ga})\big( 1 + \norm{(f,g)}_{\mathcal{L}^\infty}
        + \gamma(K)\norm{\deln u}_{L^\infty(\Ga)}\big) 
        < \infty.
    \end{align}
    To prove the corresponding estimate for $F_1^\prime(u)$, we again have to distinguish two cases: $K\in(0,\infty)$ and $K = 0$.

    \textit{The case $p=\infty$ and $K\in(0,\infty)$.} 
    We first obtain for all $r\in(2,\infty)$ the estimate
    \begin{align}\label{App:Est:Pot:Test:p}
        \begin{split}
            &\intO \abs{F_1^\prime(u_k)}^r\dx + \intG \abs{G_1^\prime(v_k)}^r \dG \\
            &\quad \le \intO f F_1^\prime(u_k)\abs{F_1^\prime(u_k)}^{r-2}\dx + \intG gG_1^\prime(v_k)\abs{G_1^\prime(v_k)}^{r-2} \\
            &\qquad - \sigma(K) \intG (\alpha v - u)\big(\alpha G_1^\prime(v_k)\abs{G_1^\prime(v_k)}^{r-2} - F_1^\prime(\alpha v_k)\abs{F_1^\prime(\alpha v_k)}^{r-2}\big)\dG.
        \end{split}
    \end{align}
    by combining \eqref{App:Est:Pot:Test} and \eqref{App:Est:Pot:Test:1}.
    Using the domination property \eqref{Domination:Prime}, the last term on the right-hand side of \eqref{App:Est:Pot:Test:p} can be estimated as
    \begin{align*}
        &- \sigma(K) \intG (\alpha v - u)\big(\alpha G_1^\prime(v_k)\abs{G_1^\prime(v_k)}^{r-2} - F_1^\prime(\alpha v_k)\abs{F_1^\prime(\alpha v_k)}^{r-2}\big)\dG \\
        &\quad\leq 2\sigma(K)\abs{\Ga}\norm{G_1^\prime(v_k)}_{L^\infty(\Ga)}^{r-1} + 2^{r-1}\kappa_1^{r-1}\sigma(K)\abs{\Ga}\norm{G_1^\prime(v_k)}_{L^\infty(\Ga)}^{r-1} + 2^{r-1}\kappa_2^{r-1}\sigma(K)\abs{\Ga}.
    \end{align*}
    Here, we also used that $\abs{u}< 1$ a.e.~in $\Omega$ and $\abs{v}< 1$ a.e.~on $\Gamma$.
    Consequently, by means of Hölder's and Young's inequalities, we conclude that
    \begin{align*}
        &\norm{F_1^\prime(u_k)}_{L^r(\Om)}^r + \norm{G_1^\prime(v_k)}_{L^r(\Ga)}^r 
        \\
        &\quad\leq \norm{f}_{L^r(\Om)} \norm{F_1^\prime(u_k)}_{L^r(\Om)}^{r-1} + \norm{g}_{L^r(\Ga)}\norm{G_1^\prime(v_k)}_{L^r(\Ga)}^{r-1}
        \\
        &\qquad+ 2\sigma(K)\abs{\Ga}
        \big(1 + 2^{r-2}\kappa_1^{r-1}\big)
        \norm{G_1^\prime(v_k)}_{L^\infty(\Ga)}^{r-1}
        + \sigma(K)\abs{\Ga} 2^{r-1}\kappa_2^{r-1} 
        \\
        &\quad\leq \frac12\norm{F_1^\prime(u_k)}_{L^r(\Om)}^r + \frac12\norm{G_1^\prime(v_k)}_{L^r(\Ga)}^r 
        \\
        &\qquad
        + \frac1r\left(\frac{2(r-1)}{r}\right)^{r-1}
        \big( \abs{\Omega} + \abs{\Gamma}\big) \,\norm{(f,g)}_{\mathcal{L}^\infty}^r 
        \\
        &\qquad+ 2\sigma(K)\abs{\Ga}
        \big(1 + 2^{r-2}\kappa_1^{r-1}\big)
        (1+\abs{\Gamma})^{r-1} 
        (1+\norm{(f,g)}_{\mathcal{L}^\infty})^r 
        + \sigma(K)\abs{\Ga} 2^{r-1}\kappa_2^{r-1} 
    \end{align*}
    Thus, after absorption of the first two terms on the right-hand side by the left-hand side, we take the $r$-th root on both sides of the resulting inequality to deduce
    \begin{align*}
        \norm{F_1^\prime(u_k)}_{L^r(\Om)} + \norm{G_1^\prime(v_k)}_{L^r(\Ga)} \leq C\big(1 + \norm{(f,g)}_{\mathcal{L}^\infty} \big)
    \end{align*}
    for a constant $C > 0$ that is independent of $r$. We thus infer
    \begin{align*}
        \norm{F_1^\prime(u_k)}_{L^\infty(\Om)} + \norm{G_1^\prime(v_k)}_{L^\infty(\Ga)} 
        \leq C\big(1 + \norm{(f,g)}_{\mathcal{L}^\infty} \big).
    \end{align*}
    By means of similar arguments as already used before, we conclude that
    \begin{align*}
        \norm{(F_1^\prime(u), G_1^\prime(v))}_{\mathcal{L}^\infty} 
        \leq C\big(1 + \norm{(f,g)}_{\mathcal{L}^\infty} \big),
    \end{align*}
    and \eqref{App:Est:Pot:L^p} follows from another application of \cite[Proposition A.1]{Knopf2024a}.

    \textit{The case $p=\infty$ and $K = 0$.} We test equation \eqref{BSE:1} once more with $F_1^\prime(u_k)\abs{F_1^\prime(v_k)}^{r-2}$ for any $r\in(2,\infty)$ and integrate the resulting equation. After integrating by parts, we find that
    \begin{align*}
        \intO \abs{F_1^\prime(u_k)}^r\dx \leq \intO \abs{f}\abs{F_1^\prime(u_k)}^{r-1}\dx - \intG \deln u \, F_1^\prime(\alpha v_k)\abs{F_1^\prime(\alpha v_k)}^{r-2}\dG.
    \end{align*}
    Using Hölder's and Young's inequalities, we have
    \begin{align}
        \label{EST:PK:1}
        \begin{split}
            \norm{F_1^\prime(u_k)}_{L^r(\Om)}^r 
            &\leq \norm{f}_{L^r(\Om)}\norm{F_1^\prime(u_k)}_{L^r(\Om)}^{r-1}      + \norm{\deln u}_{L^r(\Gamma)} 
                \norm{F_1^\prime(\alpha v_k)}_{L^r(\Gamma)}^{r-1}
            \\
            &\leq \frac12\norm{F_1^\prime(u_k)}_{L^r(\Om)}^r 
                + \frac1r\left(\frac{2(r-1)}{r}\right)^{r-1}
                \norm{f}_{L^r(\Om)}^r 
            \\
            &\quad + \frac 1r \norm{\deln u}_{L^r(\Gamma)}^r
                + \frac{r-1}{r} \norm{F_1^\prime(\alpha v_k)}_{L^r(\Gamma)}^r.
        \end{split}
    \end{align}
    Thus, we deduce that
    \begin{align}
        \label{EST:PK:2}
        \begin{split}
            \norm{F_1^\prime(u_k)}_{L^r(\Om)}^r 
            &\leq \frac2r\left(\frac{2(r-1)}{r}\right)^{r-1}\norm{f}_{L^r(\Om)}^r 
                + \frac2r\norm{\deln u}_{L^r(\Gamma)}^r
                + \frac{2(r-1)}{r}\norm{F_1^\prime(\alpha v_k)}_{L^r(\Gamma)}^r
            \\
            &\leq \frac2r\left(\frac{2(r-1)}{r}\right)^{r-1}(1+\abs{\Om}) \norm{f}_{L^\infty(\Om)}^r 
                + \frac2r(1+\abs{\Gamma})  \norm{\deln u}_{L^\infty(\Gamma)}^r
            \\
            &\quad
                + \frac{2(r-1)}{r}(1+\abs{\Gamma}) 
                \norm{F_1^\prime(\alpha v_k)}_{L^\infty(\Gamma)}^r
        \end{split}
    \end{align}
    Due to the domination property \eqref{Domination:Prime} and estimate \eqref{App:Est:Pot:G:p:infty}, we infer that
    \begin{align}
        \label{EST:PK:3}
        \begin{split}
        \norm{F_1^\prime(\alpha v_k)}_{L^\infty(\Gamma)}
        &\leq \kappa_1 \norm{G_1^\prime(v_k)}_{L^\infty(\Gamma)}
            + \kappa_2
        \\
        &\leq C\big( 1 + \norm{(f,g)}_{\mathcal{L}^\infty}
            + \norm{\deln u}_{L^\infty(\Ga)} \big).
        \end{split}
    \end{align} 
    Hence, after taking the $r$-th root on both sides of \eqref{EST:PK:2}, we use \eqref{EST:PK:3} to end up with
    \begin{align*}
        \norm{F_1^\prime(u_k)}_{L^r(\Om)} 
        \leq C\big(1 
            + \norm{(f,g)}_{\mathcal{L}^\infty}
            + \norm{\deln u}_{L^\infty(\Ga)}
        \big)
        < \infty,
    \end{align*}
    where the constant $C>0$ is independent of $r$.
    Using Fatou's lemma, we get
    \begin{align*}
        \norm{F_1^\prime(u)}_{L^r(\Om)} 
        \leq C\big(1 
            + \norm{(f,g)}_{\mathcal{L}^r}
            + \norm{\deln u}_{L^r(\Ga)}
        \big)
        < \infty.
    \end{align*}
    Then, sending $r\rightarrow\infty$, we conclude that
    \begin{align*}
        \norm{F_1^\prime(u)}_{L^\infty(\Om)} 
        \leq C\big(1 
            + \norm{(f,g)}_{\mathcal{L}^\infty}
            + \norm{\deln u}_{L^\infty(\Ga)}
        \big)
        < \infty,
    \end{align*}
    and in combination with \eqref{EST:GPR} and \cite[Proposition A.1]{Knopf2024a}, this verifies \eqref{App:Est:Pot:L^p} in the case $p=\infty$ and $K=0$.
    
    In summary, we have shown that $(F_1^\prime(u), G_1^\prime(v))\in\mathcal{L}^p$ along with estimate \eqref{App:Est:Pot:L^p} for all $p\in [2,\infty]$ in all cases $K\in[0,\infty)$. 
    We further observe that $u\in H^2(\Omega)\emb C(\overline{\Omega})$ and $v \in H^2(\Gamma) \emb C(\Gamma)$. Hence, if $p=\infty$, the separation property \eqref{BSE:SEPPROP} follows from Assumption~\ref{Assumption:Pot:Sing} combined with the regularity $(F_1^\prime(u), G_1^\prime(v)) \in \mathcal{L}^\infty$.
    The proof is complete.
\end{proof}

\medskip 

Based on the above results, we now derive improved estimates on the principal part of the differential in System~\eqref{BSE}. This will be the main ingredient in the proof of Theorem~\ref{THM:TIMEREG}.

\begin{corollary}\label{COR:BSE:OP}
    Let $(f,g)\in\mathcal{H}^1$ and let $(u,v)\in\mathcal{H}^2\cap\mathcal{H}^1_{K,\alpha}$ be the corresponding solution of System~\eqref{BSE} given by Proposition~\ref{App:Prop:Existence}. Then, there exists a constant $C > 0$, depending only on $\Om$, $K$, $\alpha$, $F$ and $G$ such that the following inequalities hold:
    \begin{enumerate}[label=\textnormal{(\alph*)},topsep=0ex,leftmargin=*]
        \item If $K = 0$, then
        \begin{align*}
            \begin{split}
            &\norm{(-\Lap u, -\Lapg v + \alpha\deln u)}_{\mathcal{L}^2}^2 \\
            &\quad\leq C\big(1 + \norm{(f,g)}_{\mathcal{H}^1}\big)\norm{\Grad u}_{L^2(\Om)}^{1/2}
                \norm{\Grad u}_{H^1(\Om)}^{1/2} 
            + \norm{(\Grad f,\Gradg g)}_{\mathcal{L}^2}\norm{(\Grad u,\Gradg v)}_{\mathcal{L}^2}.
            \end{split}
        \end{align*}
        \item If $K\in(0,\infty)$, then
        \begin{align*}
            \norm{(-\Lap u, -\Lapg v + \alpha\deln u)}_{\mathcal{L}^2}^2 \leq C\big(1 + \norm{(f,g)}_{\mathcal{H}^1}\big) + \norm{(\Grad f,\Gradg g)}_{\mathcal{L}^2}\norm{(\Grad u,\Gradg v)}_{\mathcal{L}^2}.
        \end{align*}
    \end{enumerate}
\end{corollary}

\begin{proof}
    For any $k\in\N$, let $u_k$ and $v_k$ again denote the truncation of $u$ and $v$, respectively, as in the proof of Proposition~\ref{App:Proposition:Pot:L^p}. We start by multiplying \eqref{BSE:1} by $-\Lap u$ and \eqref{BSE:2} by $-\Lapg v + \alpha\deln u$. Then, integrating the resulting equations over $\Om$ and $\Ga$, respectively, and adding them, we have
    \begin{align}\label{EST:BS-OP:1}
        \begin{split}
            &\intO \abs{\Lap u}^2\dx + \intG \abs{-\Lapg v + \alpha\deln u}^2 \dG 
            \\
            &\quad = \intO f(-\Lap u)\dx + \intG g(-\Lapg v + \alpha\deln u)\dG  
            \\
            &\qquad 
            - \intO F_1^\prime(u_k)(-\Lap u) 
            - \intO (F_1^\prime(u) - F_1^\prime(u_k))(-\Lap u)\dx
            \\
            &\qquad
            - \intG G_1^\prime(v_k)(-\Lapg v + \alpha\deln u)\dG
            - \intG (G_1^\prime(v) - G_1^\prime(v_k))(-\Lapg v + \alpha\deln u)\dG.
        \end{split}
    \end{align}
    Arguing as in the proof of Proposition~\ref{App:Proposition:Pot:L^p}, we deduce that
    \begin{align*}
        &-\intO F_1^\prime(u_k)(-\Lap u)\dx - \intG G_1^\prime(v_k)(-\Lapg v + \alpha\deln u)\dG 
        \\
        &\qquad = -\intO \Grad u\cdot\Grad u_k F_1^{\prime\prime}(u_k)\dx - \intG \Grad v\cdot\Grad v_k G_1^{\prime\prime}(v_k)\dG - \intG \big(\alpha G_1^\prime(v_k) - F_1^\prime(u_k)\big)\deln u\dG 
        \\
        &\quad\leq \intG  \big(\alpha G_1^\prime(v_k) - F_1^\prime(u_k)\big)\deln u\dG.
    \end{align*}
    After integrating by parts in the first two terms on the right-hand side of \eqref{EST:BS-OP:1}, we use Hölder's inequality to obtain
    \begin{align}\label{EST:BS-OP:2}
        \begin{split}
            &\intO \abs{\Lap u}^2\dx + \intG \abs{-\Lapg v + \alpha\deln u}^2\dG 
            \\
            &\quad\leq \norm{(\Grad f,\Gradg g)}_{\mathcal{L}^2}\norm{(\Grad u,\Gradg v)}_{\mathcal{L}^2} 
                + \intG \big(\alpha g - f)\deln u\dG 
            \\
            &\qquad - \intO (F_1^\prime(u) - F_1^\prime(u_k))(-\Lap u)\dx 
                - \intG (G_1^\prime(v) - G_1^\prime(v_k))(-\Lapg v + \alpha\deln u)\dG 
            \\
            &\qquad + \intG \big(\alpha G_1^\prime(v_k) - F_1^\prime(u_k)\big)\deln u\dG.
        \end{split}
    \end{align}
    We point out that the third and the fourth summand on the right-hand side of this inequality vanish as $k\rightarrow\infty$. We now handle the cases $K=0$ and $K\in (0,\infty)$ separately.
    
    If $K = 0$, we make use of the trace relation $u_k = \alpha v_k$ a.e.~on $\Ga$, Lemma~\ref{Prelim:Lemma:Interpol} and Proposition~\ref{App:Prop:Existence} to infer that
    \begin{align}
        \label{EST:BS-OP:5}
        \begin{split}
        &-\intG \big(\alpha G_1^\prime(v_k) - F_1^\prime(u_k)\big)\deln u\dG 
        \\ 
        &\quad\leq C\norm{\big(F_1^\prime(u_k),G_1^\prime(v_k)\big)}_{\mathcal{L}^2}\norm{\Grad u}_{L^2(\Ga)} 
        \\ 
        &\quad\leq C\big( 1 + \norm{\big(F_1^\prime(u),G_1^\prime(v)\big)}_{\mathcal{L}^2}\big)\norm{\Grad u}_{L^2(\Ga)} 
        \\
        &\quad\leq C\big(1 + \norm{(f,g)}_{\mathcal{L}^2}\big)\norm{\Grad u}_{L^2(\Om)}^{1/2}\norm{\Grad u}_{H^1(\Om)}^{1/2}
        \end{split}
    \end{align}
    if $k$ is sufficiently large.
    Analogously, again using Lemma~\ref{Prelim:Lemma:Interpol}, we have
    \begin{align}
        \label{EST:BS-OP:6}
        \intG \big(\alpha g - f\big)\deln u\dG \leq \norm{(f,g)}_{\mathcal{H}^1}\norm{\Grad u}_{L^2(\Om)}^{1/2}\norm{\Grad u}_{H^1(\Om)}^{1/2}.
    \end{align}
    Hence, the claim follows by combining \eqref{EST:BS-OP:2} with \eqref{EST:BS-OP:5} and \eqref{EST:BS-OP:6}, and passing to the limit $k\to\infty$.

    If $K\in(0,\infty)$, we first use \eqref{BSE:3} to reformulate the normal derivative $\deln u$. Then, by Lemma~\ref{Prelim:Lemma:Interpol} and Proposition~\ref{App:Est:H^2+L^2} and proceeding similarly as in the case $K=0$, we infer that
    \begin{align}
        \label{EST:BS-OP:3}
        \begin{split}
            &-\intG \big(\alpha G_1^\prime(v_k) - F_1^\prime(u_k)\big)\deln u \dG 
            = - \frac{1}{K} \intG (\alpha v - u)\big(\alpha G_1^\prime(v_k) - F_1^\prime(u_k)\big)\dG  
            \\
            &\quad\leq C\norm{\big(F_1^\prime(u_k),G_1^\prime(v_k)\big)}_{\mathcal{L}^2}
            \leq C\big(1 + \norm{(f,g)}_{\mathcal{L}^2}\big)
        \end{split}
    \end{align}
    if $k$ is sufficiently large.
    Due to the trace embedding $H^1(\Omega)\emb L^2(\Gamma)$, we further have
    \begin{align}
        \label{EST:BS-OP:4}
        \intG \big(\alpha g - f\big)\deln u\dG = -\frac1K\intG (\alpha v - u)\big(\alpha g - f\big)\dG\leq C\norm{(f,g)}_{\mathcal{H}^1}.
    \end{align}
    Thus, the claim follows by combining \eqref{EST:BS-OP:2} with \eqref{EST:BS-OP:3} and \eqref{EST:BS-OP:4}, and passing to the limit $k\to\infty$.
\end{proof}

\medskip

Lastly, invoking the $L^p$-estimates on the potentials established in Proposition~\ref{App:Proposition:Pot:L^p}, we are able to prove higher Sobolev estimates.
\begin{proposition}\label{App:Proposition:Sol:L^p+H^k}
    Let $(f,g)\in\mathcal{H}^1$ and let $(u,v)\in\mathcal{H}^2\cap\mathcal{H}^1_{K,\alpha}$ be the corresponding solution of System~\eqref{BSE}.
    Let $p\in [2,6]$ if $d=3$ and for any $p\in [2,\infty)$ if $d=2$.
    Then there exists a constant $C(p)> 0$ such that
    \begin{align}\label{App:Est:Sol:L^p}
        \norm{(u,v)}_{\mathcal{W}^{2,p}} + \norm{(F_1^\prime(u),G_1^\prime(v))}_{\mathcal{L}^p} \leq C(p) \big(1 + \norm{(f,g)}_{\mathcal{H}^1}\big).
    \end{align}
    Moreover, if $\Omega$ is of class $C^{k+2}$ and $(f,g)\in\mathcal{H}^k\cap\mathcal{L}^\infty$ for $k\in\{1,2\}$, we additionally have $(u,v)\in\mathcal{H}^{k+2}$ and there exists a constant $C(k)>0$ such that
    \begin{align}\label{App:Est:Sol:H^k}
        \norm{(u,v)}_{\mathcal{H}^{k+2}} \leq C(k)\big( 1 + \norm{(f,g)}_{\mathcal{H}^k}\big).
    \end{align}
\end{proposition}

\medskip

\begin{proof}
    Since the estimate \eqref{App:Est:Sol:L^p} for $p=2$ follows directly from Proposition~\ref{App:Prop:Existence}, we only consider $p\in (2,6]$ if $d=3$ or $p\in (2,\infty)$ if $d=2$. 
    Let $(u,v)\in\mathcal{H}^2\cap\mathcal{H}^1_{K,\alpha}$ be the solution of System~\eqref{BSE} with source terms $(f,g)\in\mathcal{H}^1.$
    Now, Proposition~\ref{App:Proposition:Pot:L^p} yields $(u,v)\in\mathcal{W}^{2,p}$ and $(F_1^\prime(u), G_1^\prime(v)) \in \mathcal{L}^p$ along with the existence of a constant $C(p) > 0$ such that
    \begin{align}\label{EST:W2P:1}
        \norm{(u,v)}_{\mathcal{W}^{2,p}} + \norm{(F_1^\prime(u), G_1^\prime(v))}_{\mathcal{L}^p} \leq C(p)\big(1 + \norm{(f,g)}_{\mathcal{L}^p} + \gamma(K)\norm{\deln u}_{L^p(\Ga)}\big).
    \end{align}
    Recalling the trace embedding $W^{2/p,p}(\Omega)\emb L^p(\Gamma)$ as well as the compact Sobolev embeddings $W^{2,p}(\Omega)\emb W^{1+2/p,p}(\Omega) \emb L^\infty(\Omega)$, we use Ehrling's lemma to deduce that
    \begin{align*}
        \norm{\deln u}_{L^p(\Gamma)}
        \le C(p) \norm{u}_{W^{1+2/p,p}(\Omega)}
        \le \varepsilon \norm{u}_{W^{2,p}(\Omega)} 
            + C(p,\varepsilon) \norm{u}_{L^\infty(\Omega)}
    \end{align*}
    for any $\varepsilon >0$ and a constant $C(p,\varepsilon) > 0$ depending on $\varepsilon $.
    Hence, recalling that $\abs{u}\le 1$ a.e.~in $\Omega$, we infer from estimate \eqref{App:Est:Pot:L^p} of
    Proposition~\ref{App:Proposition:Pot:L^p} that
    \begin{align}
        \label{EST:W2P:2}
        \begin{split}
        \norm{(u,v)}_{\mathcal{W}^{2,p}} &\leq C\big(1 + \norm{(f,g)}_{\mathcal{L}^p} + \gamma(K)\norm{\deln u}_{L^p(\Ga)}\big) \\
        &\leq C(p)\big(1 + C(p,\varepsilon) + \norm{(f,g)}_{\mathcal{L}^p}\big) + C(p)\varepsilon\norm{(u,v)}_{\mathcal{W}^{2,p}}.
        \end{split}
    \end{align}
    Combining \eqref{EST:W2P:1} and \eqref{EST:W2P:2}, and choosing a sufficiently small $\varepsilon>0$, we conclude that \eqref{App:Est:Sol:L^p} holds.

    To prove the second assertion for any $k\in\{1,2\}$, we first recall that, according to Proposition~\ref{App:Proposition:Pot:L^p}, the regularity $(f,g)\in \mathcal{L}^\infty$ already implies that $u$ and $v$ satisfy the separation property \eqref{BSE:SEPPROP}. As we also have $(u,v)\in \mathcal{H}^2$, we infer that $(F_1^\prime(u), G_1^\prime(v)) \in \mathcal{H}^2$ with
    \begin{align}
        \label{EST:FPGP:H2}
        \norm{(F_1^\prime(u), G_1^\prime(v))}_{\mathcal{H}^2} 
        \le C \norm{(u,v)}_{\mathcal{H}^2} 
        \le C \norm{(f,g)}_{\mathcal{L}^2}.
    \end{align}    
    In view of $(f,g)\in\mathcal{H}^k$, regularity theory for elliptic systems with bulk-surface coupling (see \cite[Proposition~A.1]{Knopf2024a}) yields $(u,v) \in\mathcal{H}^{k+2}$ along with the estimate
    \begin{align*}
        \norm{(u,v)}_{\mathcal{H}^{k+2}}
        &\le C(k) \norm{(f,g)}_{\mathcal{H}^k} 
            + C(k) \norm{(F_1^\prime(u), G_1^\prime(v))}_{\mathcal{H}^2}.
    \end{align*}
    In view of \eqref{EST:FPGP:H2}, this verifies \eqref{App:Est:Sol:H^k} for any $k\in\{1,2\}$ and thus, the proof is complete. 
\end{proof}

\begin{remark} \label{REM:YOS}
    It is straightforward to check that the statements of Proposition~\ref{App:Prop:Existence}, Proposition~\ref{App:Prop:ContDep}, Proposition~\ref{App:Proposition:Pot:L^p} except for the separation property \eqref{BSE:SEPPROP}, Corollary~\ref{COR:BSE:OP} and Proposition~\ref{App:Proposition:Sol:L^p+H^k} remain true when the potentials $F$ and $G$ are replaced by their respective Moreau--Yosida approximations $F_\lambda$ and $G_\lambda$ with $\lambda>0$ (cf.~Section~\ref{SECT:SUBDIFF}). 
    Although the separation property separation property \eqref{BSE:SEPPROP} cannot be obtained when using the Moreau--Yosida approximations, the corresponding proofs can be carried out similarly. As the potentials $F_\lambda$ and $G_\lambda$ are \mbox{regular}, the proofs do not require any truncation of the functions $u$ and $v$.
    It is crucial that for $F_\lambda$ and $G_\lambda$, the constants in the estimates \eqref{App:Est:Pot:L^p} and \eqref{App:Est:Sol:L^p} are still independent of $\lambda$.
\end{remark}

\section{Existence of weak solutions}\label{SECT:EXWS}

This section is devoted to the proofs of our main results regarding weak solutions, namely Theorem~\ref{THEOREM:EOWS}, Theorem~\ref{THM:TIMEREG} and Theorem~\ref{THEOREM:UNIQUE:SING}.
\begin{remark}\label{REMARK:CONST}
In the remainder of this paper (unless stated otherwise), the letter $C$ will denote generic positive constants that may depend on the prescribed initial data and velocity fields, the domain $\Omega$, the final time $T$, the parameters $K$, $L$, $\alpha$, $\beta$, the constants introduced in \eqref{Assumption:Mobility} and the potentials $F$ and $G$ along with the constants introduced in \eqref{Assumption:Pot:Convexity} and \eqref{Domination:Prime}. The exact value of $C$ may vary throughout the computations. On the other hand, we use $C_0$ to denote generic positive constants, which are independent of the final time $T$ as well as the velocity fields but otherwise depend on the same data as $C$.
\end{remark}

\begin{proof}[Proof of Theorem~\ref{THEOREM:EOWS}]
    The existence of a global-in-time weak solution was proven in \cite[Theorem~3.4]{Knopf2024a} by a compactness method relying on the Yosida regularization of the singular potentials. Throughout this proof, the stronger assumption \eqref{ASS:VEL:WEAK:OLD} was needed as it cannot be ensured that the approximate phase-fields have the property \eqref{PROP:CONF} (or any similar property).
    Nonetheless, as we already know that the weak solution $(\phi,\psi)$ fulfills \eqref{PROP:CONF}, the regularity assumption on the velocity fields can be weakened a posteriori. As most of this can be done by proceeding similarly as in the proof of \cite[Theorem~3.4]{Knopf2024a}, we only outline the most important steps.
    
    Let $(\boldsymbol{v},\boldsymbol{w})\in L^2(0,T;\mathbf{L}^2_\Div(\Om)\times\mathbf{L}^2_\tau(\Ga))$ be the velocity fields. Consider a sequence
    $\{(\boldsymbol{v}_k,\boldsymbol{w}_k)\}_{k\in\N}$ satisfying \eqref{ASS:VEL:WEAK:OLD} such that
    \begin{align}\label{Est:Thm:Eows}
        (\boldsymbol{v}_k,\boldsymbol{w}_k)\rightarrow (\boldsymbol{v},\boldsymbol{w}) \qquad\text{strongly in~} L^2(0,T;\mathcal{L}^2) \text{~as~}k\to \infty.
    \end{align}
    For any $k\in\N$, according to \cite[Theorem~3.4]{Knopf2024a}, there exists an approximating weak solution $(\phi_k,\psi_k,\mu_k,\theta_k)$ to \eqref{EQ:SYSTEM} in the sense of Definition~\ref{DEF:SING:WS} with $(\boldsymbol{v},\boldsymbol{w})$ replaced by $(\boldsymbol{v}_k,\boldsymbol{w}_k)$. Using $\abs{\phi_k} < 1$ a.e.~in $Q$ and $\abs{\psi_k} < 1$ a.e.~on $\Sigma$, we can follow the line of argument in the proof of \cite[Theorem~3.4]{Knopf2024a} to derive the uniform estimate
    \begin{align*}
        \begin{split}
        &\norm{(\phi_k,\psi_k)}_{L^\infty(0,T;\mathcal{H}^1)}
        + \norm{(\phi_k,\psi_k)}_{H^1(0,T;(\mathcal{H}^1_{L,\beta})')}
        \\
        &\quad
        + \norm{(\mu_k,\theta_k)}_{L^2(0,T;\mathcal{H}^1)}
        + \big\|\big(F_1'(\phi_k),G_1'(\psi_k)\big)\big\|_{L^2(0,T;\mathcal{L}^2)}
        \le C,
        \end{split}
    \end{align*}
    where $C$ denotes a non-negative constant that does not depend on $k$.
    Hence, by means of the Banach--Alaoglu theorem and the Aubin--Lions--Simon lemma, we infer that there exist functions $\phi,\psi,\mu,\theta,f,g$ such that
    \begin{alignat*}{2}
        (\phi_k,\psi_k) &\to (\phi,\psi)
        &&\quad\text{weakly-$\ast$ in $L^\infty(0,T;\mathcal{H}^1_{K,\alpha})$,
            weakly in $H^1(0,T;(\mathcal{H}^1_{L,\beta})')$}
        \\
        &&&\quad\text{and strongly in $C([0,T];\mathcal{L}^2)$,}
        \\
        (\mu_k,\theta_k) &\to (\mu,\theta)
        &&\quad\text{weakly in $L^2(0,T;\mathcal{H}^1_{L,\beta})$,}
        \\
        \big(F_1'(\phi_k),G_1'(\psi_k)\big) &\to (f,g)
        &&\quad\text{weakly in $L^2(0,T;\mathcal{L}^2)$,}
    \end{alignat*}    
    as $k\to\infty$, up to subsequence extraction.
    In particular, the quadruplet $(\phi,\psi,\mu,\theta)$ has the regularities demanded in \eqref{REG:SING}. Via the weak-strong convergence principle, we further deduce
    \begin{align*}
        \underset{k\to\infty}{\lim}\,\intO F_1'(\phi_k) \phi_k \dx 
        &= \intO f \phi \dx,
        \\
        \underset{k\to\infty}{\lim}\,\intG G_1'(\psi_k) \psi_k \dx 
        &= \intG g \psi \dG.
    \end{align*}
    As $F_1'$ and $G_1'$ are single-valued maximal monotone operators on their effective domains $D(F_1') = D(G_1') = (-1,1)$, we can use the theory of monotone operators (see \cite[Proposition 1.1, p.~42]{Barbu1976} and also \cite[Section~5.2]{Garcke2017}) to conclude
    \begin{equation*}
        \phi\in D(F_1') = (-1,1)
        \quad\text{a.e.~in $Q$}
        \quad\text{and}\quad
        \psi\in D(G_1') = (-1,1)
        \quad\text{a.e.~on $\Sigma$}
    \end{equation*}
    as well as
    \begin{equation*}
        f = F_1'(\phi)
        \quad\text{a.e.~in $Q$}
        \quad\text{and}\quad
        g = G_1'(\psi)
        \quad\text{a.e.~on $\Sigma$}.
    \end{equation*}
    By the aforementioned convergences, it is now straightforward to pass to the limit in the weak formulation \eqref{SING:WF} and to verify the mass conservation law \eqref{MCL:SING} and the energy inequality \eqref{WEDL:SING}. This proves that the quadruplet $(\phi,\psi,\mu,\theta)$ is a weak solution to the velocity fields $(\boldsymbol{v},\boldsymbol{w})$ in the sense of Definition~\ref{DEF:SING:WS}. As a last step, we have to show that if $L\in(0,\infty]$, any weak solution satisfies the energy equality. To this end, we consider again the convex part of the free energy $\widetilde E_K$ as defined in Chapter~\ref{SECT:SUBDIFF}. Then, according to Lemma~\ref{Lemma:ConvEnergy:Prop}, $\widetilde E_K$ is proper, convex and lower semi-continuous. Owing to \cite[Lemma~4.1]{Rocca2004} in combination with Proposition~\ref{App:Proposition:Subdiff}, we deduce that $[0,T]\ni t\mapsto\widetilde E_K(\phi(t),\psi(t))$ is absolutely continuous and it holds
	\begin{align*}
		\ddt \widetilde E_K(\phi,\psi) &= \bigang{(\delt\phi,\delt\psi)}{(-\Lap\phi + F_1^\prime(\phi), -\Lapg\psi + G_1^\prime(\psi) + \alpha\deln\phi)}_{\mathcal{H}^1} \\
		&= \bigang{(\delt\phi,\delt\psi)}{(\mu - F_2^\prime(\phi), 			\theta - G_2^\prime(\psi))}_{\mathcal{H}^1}
	\end{align*}
	almost everywhere in $(0,T)$. As a consequence, since $F$ is bounded and $\phi\in C([0,T];\mathcal{L}^2)$, the Lebesgue theorem entails that $\intO F(\phi(\cdot))\dx\in C([0,T];\mathcal{L}^2)$, which, in turn, implies that $\phi\in C([0,T];H^1(\Om))$. Analogous reasoning provides that $\psi\in C([0,T];H^1(\Ga))$. Hence, testing \eqref{WF:PP:SING} with $(\mu,\theta)$ and exploiting the standard chain rule in $L^2(0,T;\mathcal{H}^1)\cap H^1(0,T;(\mathcal{H}^1)^\prime)$, we obtain
    \begin{align*}
        &\ddt E_K(\phi,\psi) + \intO \abs{\Grad\mu}^2\dx + \intG\abs{\Gradg\theta}^2\dG + \chi(L)\intG(\beta\theta-\mu)^2\dG \\
        &\quad - \intO \phi\boldsymbol{v}\cdot\Grad\mu\dx - \intG \psi\boldsymbol{w}\cdot\Gradg\theta\dG = 0,
    \end{align*}
    a.e. on $(0,T)$, from which the energy equality follows.
\end{proof}

Next, as a direct consequence of the results established in Section~\ref{SECT:BSE}, we can prove Theorem~\ref{THM:TIMEREG}.

\begin{proof}[Proof of Theorem~\ref{THM:TIMEREG}.]
    The claim directly follows from combining the inequalities obtained in Corollary~\ref{COR:BSE:OP} with the regularities \eqref{REGPP:SING}-\eqref{REGMT:SING} of the considered weak solution.
\end{proof}

Lastly, we present the proof of the continuous dependence estimate.

\begin{proof}[Proof of Theorem~\ref{THEOREM:UNIQUE:SING}.]
    We consider two weak solutions $(\phi_1,\psi_1,\mu_1,\theta_1)$ and $(\phi_2,\psi_2,\mu_2,\theta_2)$ corresponding to the initial data $(\phi_0^1,\psi_0^1), (\phi_0^2,\psi_0^2)\in\mathcal{H}^1_{K,\alpha}$ and the velocity fields 
    \begin{align*}
        (\boldsymbol{v}_1,\boldsymbol{w}_1) \in L^2(0,T;\mathcal{L}^2_\Div)
        \quad\text{and}\quad
        (\boldsymbol{v}_2,\boldsymbol{w}_2) \in L^2(0,T;\mathbf{L}^3_\Div(\Om)\times\mathbf{L}^{2+\omega}_\tau(\Ga)).
    \end{align*}
    Then we set
    \begin{align*}
        (\phi,\psi,\mu,\theta,\boldsymbol{v},\boldsymbol{w}) \coloneqq (\phi_1 - \phi_2, \psi_1 - \psi_2, \mu_1 - \mu_2, \theta_1 - \theta_2, \boldsymbol{v}_1 - \boldsymbol{v}_2, \boldsymbol{w}_1 - \boldsymbol{w}_2).
    \end{align*}
    Then, repeating the line of argument presented in \cite[Theorem~3.5]{Knopf2024} combined with the one in \cite[Theorem~3.6]{Knopf2024a}, we can show that
    \begin{align}\label{Uniq:DiffIneq}
        \begin{split}
            &\ddt\frac12\norm{(\phi,\psi)}_{L,\beta,\ast}^2 + \frac34\norm{(\phi,\psi)}_{K,\alpha}^2\\
            &\quad \leq C\norm{(\phi,\psi)}_{L,\beta,\ast}^2 +  \intO \big(\phi_1\boldsymbol{v} + \phi\boldsymbol{v}_2\big)\cdot\Grad\mathcal{S}_{L,\beta}^\Om(\phi,\psi)\dx + \intG \big(\psi_1\boldsymbol{w} + \psi\boldsymbol{w}_2\big)\cdot\Gradg\mathcal{S}_{L,\beta}^\Ga(\phi,\psi)\dG.
        \end{split}
    \end{align}
    Now, exploiting the facts that $\abs{\phi_1} < 1$ a.e. in $Q$ as well as $\abs{\psi_1} < 1$ a.e. on $\Sigma$, we can estimate the last two terms in \eqref{Uniq:DiffIneq}, using the Sobolev inequality and the bulk-surface Poincar\'{e} inequality, as
    \begin{align}\label{Uniq:Est:Conv}
        \begin{split}
            &\intO \big(\phi_1\boldsymbol{v} + \phi\boldsymbol{v}_2\big)\cdot\Grad\mathcal{S}_{L,\beta}^\Om(\phi,\psi)\dx + \intG \big(\psi_1\boldsymbol{w} + \psi\boldsymbol{w}_2\big)\cdot\Gradg\mathcal{S}_{L,\beta}^\Ga(\phi,\psi)\dG \\
            &\quad\leq\norm{(\boldsymbol{v},\boldsymbol{w})}_{\mathcal{L}^2}\norm{\mathcal{S}_{L,\beta}(\phi,\psi)}_{L,\beta,\ast} + C\norm{(\phi,\psi)}_{K,\alpha}\norm{(\boldsymbol{v}_2,\boldsymbol{w}_2)}_{\mathbf{L}^3(\Om)\times\mathbf{L}^{2+\omega}(\Ga)}\norm{\mathcal{S}_{L,\beta}(\phi,\psi)}_{L,\beta,\ast} \\
            &\quad\leq \frac12\norm{(\boldsymbol{v},\boldsymbol{w})}_{\mathcal{L}^2} + \frac14\norm{(\phi,\psi)}_{K,\alpha}^2 + C_0\big(1 + \norm{(\boldsymbol{v}_2,\boldsymbol{w}_2)}_{\mathbf{L}^3(\Om)\times\mathbf{L}^{2+\omega}(\Ga)}^2\big)\norm{\mathcal{S}_{L,\beta}(\phi,\psi)}_{L,\beta,\ast}^2
        \end{split}
    \end{align}
    Combining the two estimates \eqref{Uniq:DiffIneq}-\eqref{Uniq:Est:Conv}, an application of Gronwall's lemma readily yields the claim.
\end{proof}

\section{Higher regularity and strong solutions}
\label{SECT:HIGHREG}

In this final section, we present the proof of Theorem~\ref{Theorem:HighReg}.

\begin{proof}[Proof of Theorem~\ref{Theorem:HighReg}.]
As the mobility functions are constant, we assume, without loss of generality, that $m_\Om \equiv 1$ and $m_\Ga \equiv 1$.
The proof is divided into several steps. 

\textbf{Step 1: Approximation of the singular potentials.} 
We introduce families of regular potentials $F_\lambda$ and $G_\lambda$ to approximate the singular potentials $F$ and $G$, respectively. To be precise, for $\lambda>0$, we set
\begin{align*}
    F_\lambda(s) = F_{1,\lambda}(s) + F_2(s), \qquad G_\lambda(s) = G_{1,\lambda}(s) + G_2(s), \qquad\text{for all~}s\in\R,
\end{align*}
where, as in Section~\ref{SECT:SUBDIFF}, for any $\lambda\in(0,\infty)$, $F_{1,\lambda}$ and $G_{1,\lambda}$ denote the Moreau--Yosida regularizations of the singular parts $F_1$ and $G_1$, respectively. We recall from Section~\ref{SECT:SUBDIFF} that for any $\lambda>0$, $F_{1,\lambda}$ and $G_{1,\lambda}$ have the properties \ref{Yosida:Reg}--\ref{Yosida:Convergence}.
In addition, as discussed in Section~\ref{SECT:SUBDIFF} (see \eqref{DOMINATION:YOSIDA}), it holds that
\begin{align}\label{Yosida:Domination}
    \abs{F_{1,\lambda}^\prime(\alpha r)} \leq \kappa_1 \abs{G_{1,\lambda}^\prime(r)} + \kappa_2 
\end{align}
for all $r\in\R$ and every $\lambda>0$.
Moreover, the assumptions on $F_2$ together with \ref{Yosida:Growth} imply that there exist constants $C_1,C_2 > 0$, which comply with the convention stated in Remark~\ref{REMARK:CONST}, such that
\begin{align}\label{Pot:Approx:Growth}
    F_\lambda(r) \geq C_1 r^2 - C_2 \qquad\text{for all~}r\in\R \text{~and~}\lambda\in(0,\overline\lambda),
\end{align}
where $\overline\lambda$ is chosen as in \ref{Yosida:Growth}. An analogous estimate holds for $G_{1,\lambda}$, and we write $\lambda_\ast>0$ to denote the common threshold such that \eqref{Pot:Approx:Growth} holds for both $F_\lambda$ and $G_\lambda$. Therefore, in the following, we simply consider $\lambda\in (0,\lambda_*)$.

\textbf{Step 2: Approximation of the initial datum.} We approximate the initial condition $(\phi_0,\psi_0)$ in a suitable way such that the approximation is compatible with the Yosida approximation of the potentials. To this end, we consider
\begin{subequations}\label{Approx:IC:lambda}
    \begin{alignat}{2}
        -\Lap\phi_{0,\lambda}  + F_{1,\lambda}^\prime(\phi_{0,\lambda}) &= \mu_0 - F_2^\prime(\phi_0) &&\qquad\text{in~}\Om, \\
        -\Lapg\psi_{0,\lambda} + G_{1,\lambda}^\prime(\psi_{0,\lambda}) + \alpha\deln\phi_{0,\lambda} &= \theta_0 - G_2^\prime(\psi_0) &&\qquad\text{on~}\Ga, \\
        K\deln\phi_{0,\lambda} &= \alpha\psi_{0,\lambda} - \phi_{0,\lambda} &&\qquad\text{on~}\Ga.
    \end{alignat}
\end{subequations}
According to Proposition~\ref{App:Prop:Existence}, Proposition~\ref{App:Proposition:Sol:L^p+H^k} and Remark~\ref{REM:YOS}, this system possesses a unique strong solution $(\phi_{0,\lambda},\psi_{0,\lambda})\in\mathcal{W}^{2,6}\cap\mathcal{H}^1_{K,\alpha}$. Moreover, there exists a constant $C_{\mathrm{init}} > 0$, that is independent of $\lambda$, such that
\begin{align}\label{Est:Init:PP}
    \norm{(\phi_{0,\lambda},\psi_{0,\lambda})}_{\mathcal{W}^{2,6}} + \norm{(F_{1,\lambda}^\prime(\phi_{0,\lambda}),G_{1,\lambda}^\prime(\psi_{0,\lambda}))}_{\mathcal{L}^6} \leq C_{\mathrm{init}}.
\end{align}
In particular, there exists a pair $(\tilde\phi_0,\tilde\psi_0)\in\mathcal{W}^{2,6}$ such that
\begin{alignat}{2}
    (\phi_{0,\lambda},\psi_{0,\lambda}) &\rightarrow (\tilde\phi_0,\tilde\psi_0) &&\qquad\text{weakly in~}\mathcal{W}^{2,6}, \text{~strongly in~}\mathcal{H}^1_{K,\alpha}, \label{Conv:Init:PP}\\
    (F_{1,\lambda}^\prime(\phi_{0,\lambda}),G_{1,\lambda}^\prime(\psi_{0,\lambda})) &\rightarrow (F_1^\prime(\tilde\phi_0),G_1^\prime(\tilde\psi_0)) &&\qquad\text{weakly in~}\mathcal{L}^6 \label{Conv:Init:Pot}
\end{alignat}
as $\lambda\rightarrow 0$ after possibly extracting a subsequence. Now, we define
\begin{align}
    \label{Def:Approx:M:0:lambda}
    \mu_{0,\lambda} &\coloneqq \mu_0 - F_2^\prime(\phi_0) + F_2^\prime(\phi_{0,\lambda}), 
    \\
    \label{Def:Approx:T:0:lambda}
    \theta_{0,\lambda} &\coloneqq \theta_0 - G_2^\prime(\psi_0) + G_2^\prime(\psi_{0,\lambda}).
\end{align}
This means that $(\phi_{0,\lambda},\psi_{0,\lambda})$ solves
\begin{subequations}\label{Approx:IC:lambda:2}
    \begin{alignat}{2}
        -\Lap\phi_{0,\lambda}  + F_\lambda^\prime(\phi_{0,\lambda}) &= \mu_{0,\lambda} &&\qquad\text{in~}\Om, \\
        -\Lapg\psi_{0,\lambda} + G_\lambda^\prime(\psi_{0,\lambda}) + \alpha\deln\phi_{0,\lambda} &= \theta_{0,\lambda} &&\qquad\text{on~}\Ga, \\
        K\deln\phi_{0,\lambda} &= \alpha\psi_{0,\lambda} - \phi_{0,\lambda} &&\qquad\text{on~}\Ga.
    \end{alignat}
\end{subequations}
Then, due to \eqref{Conv:Init:PP} and the Lipschitz continuity of $F_2^\prime$ and $G_2^\prime$, respectively, it holds that
\begin{align}
    (\mu_{0,\lambda},\theta_{0,\lambda}) \rightarrow (\mu_0,\theta_0) \qquad\text{weakly in~}\mathcal{H}^1, \text{~strongly in~}\mathcal{L}^2 \label{Conv:Init:MT}
\end{align}
as $\lambda\rightarrow 0$. In view of the convergences \eqref{Conv:Init:PP}, \eqref{Conv:Init:Pot} and \eqref{Conv:Init:MT}, it is straightforward to check that
\begin{align*}
    &\intO \Grad\tilde\phi_0\cdot\Grad\eta + F^\prime(\tilde\phi_0)\eta\dx + \intG \Gradg\tilde\psi_0\cdot\Gradg\vartheta + G^\prime(\tilde\psi_0)\vartheta\dG \\
    &\qquad + \sigma(K)\intG (\alpha\tilde\psi_0 - \tilde\phi_0)(\alpha\vartheta - \eta)\dG \\
    &\quad = \intO \mu_0\eta\dx + \intG \theta_0\vartheta\dG
\end{align*}
for all $(\eta,\vartheta)\in\mathcal{H}^1_{K,\alpha}$. 
In view of Assumption~\ref{cond:MT:0}, we thus know that both $(\tilde\phi_0,\tilde\psi_0)$ and $(\phi_0,\psi_0)$ are a weak solution of System~\ref{BSE} with $(f,g) = (\mu_0,\theta_0)$.
Hence, Theorem~\ref{App:Prop:ContDep} implies that
\begin{align}
    \label{ID:INI}
    (\tilde\phi_0,\tilde\psi_0) = (\phi_0,\psi_0) \qquad\text{a.e. in~}\Om\times\Ga.
\end{align}
Furthermore, thanks to this identification and the strong $\mathcal{H}^1$-convergence \eqref{Conv:Init:PP}, there exist $\overline{m}\in(0,1)$, which is independent of $\lambda$, such that, after possibly updating the value of $\lambda_\ast$, it holds
\begin{align}\label{ApproxIC:mean:L}
    \abs{\beta\mean{\phi_{0,\lambda}}{\psi_{0,\lambda}}} \leq \overline m < 1, \qquad \abs{\mean{\phi_{0,\lambda}}{\psi_{0,\lambda}}} \leq \overline m < 1 \qquad\text{for all~} 0 < \lambda < \lambda_\ast
\end{align}
if $L\in(0,\infty)$, and
\begin{align}\label{ApproxIC:mean:inf}
    \abs{\meano{\phi_{0,\lambda}}} \leq \overline m < 1, \qquad \abs{\meang{\psi_{0,\lambda}}} \leq \overline m < 1 \qquad\text{for all~} 0 < \lambda < \lambda_\ast
\end{align}
if $L = \infty$.
Additionally, as we have the uniform convergence $(\phi_{0,\lambda},\psi_{0,\lambda})\rightarrow (\phi_0,\psi_0)$ as $\lambda\rightarrow 0$ and knowing that $\norm{\phi_0}_{L^\infty(\Om)} \leq 1$ as well as $\norm{\psi_0}_{L^\infty(\Ga)} \leq 1$, we infer that 
\begin{align}\label{Est:Approx:Init:infty}
    \norm{\phi_{0,\lambda}}_{L^\infty(\Om)} \leq 2, \qquad \norm{\psi_{0,\lambda}}_{L^\infty(\Ga)} \leq 2 \qquad\text{for all~}0 < \lambda < \lambda_\ast
\end{align}
after a possible redefinition of $\lambda_\ast$.

\textbf{Step 3: Approximation of the velocity fields.} As the last step of the approximation procedure, we regularize the velocity fields. To this end, we start by extending $\boldsymbol{v}$ and $\boldsymbol{w}$ in time by zero to obtain $(\widetilde{\boldsymbol{v}},\widetilde{\boldsymbol{w}})\in L^\infty(\R;\mathcal{L}^2_\Div)\cap L^2(\R;\mathcal{H}^1)$.
We now choose a mollifier $\rho\in C^\infty_c(\R;[0,\infty))$ with $\norm{\rho}_{L^1(\R)}=1$, and for any $\lambda\in(0,\lambda_*)$, we set $\rho_\lambda := \lambda^{-1} \rho(\lambda^{-1}\, \cdot)$. We further define $\boldsymbol{v}_\lambda = \widetilde{\boldsymbol{v}}\ast_t\rho_\lambda$ and $\boldsymbol{w}_\lambda = \widetilde{\boldsymbol{w}}\ast_t\rho_\lambda$, where $\ast_t$ denotes the convolution with respect to the time variable. This provides a sequence $\{(\boldsymbol{v}_\lambda,\boldsymbol{w}_\lambda)\}_{\lambda\in(0,\lambda_*)}$ such that $(\boldsymbol{v}_\lambda,\boldsymbol{w}_\lambda)\in C^\infty(\R;\mathcal{H}^1)$ and
\begin{align}\label{Conv:vw:ep}
    \begin{split}
        (\boldsymbol{v}_\lambda,\boldsymbol{w}_\lambda) \rightarrow (\boldsymbol{v},\boldsymbol{w}) &\qquad\text{strongly in~} L^2(0,T;\mathcal{H}^1), \\
        &\qquad\text{weakly-$\ast$ in~} L^\infty(0,T;\mathcal{L}^2),
    \end{split}
\end{align}
as $\lambda\rightarrow 0$, up to subsequence extraction.
For more details, we refer to \cite[S.~4.13-4.15]{Alt2016}.
In particular, according to \cite[S.~4.13 (2)]{Alt2016}, we have
\begin{align}\label{Est:vw:ep}
    \norm{(\boldsymbol{v}_\lambda,\boldsymbol{w}_\lambda)}_{L^\infty(0,T;\mathcal{L}^2)\cap L^2(0,T;\mathcal{H}^1)} 
    \leq \norm{(\boldsymbol{v},\boldsymbol{w})}_{L^\infty(0,T;\mathcal{L}^2)\cap L^2(0,T;\mathcal{H}^1)} 
\end{align}
for all $\lambda\in(0,\lambda_*)$.
We further point out that the approximate velocity fields are still divergence-free and preserve the trace relation $\boldsymbol{v}_\lambda\vert_\Ga = \boldsymbol{w}_\lambda$ a.e.~on $\Sigma$ if $K = 0$ as the convolution only affects the time variable. 

\textbf{Step 4: The approximate system.}
Let now $\lambda\in(0,\lambda_\ast)$ be arbitrary. Then, in view of \cite[Theorem 3.2]{Knopf2024}, there exists a unique quadruplet $(\phi_\lambda, \psi_\lambda, \mu_\lambda,\theta_\lambda)$ such that
\begin{align*}
    (\phi_\lambda,\psi_\lambda)&\in C([0,T];\mathcal{L}^2)\cap H^1(0,T;(\mathcal{H}^1)^\prime)\cap L^\infty(0,T;\mathcal{H}^1_{K,\alpha}), \\
    (\mu_\lambda,\theta_\lambda)&\in L^2(0,T;\mathcal{H}^1),
\end{align*}
which solves
\begin{subequations}\label{Approx:Problem}
    \begin{align}
        \label{Approx:Problem:1}
        &\delt\phi_\lambda + \Div(\phi_\lambda\boldsymbol{v}_\lambda) = \Lap\mu_\lambda && \text{in} \ Q, \\
        \label{Approx:Problem:2}
        &\mu_\lambda = -\Lap\phi_\lambda + F_\lambda'(\phi_\lambda)   && \text{in} \ Q, \\
        \label{Approx:Problem:3}
        &\delt\psi_\lambda + \Divg(\psi_\lambda\boldsymbol{w}_\lambda) = \Lapg\theta_\lambda - \beta \deln\mu_\lambda && \text{on} \ \Sigma, \\
        \label{Approx:Problem:4}
        &\theta_\lambda = - \Lapg\psi_\lambda + G_\lambda'(\psi_\lambda) + \alpha\deln\phi_\lambda && \text{on} \ \Sigma, \\
        \label{Approx:Problem:5}
        &K\deln\phi_\lambda = \alpha\psi_\lambda - \phi_\lambda \qquad\text{for~}K\in [0,\infty) && \text{on} \ \Sigma, \\
        \label{Approx:ProblemM:6}
        &\begin{cases} 
        L \deln\mu_\lambda = \beta\theta_\lambda - \mu_\lambda &\text{if} \  L\in(0,\infty), \\
        \deln\mu_\lambda = 0 &\text{if} \ L=\infty
        \end{cases} &&\text{on} \ \Sigma, \\
        \label{Approx:Problem:7}
        &\phi_\lambda\vert_{t=0} = \phi_{0,\lambda} &&\text{in} \ \Om, \\
        \label{Approx:Problem:8}
        &\psi_\lambda\vert_{t=0} = \psi_{0,\lambda} &&\text{on} \ \Ga,
    \end{align}
\end{subequations}
in the sense that it satisfies the weak formulation \eqref{WF:PP:SING}-\eqref{WF:MT:SING} with $F$ and $G$ replaced by $F_\lambda$ and $G_\lambda$, respectively, and $\boldsymbol{v}$ and $\boldsymbol{w}$ by $\boldsymbol{v}_\lambda$ and $\boldsymbol{w}_\lambda$, respectively. Furthermore, the solution satisfies the mass conservation law \eqref{MCL:SING} as well as the energy inequality
\begin{align}\label{EI:ApproxProblem}
    \begin{split}
        &E_{K,\lambda}(\phi_\lambda(t),\psi_\lambda(t)) + \int_0^t\intO \abs{\Grad\mu_\lambda}^2\dxs + \int_0^t\intG \abs{\Gradg\theta_\lambda}^2\dGs \\
        &\quad + \sigma(L) \int_0^t\intG (\beta\theta_\lambda-\mu_\lambda)^2\dGs \\
        &\leq E_{K,\lambda}(\phi_{0,\lambda},\psi_{0,\lambda}) + \int_0^t\intO \phi_\lambda\boldsymbol{v}_\lambda\cdot\Grad\mu_\lambda \dxs + \int_0^t\intG \psi_\lambda\boldsymbol{w}_\lambda\cdot\Gradg\theta_\lambda\dGs
    \end{split}
\end{align}
for all $t\in[0,T]$, where the approximate energy $E_{K,\lambda}$ is given by
\begin{align}
    \label{ApproxEnergy}
    \begin{split}
        E_{K,\lambda}(\zeta,\xi) &= \intO \frac12 \abs{\Grad\zeta}^2 + F_\lambda(\zeta)\dx + \intG \frac12 \abs{\Gradg\xi}^2 + G_\lambda(\xi)\dG + \sigma(K) \intG\frac12 (\alpha\xi - \zeta)^2 \dG.
    \end{split}
\end{align}
To control the approximated initial energy in \eqref{EI:ApproxProblem}, we test \eqref{Approx:IC:lambda} with $(\phi_{0,\lambda},\psi_{0,\lambda})$ and obtain
\begin{align*}
    &\norm{(\phi_{0,\lambda},\psi_{0,\lambda})}_{K,\alpha}^2 + \intO F_{1,\lambda}^\prime(\phi_{0,\lambda})\phi_{0,\lambda}\dx + \intG G_{1,\lambda}^\prime(\psi_{0,\lambda})\psi_{0,\lambda}\dG \\
    &\quad = \intO \big(\mu_0 - F_2^\prime(\phi_0)\big)\phi_{0,\lambda}\dx + \intG \big(\theta_0 - G_2^\prime(\psi_0)\big)\psi_{0,\lambda}\dG.
\end{align*}
Since by convexity $F_{1,\lambda}^\prime(s)s\geq F_{1,\lambda}(s)$ as well as $G_{1,\lambda}^\prime(s)s\geq G_{1,\lambda}(s)$ for all $s\in\R$, we use the bounds \eqref{Est:Approx:Init:infty} to deduce from the equation above that
\begin{align*}
    \intO F_{1,\lambda}(s)\dx + \intG G_{1,\lambda}(s)\dG \leq C\big(1 + \norm{(\mu_0,\theta_0)}_{L,\beta}^2\big),
\end{align*}
and readily find that
\begin{align}\label{EST:Energy:INIT}
    E_{K,\lambda}(\phi_{0,\lambda},\psi_{0,\lambda}) 
    \leq E_{K}(\phi_{0,\lambda},\psi_{0,\lambda}) 
    \leq C\big(1 + \norm{(\phi_0,\psi_0)}_{K,\alpha}^2 + \norm{(\mu_0,\theta_0)}_{L,\beta}^2\big).
\end{align}
Furthermore, we bound the convective terms on the right-hand side of \eqref{EI:ApproxProblem} by Hölder's, Young's and Sobolev's inequalities. Namely, it holds that
\begin{align}
    \label{Est:ApproxProblem:Convective}
    \begin{split}
        &\abs{\int_0^t\intO \phi_\lambda\boldsymbol{v}_\lambda\cdot\Grad\mu_\lambda \dxs + \int_0^t\intG \psi_\lambda\boldsymbol{w}_\lambda\cdot\Gradg\theta_\lambda\dGs} \\
        &\quad\leq \frac{1}{2}\int_0^t\intO \abs{\Grad\mu_\lambda}^2\dxs + \frac{1}{2}\int_0^t\intG \abs{\Gradg\theta_\lambda}^2\dGs \\
        &\qquad + \frac12\int_0^t \norm{(\boldsymbol{v}_\lambda, \boldsymbol{w}_\lambda)}_{\mathcal{L}^3}^2\norm{(\phi_\lambda, \psi_\lambda)}_{\mathcal{H}^1}^2\ds.
    \end{split}
\end{align}
Next, in view of \eqref{Pot:Approx:Growth}, we find that
\begin{align}\label{Est:Appl:Growth}
    \intO F_\lambda(\phi_\lambda)\dx + \intG G_\lambda(\psi_\lambda)\dG \geq C\norm{(\phi_\lambda,\psi_\lambda)}_{\mathcal{L}^2}^2 - C
\end{align}
since $\lambda\in(0,\lambda_\ast)$. This implies that
\begin{align}\label{Est:Appl:Growth:2}
    c\,\norm{(\phi_\lambda,\psi_\lambda)}_{\mathcal{H}^1}^2 &\leq C + \frac12\norm{(\Grad\phi_\lambda,\Gradg\psi_\lambda)}_{\mathcal{L}^2}^2 + \intO F_\lambda(\phi_\lambda)\dx + \intG G_\lambda(\psi_\lambda) \dG.
\end{align}
for some constant $c>0$.
Collecting \eqref{EST:Energy:INIT}, \eqref{Est:ApproxProblem:Convective} and \eqref{Est:Appl:Growth:2}, we infer from \eqref{EI:ApproxProblem} that
\begin{align}\label{Est:ApproxProblem:PreGronwall} 
    \begin{split}
        &c\norm{(\phi_\lambda,\psi_\lambda)}_{\mathcal{H}^1}^2 + \frac{1}{2}\int_0^t\intO \abs{\Grad\mu_\lambda}^2\dxs + \frac{1}{2}\int_0^t\intG \abs{\Gradg\theta_\lambda}^2\dGs \\
        &\qquad + \sigma(L) \int_0^t\intG (\beta\theta_\lambda - \mu_\lambda)^2\dGs \\
        &\quad\leq C + C\int_0^t \norm{(\boldsymbol{v}_\lambda, \boldsymbol{w}_\lambda)}_{\mathcal{L}^3}^2\norm{(\phi_\lambda, \psi_\lambda)}_{\mathcal{H}^1}^2\ds.
    \end{split}
\end{align}
Hence, an application of Gronwall's lemma together with \eqref{Est:vw:ep} readily yields that
\begin{align}
    \label{Est:ApproxProblem:Gronwall}
    \begin{split}
        \norm{(\phi_\lambda,\psi_\lambda)}_{\mathcal{H}^1}^2 &\leq C \exp\Big(C\int_0^T \norm{(\boldsymbol{v}_\lambda, \boldsymbol{w}_\lambda)}_{\mathcal{L}^3}^2\ds\Big) \leq C.
    \end{split}
\end{align}
Recalling that $F_\lambda$ and $G_\lambda$ are bounded from below (cf.~\eqref{Pot:Approx:Growth}), we deduce from \eqref{Est:ApproxProblem:PreGronwall} and \eqref{Est:ApproxProblem:Gronwall} that
\begin{align}
    \label{Est:ApproxProblem:uniform:1}
    \begin{split}
        &\norm{(\phi_\lambda,\psi_\lambda)}_{L^\infty(0,T;\mathcal{H}^1)}^2 + \norm{(F_\lambda(\phi_\lambda),G_\lambda(\psi_\lambda))}_{L^\infty(0,T;\mathcal{L}^1)} + \sigma(K) \norm{\alpha\psi_\lambda - \phi_\lambda}_{L^\infty(0,T;L^2(\Ga))}^2 \\
        &\quad + \norm{(\Grad\mu_\lambda, \Gradg\theta_\lambda)}_{L^2(0,T;\mathcal{L}^2)}^2 + \sigma(L)\norm{\beta\theta_\lambda - \mu_\lambda}_{L^2(0,T;L^2(\Ga))}^2 \leq C.
    \end{split}
\end{align}
In order to find an $\mathcal{H}^1$ 
-estimate on the chemical potentials $\mu_\lambda$ and $\theta_\lambda$, we only consider the case $L\in(0,\infty)$. The proof in the case $L = \infty$ works similarly with obvious modifications. First, we recall the inequality 
\begin{align}
    \label{Est:ApproxProblem:MZ}
    \begin{split}
        & c_1\norm{F_{1,\lambda}^\prime(\phi_\lambda)}_{L^1(\Om)} + c_2\norm{G_{1,\lambda}^\prime(\psi_\lambda)}_{L^1(\Ga)} - c_3 \\
        &\quad\leq \intO F_{1,\lambda}^\prime(\phi_\lambda)(\phi_\lambda - \beta\mean{\phi_\lambda}{\psi_\lambda})\dx + \intG G_{1,\lambda}^\prime(\psi_\lambda)(\psi_\lambda - \mean{\phi_\lambda}{\psi_\lambda})\dG,
    \end{split}
\end{align}
(see \cite[Proposition~A.1]{Miranville2004} or \cite[p.908]{Gilardi2009}),
where $c_1,c_2$ are positive constants and $c_3$ is non-negative. Note that these constants are independent of $\lambda$. Here, we used again the mass conservation law, \eqref{ApproxIC:mean:L} and \eqref{ApproxIC:mean:inf}. Using \eqref{Est:ApproxProblem:MZ}, we can argue analogously to the proof of \cite[Theorem 3.4]{Knopf2024a} to deduce the estimate
\begin{align}
    \label{Est:ApproxProblem:Pot:L^1:pre}
    \norm{(F_{1,\lambda}^\prime(\phi_\lambda), G_{1,\lambda}^\prime(\psi_\lambda))}_{\mathcal{L}^1} \leq C\big(1 + \norm{(\mu_\lambda,\theta_\lambda)}_{L,\beta}\big) + C\gamma(K)\norm{\deln\phi_\lambda}_{L^2(\Ga)},
\end{align}
where, as before, $\gamma(K) = \mathbf{1}_{\{0\}}(K)$ denotes the characteristic function of the set $\{0\}$.

Next, we derive $L^2$-bounds on $F_{1,\lambda}^\prime(\phi_\lambda)$ and $G_{1,\lambda}^\prime(\psi_\lambda)$. Invoking the Lipschitz continuity of $F_2^\prime$ and $G_2^\prime$ as well as \eqref{Est:ApproxProblem:uniform:1} and \eqref{Est:ApproxProblem:Pot:L^1:pre}, we deduce that
\begin{align}
    \label{Est:ApproxProblem:Pot:L^1:K}
    \norm{(F_\lambda^\prime(\phi_\lambda), G_\lambda^\prime(\psi_\lambda))}_{\mathcal{L}^1} \leq C\big(1 + \norm{(\mu_\lambda,\theta_\lambda)}_{L,\beta}\big) + C\gamma(K)\norm{\deln\phi_\lambda}_{L^2(\Ga)}.
\end{align}
Thus, testing \eqref{WF:MT:SING} with $(\beta^2\abs{\Om} + \abs{\Ga})^{-1}(\beta,1)$, we deduce from \eqref{Est:ApproxProblem:uniform:1} and \eqref{Est:ApproxProblem:Pot:L^1:K} that
\begin{align}
    \label{Est:ApproxProblem:mean:MT:K}
    \abs{\mean{\mu_\lambda}{\theta_\lambda}} \leq C\big(1 + \norm{(\mu_\lambda,\theta_\lambda)}_{L,\beta}\big) + C\gamma(K)\norm{\deln\phi_\lambda}_{L^2(\Ga)}.
\end{align}
Consequently, using \eqref{Est:ApproxProblem:mean:MT:K} together with the bulk-surface Poincar\'{e} inequality (Lemma~\ref{Prelim:Poincare}), we infer that
\begin{align}
    \label{Est:ApproxProblem:MT:L^2:K}
    \norm{(\mu_\lambda,\theta_\lambda)}_{\mathcal{H}^1} \leq C\big(1 + \norm{(\mu_\lambda,\theta_\lambda)}_{L,\beta}\big) + C\gamma(K)\norm{\deln\phi_\lambda}_{L^2(\Ga)}.
\end{align}
Now, according to Proposition~\ref{App:Proposition:Sol:L^p+H^k} and Remark~\ref{REM:YOS} we have
\begin{align}\label{Est:ApproxProblem:PP:POT:L^6:1}
    \norm{(\phi_\lambda,\psi_\lambda)}_{\mathcal{W}^{2,6}} + \norm{(F_{1,\lambda}^\prime(\phi_\lambda),G_{1,\lambda}^\prime(\psi_\lambda))}_{\mathcal{L}^6} \leq C\big(1 + \norm{(\mu_\lambda - F_2^\prime(\phi_\lambda),\theta_\lambda - G_2^\prime(\psi_\lambda))}_{\mathcal{H}^1}\big).
\end{align}
In view of the Lipschitz continuity of $F_2^\prime$ and $G_2^\prime$, \eqref{Est:ApproxProblem:Gronwall} and \eqref{Est:ApproxProblem:MT:L^2:K}, we deduce from \eqref{Est:ApproxProblem:PP:POT:L^6:1} that
\begin{align}\label{Est:ApproxProblem:PP:POT:L^6:2}
    \norm{(\phi_\lambda,\psi_\lambda)}_{\mathcal{W}^{2,6}} + \norm{(F_{1,\lambda}^\prime(\phi_\lambda),G_{1,\lambda}^\prime(\psi_\lambda))}_{\mathcal{L}^6} \leq C\big(1 + \norm{(\mu_\lambda,\theta_\lambda)}_{L,\beta} + \gamma(K)\norm{\deln\phi_\lambda}_{L^2(\Ga)}\big)
\end{align}
Lastly, we establish a uniform $L^2$-bound on $\deln\phi_\lambda$. To this end, by the trace theorem (see, e.g., \cite[Chapter 2, Theorem 2.24]{Brezzi1987}), we have
\begin{align}\label{Est:ApproxProblem:Trace}
    \norm{\deln\phi_\lambda}_{L^2(\Ga)} \leq C\norm{\phi_\lambda}_{H^{7/4}(\Om)}.
\end{align}
Thus, invoking\eqref{Est:ApproxProblem:PP:POT:L^6:2}-\eqref{Est:ApproxProblem:Trace} and applying the compact embeddings $W^{2,6}(\Om)\emb H^{7/4}(\Om) \emb L^2(\Om)$ along with Ehrling's lemma, we obtain 
\begin{align}\label{Est:Approx:H^2+trace:Phi}
    \begin{split}
        &\norm{(\phi_\lambda,\psi_\lambda)}_{\mathcal{W}^{2,6}} + \norm{\deln\phi_\lambda}_{L^2(\Ga)}  \\
        &\quad\leq C\big(1 + \norm{(\mu_\lambda,\theta_\lambda)}_{L,\beta}\big) + C\norm{\phi_\lambda}_{H^{7/4}(\Om)} \\
        &\quad\leq \frac12\norm{\phi_\lambda}_{W^{2,6}(\Om)} + C\big(1 + \norm{(\mu_\lambda,\theta_\lambda)}_{L,\beta}\big).
    \end{split}
\end{align}
Whence, absorbing the respective norm, we infer that
\begin{align}\label{Est:ApproxProblem:PP:W^26}
    \norm{(\phi_\lambda,\psi_\lambda)}_{\mathcal{W}^{2,6}} + \norm{\deln\phi_\lambda}_{L^2(\Ga)}\leq C\big(1 + \norm{(\mu_\lambda,\theta_\lambda)}_{L,\beta}\big),
\end{align}
and thus, in view of \eqref{Est:ApproxProblem:MT:L^2:K} and \eqref{Est:ApproxProblem:PP:POT:L^6:1}, we conclude that
\begin{align}\label{Est:ApproxProblem:MT:POT:H^1:L^6}
    \norm{(F_{1,\lambda}^\prime(\phi_\lambda), G_{1,\lambda}^\prime(\psi_\lambda))}_{\mathcal{L}^6} + \norm{(\mu_\lambda, \theta_\lambda)}_{\mathcal{H}^1} \leq C\big(1 + \norm{(\mu_\lambda,\theta_\lambda)}_{L,\beta}\big).
\end{align}

\textbf{Step 5: Higher order regularity.} As the approximate velocity fields $\boldsymbol{v}_\lambda$ and $\boldsymbol{w}_\lambda$ fulfill the regularity assumption \eqref{ASS:VEL:STRONG:OLD} that was used in \cite[Theorem 3.7]{Knopf2024a}, we can adapt the proof to conclude that
\begin{subequations}\label{REG:STRG:ep}
    \begin{align}
        \scp{\delt\phi_\lambda}{\delt\psi_\lambda}&\in L^\infty(0,T;(\mathcal{H}^1)^\prime)\cap L^2(0,T;\mathcal{H}^1), \\
        \label{REG:PP:W26inf}
        \scp{\phi_\lambda}{\psi_\lambda}&\in L^\infty(0,T;\mathcal{W}^{2,6})\cap\Big( C\big(\overline{Q}\big)\times C\big(\overline{\Sigma}\big)\Big), \\
        \label{REG:MT:H1inf}
        \scp{\mu_\lambda}{\theta_\lambda}&\in L^\infty(0,T;\mathcal{H}^1)\cap L^2(0,T;\mathcal{H}^2), \\
        \scp{F_\lambda^\prime(\phi_\lambda)}{G_\lambda^\prime(\psi_\lambda)}&\in L^2(0,T;\mathcal{L}^\infty) \cap L^\infty(0,T;\mathcal{L}^6).
    \end{align}
   \end{subequations}
In addition, $(\phi_\lambda, \psi_\lambda, \mu_\lambda, \theta_\lambda)$ satisfies the equations
\begin{subequations}\label{STRG:PP:EP}
    \begin{alignat}{2}
        \label{STR:PHI:EP}
        &\delt\phi_\lambda + \Div(\phi_\lambda\boldsymbol{v}_\lambda) = \Lap\mu_\lambda &&\qquad\text{a.e.~in~}Q, \\
        \label{STR:PSI:EP}
        &\delt\psi_\lambda + \Divg(\psi_\lambda\boldsymbol{w}_\lambda) = \Lapg\theta_\lambda - \beta\deln\mu_\lambda&&\qquad\text{a.e.~on~}\Sigma, \\
        &\begin{cases}
            L\deln\mu_\lambda = \beta\theta_\lambda - \mu_\lambda & \text{if~}L\in(0,\infty), \\
            \deln\mu_\lambda = 0 & \text{if~}L = \infty
        \end{cases}
        &&\qquad\text{a.e.~on~}\Sigma.
    \end{alignat}
\end{subequations}
We now prove several higher regularity results, which are not uniform with respect to $\lambda\in(0,\lambda_\ast)$, but are crucial for our analysis to prove the desired uniform estimates. 

\textbf{Step 5.1:} First, we prove that $(\delt\mu_\lambda,\delt\theta_\lambda)$ exists and belongs to $L^2(0,T;(\mathcal{H}^1_{K,\alpha})^\prime)$. To this end, we take the difference quotient of \eqref{Approx:Problem:2}, \eqref{Approx:Problem:4} and \eqref{Approx:Problem:5}, respectively. Then we have
\begin{subequations}\label{STRG:DELTH:lambda}
    \begin{align}
        \delt^h\mu_\lambda &= -\Lap\delt^h\phi_\lambda + \delt^h [F_\lambda^\prime(\phi_\lambda)] \qquad &&\text{a.e.~in~} Q, \label{STRG:DELTH:MU:lambda}\\
        \delt^h\theta_\lambda &= - \Lapg\delt^h\psi_\lambda + \delt^h [G_\lambda^\prime(\psi_\lambda)] + \alpha\deln\delt^h\phi_\lambda \qquad &&\text{a.e.~on~} \Sigma, \label{STRG:DELTH:THETA:lambda}\\
        K\deln\delt^h\phi_\lambda &= \alpha\delt^h\psi_\lambda - \delt^h\phi_\lambda &&  \text{a.e.~on } \Sigma,
    \end{align}
\end{subequations}
for $K\in[0,\infty)$. Here, the difference quotient of a function $f$ at time $t$ is defined as $\delth f(t) = \frac1h(f(t+h) - f(t))$ for $h\in(0,1)$.
Let now $(\eta,\vartheta)\in\mathcal{H}^1_{K,\alpha}$. Multiplying \eqref{STRG:DELTH:MU:lambda} and \eqref{STRG:DELTH:THETA:lambda} by $\eta$ and $\vartheta$, and integrating over $\Om$ and $\Ga$, respectively, an application of integration by parts yields
\begin{align}\label{WF:MUTHETA:h}
    \begin{split}
        &\ang{(\delth\mu_\lambda, \delth\theta_\lambda)}{(\eta,\vartheta)}_{\mathcal{H}^1_{K,\alpha}} = \intO \delth\mu_\lambda\eta\dx + \intG \delth\theta_\lambda\vartheta\dG \\
        &\quad = \intO \Grad\delth\phi_\lambda\cdot\Grad\eta + \delth [F_\lambda^\prime(\phi_\lambda)]\eta \dx \\
        &\qquad + \intG \Gradg\delth\psi_\lambda\cdot\Gradg\vartheta + \delth [G_\lambda^\prime(\psi_\lambda)]\vartheta \dG \\
        &\qquad + \sigma(K) \intG (\alpha\delth\psi_\lambda - \delth\phi_\lambda)(\alpha\vartheta - \eta) \dG.
    \end{split}
\end{align}
Consequently, we have
\begin{align*}
	&\abs{\bigang{(\delth\mu_\lambda, \delth\theta_\lambda)}{(\eta,\vartheta)}_{\mathcal{H}^1_{K,\alpha}}} \\
    &\qquad \leq C\Big(\norm{(\delth\phi_\lambda, \delth\psi_\lambda)}_{K,\alpha} + \norm{(\delth[F_\lambda^\prime(\phi_\lambda)], \delt^h[G_\lambda^\prime(\psi_\lambda)])}_{\mathcal{L}^2}\Big)\norm{(\eta,\vartheta)}_{\mathcal{H}^1}.
\end{align*}
Now, taking the supremum over all $(\eta, \vartheta)\in\mathcal{H}^1_{K,\alpha}$ with $\norm{(\eta, \vartheta)}_{\mathcal{H}^1} \leq 1$ yields
\begin{align}\label{EST:DELT:MT:EP}
	\norm{(\delth\mu_\lambda, \delth\theta_\lambda)}_{(\mathcal{H}^1_{K,\alpha})^\prime} &\leq C\Big(\norm{(\delth\phi_\lambda, \delth\psi_\lambda)}_{K,\alpha} + \norm{(\delth[F_\lambda^\prime(\phi_\lambda)], \delt^h[G_\lambda^\prime(\psi_\lambda)])}_{\mathcal{L}^2}\Big).
\end{align}
Moreover, we note that
\begin{align*}
    \norm{(\delth\phi_\lambda, \delth\psi_\lambda)}_{K,\alpha} 
    \leq C \norm{(\delt\phi_\lambda,\delt\psi_\lambda)}_{K,\alpha}.
\end{align*}
In view of \ref{Yosida:Lipschitz}, we infer
\begin{align*}
    \norm{(\delth[F_\lambda^\prime(\phi_\lambda)], \delt^h[G_\lambda^\prime(\psi_\lambda)])}_{\mathcal{L}^2} 
    \leq \frac{C}{\lambda} \norm{(\delt\phi_\lambda,\delt\psi_\lambda)}_{\mathcal{L}^2}
    \leq \frac{C}{\lambda} \norm{(\delt\phi_\lambda,\delt\psi_\lambda)}_{K,\alpha}.
\end{align*}
Consequently, due to \eqref{EST:DELT:MT:EP}, we obtain the bound
\begin{align}\label{EST:DELT:MT:EP*}
	\norm{(\delth\mu_\lambda, \delth\theta_\lambda)}_{(\mathcal{H}^1_{K,\alpha})^\prime} &\leq \frac{C}{\lambda} \norm{(\delt\phi_\lambda,\delt\psi_\lambda)}_{K,\alpha},
\end{align}
which is uniform with respect to $h$.
Hence, we conclude that $(\delt\mu_\lambda, \delt\theta_\lambda)\in L^2(0,T;(\mathcal{H}^1_{K,\alpha})^\prime)$. Furthermore, passing to the limit $h\rightarrow 0$ in \eqref{WF:MUTHETA:h}, we conclude that
\begin{align}\label{WF:DELT:MT}
	\begin{split}
		\bigang{(\delt\mu_\lambda, \delt\theta_\lambda)}{(\eta,\vartheta)}_{\mathcal{H}^1_{K,\alpha}} &= \intO \Grad\delt\phi_\lambda\cdot\Grad\eta + F_\lambda^{\prime\prime}(\phi_\lambda)\delt\phi_\lambda\eta\dx \\
		&\qquad + \intG \Gradg\delt\psi_\lambda\cdot\Gradg\vartheta + G_\lambda^{\prime\prime}(\psi_\lambda)\delt\psi_\lambda\vartheta \dG \\
		&\qquad + \sigma(K)\intG (\alpha\delt\psi_\lambda - \delt\phi_\lambda)(\alpha\vartheta - \eta)\dG
	\end{split}
\end{align}
almost everywhere on $[0,T]$ for all $(\eta,\vartheta)\in\mathcal{H}^1_{K,\alpha}$. 

\textbf{Step 5.2:} Next, we show that $(\phi_\lambda,\psi_\lambda)\in L^2(0,T;\mathcal{W}^{3,6})$.  To this end, we apply $L^p$ regularity theory for bulk-surface elliptic systems (see, e.g., \cite[Proposition A.1]{Knopf2024a}, adapted to $k=3$), and get
\begin{align}
    \label{Est:ApproxProblem:PP:W^3,6}
	\begin{split}
	\norm{(\phi_\lambda, \psi_\lambda)}_{\mathcal{W}^{3,6}} &\leq C\big(1 + \norm{(\mu_\lambda - F_\lambda^\prime(\phi_\lambda), \theta_\lambda - G_\lambda^\prime(\psi_\lambda))}_{\mathcal{W}^{1,6}}\big) \\
	&\leq C\big(1 + \norm{(\mu_\lambda, \theta_\lambda)}_{\mathcal{W}^{1,6}} + \norm{(F_\lambda^\prime(\phi_\lambda), G_\lambda^\prime(\psi_\lambda))}_{\mathcal{W}^{1,6}}\big).
	\end{split}
\end{align}
For the second term on the right-hand side of \eqref{Est:ApproxProblem:PP:W^3,6}, we note that $F_\lambda^\prime(\phi_\lambda)\in L^\infty(0,T;L^6(\Om))$. In addition, $\abs{\Grad F_\lambda^\prime(\phi_\lambda)} = \abs{F_\lambda^{\prime\prime}(\phi_\lambda)\Grad\phi_\lambda} \leq (C/\lambda)\abs{\Grad\phi_\lambda}$, and thus $\Grad F_\lambda^\prime(\phi_\lambda)\in L^\infty(0,T;L^6(\Om))$. Consequently, $F_\lambda^\prime(\phi_\lambda)\in L^\infty(0,T; W^{1,6}(\Om))$. In an analogous manner, we derive $G_\lambda^\prime(\psi_\lambda)\in L^\infty(0,T; W^{1,6}(\Ga))$. As also $(\mu,\theta)\in L^2(0,T;\mathcal{H}^2)\emb L^2(0,T;\mathcal{W}^{1,6})$, we deduce that $(\phi_\lambda, \psi_\lambda)\in L^2(0,T;\mathcal{W}^{3,6})$.

\textbf{Step 5.3:} Lastly, we prove that $(\mu_\lambda, \theta_\lambda)\in L^2(0,T;\mathcal{H}^3)\cap C([0,T];\mathcal{H}^1)$. Once again, by the elliptic regularity theory for bulk-surface elliptic systems, we obtain
\begin{align}
	\begin{split}
		&\norm{(\mu_\lambda, \theta_\lambda)}_{\mathcal{H}^3} \\
        &\quad\leq C \norm{(\delt\phi_\lambda - \Grad\phi_\lambda\cdot\boldsymbol{v}_\lambda, \delt\psi_\lambda - \Gradg\psi_\lambda\cdot\boldsymbol{w}_\lambda)}_{\mathcal{H}^1} + \abs{\mean{\mu_\lambda}{\theta_\lambda}} \\
		&\quad\leq C \norm{(\delt\phi_\lambda, \delt\psi_\lambda)}_{\mathcal{H}^1} + C\norm{(\Grad\phi_\lambda, \Gradg\psi_\lambda)}_{\mathcal{W}^{1,\infty}}\norm{(\boldsymbol{v}_\lambda, \boldsymbol{w}_\lambda)}_{\mathcal{H}^1} + C\big(1 + \norm{(\mu_\lambda,\theta_\lambda)}\big) \\
		&\quad\leq C \norm{(\delt\phi_\lambda, \delt\psi_\lambda)}_{\mathcal{H}^1} + C\norm{(\phi_\lambda,\psi_\lambda)}_{\mathcal{W}^{3,6}}\norm{(\boldsymbol{v}_\lambda, \boldsymbol{w}_\lambda)}_{\mathcal{H}^1} + C\big(1 + \norm{(\mu_\lambda,\theta_\lambda)}\big).
	\end{split}
\end{align}
This immediately entails $(\mu_\lambda, \theta_\lambda)\in L^2(0,T; \mathcal{H}^3)$ due to \eqref{Est:vw:ep} and the previous step. Moreover, as $(\delt\mu_\lambda,\delt\theta_\lambda)\in L^2(0,T;(\mathcal{H}^1_{K,\alpha})^\prime)$ and $(\mu_\lambda,\theta_\lambda)\in L^\infty(0,T;\mathcal{H}^1)$, the Aubin--Lions lemma implies that $(\mu_\lambda,\theta_\lambda)\in C([0,T];\mathcal{L}^2)$. Next, recalling that in the case $K = 0$ we have $\boldsymbol{v}_\lambda\vert_\Ga = \boldsymbol{w}_\lambda$ a.e.~on $\Sigma$, we compute
\begin{align*}
    \Grad\phi_\lambda\cdot\boldsymbol{v}_\lambda = [\Grad\phi_\lambda - \n(\Grad\phi_\lambda\cdot\n)]\cdot\boldsymbol{v}_\lambda = \Gradg\phi_\lambda\cdot\boldsymbol{v}_\lambda = \alpha\Gradg\psi_\lambda\cdot\boldsymbol{w}_\lambda \quad\text{a.e.~on $\Sigma$ if $K=0$},
\end{align*}
from which we obtain $(-\Lap\mu_\lambda, -\Lapg\theta_\lambda + \beta\deln\mu_\lambda) = - (\delt\phi_\lambda + \Grad\phi_\lambda\cdot\boldsymbol{v}_\lambda, \delt\psi_\lambda + \Gradg\psi_\lambda\cdot\boldsymbol{w}_\lambda)\in L^2(0,T;\mathcal{H}^1_{K,\alpha})$.
This allows us to apply \cite[Proposition~A.1]{Knopf2024}, which yields $(\mu_\lambda,\theta_\lambda)\in C([0,T];\mathcal{H}^1)$ along with the chain rule 
\begin{align}\label{CR:MUTHETA_:lambda}
\ddt \frac12\norm{(\mu_\lambda, \theta_\lambda)}_{L,\beta}^2 = \bigang{(\delt\mu_\lambda, \delt\theta_\lambda)}{(-\Lap\mu_\lambda, -\Lapg\theta_\lambda + \beta\deln\mu_\lambda)}_{\mathcal{H}^1_{K,\alpha}}.
\end{align}
We further notice that in view of the continuity $(\mu_\lambda,\theta_\lambda)\in C([0,T];\mathcal{H}^1)$, it holds
\begin{align}\label{MT:InitCond:Ident}
    (\mu_\lambda(0),\theta_\lambda(0)) = (\mu_{0,\lambda}, \theta_{0,\lambda})
    \quad\text{a.e.~in $\Omega\times \Gamma$.}
\end{align}
In particular, due to \eqref{Conv:Init:MT}, we thus have
\begin{align}\label{MT:InitCond:Conv}
    (\mu_\lambda(0),\theta_\lambda(0)) \rightarrow (\mu_0,\theta_0) 
    \quad\text{weakly in $\mathcal{H}^1$, strongly in $\mathcal{L}^2$}
\end{align}
as $\lambda\rightarrow 0$.

\textbf{Step 6: Higher order uniform estimates.} Now, we establish the main estimates for the regularity argument, which will be uniform with respect to $\lambda\in(0,\lambda_\ast)$. To this end, we test \eqref{WF:DELT:MT} with $(\delt\phi_\lambda, \delt\psi_\lambda)\in\mathcal{H}^1_{K,\alpha}$ obtaining
\begin{align}\label{HigherOrder1}
	\begin{split}
		&\bigang{(\delt\mu_\lambda, \delt\theta_\lambda)}{(\delt\phi_\lambda, \delt\psi_\lambda)}_{\mathcal{H}^1_{K,\alpha}} \\
		&\quad = \intO \abs{\Grad\delt\phi_\lambda}^2 + F_\lambda^{\prime\prime}(\phi_\lambda)\abs{\delt\phi_\lambda}^2 \dx + \intG \abs{\Gradg\delt\psi}^2 + G_\lambda^{\prime\prime}(\psi_\lambda)\abs{\delt\psi_\lambda}^2\dG \\
		&\qquad + \sigma(K) \intG \abs{\alpha\delt\psi_\lambda - \delt\phi_\lambda}^2\dG \\
		&\quad = \norm{(\delt\phi_\lambda, \delt\psi_\lambda)}_{K,\alpha}^2 + \intO F_\lambda^{\prime\prime}(\phi_\lambda)\abs{\delt\phi_\lambda}^2 \dx + \intG G_\lambda^{\prime\prime}(\psi_\lambda)\abs{\delt\psi_\lambda}^2\dG.
	\end{split}
\end{align}
On the other hand, using the identities \eqref{STR:PHI:EP}-\eqref{STR:PSI:EP} for $(\delt\phi_\lambda, \delt\psi_\lambda)$, as well as the chain rule \eqref{CR:MUTHETA_:lambda}, we find
\begin{align}\label{HigherOrder2}
	\begin{split}
		&\bigang{(\delt\mu_\lambda, \delt\theta_\lambda)}{(\delt\phi_\lambda, \delt\psi_\lambda)}_{\mathcal{H}^1_{K,\alpha}} \\
		&\quad = \bigang{(\delt\mu_\lambda, \delt\theta_\lambda)}{(\Lap\mu_\lambda - \Grad\phi_\lambda\cdot\boldsymbol{v}_\lambda, \Lapg\theta_\lambda - \Gradg\psi_\lambda\cdot\boldsymbol{w}_\lambda - \beta\deln\mu_\lambda)}_{\mathcal{H}^1_{K,\alpha}} \\
		&\quad = - \bigang{(\delt\mu_\lambda, \delt\theta_\lambda)}{(-\Lap\mu_\lambda, -\Lapg\theta_\lambda + \beta\deln\mu_\lambda)}_{\mathcal{H}^1_{K,\alpha}} \\
        &\qquad - \bigang{(\delt\mu_\lambda, \delt\theta_\lambda)}{(\Grad\phi_\lambda\cdot\boldsymbol{v}_\lambda, \Gradg\psi_\lambda\cdot\boldsymbol{w}_\lambda)}_{\mathcal{H}^1_{K,\alpha}} \\
		&\quad = - \ddt\frac12\norm{(\mu_\lambda, \theta_\lambda)}_{L,\beta}^2 - \bigang{(\delt\mu_\lambda, \delt\theta_\lambda)}{(\Grad\phi_\lambda\cdot\boldsymbol{v}_\lambda, \Gradg\psi_\lambda\cdot\boldsymbol{w}_\lambda)}_{\mathcal{H}^1_{K,\alpha}}.
	\end{split}
\end{align}
Now, for the second term on the right-hand side, we note again that $(\Grad\phi_\lambda\cdot\boldsymbol{v}_\lambda, \Gradg\psi_\lambda\cdot\boldsymbol{w}_\lambda)\in L^2(0,T;\mathcal{H}^1_{K,\alpha})$. This allows us to choose this pair as a test function in the weak formulation \eqref{WF:DELT:MT}. In this way, we obtain
\begin{align}\label{HigherOrder3}
    \begin{split}
        &\bigang{(\delt\mu_\lambda, \delt\theta_\lambda)}{(\Grad\phi_\lambda\cdot\boldsymbol{v}_\lambda, \Gradg\psi_\lambda\cdot\boldsymbol{w}_\lambda)}_{\mathcal{H}^1_{K,\alpha}}\\
         &\quad = \bigscp{(\delt\phi_\lambda, \delt\psi_\lambda)}{(\Grad\phi_\lambda\cdot\boldsymbol{v}_\lambda, \Gradg\psi_\lambda\cdot\boldsymbol{w}_\lambda)}_{K,\alpha} \\
         &\qquad + \intO F_\lambda^{\prime\prime}(\phi_\lambda)\delt\phi_\lambda\Grad\phi_\lambda\cdot\boldsymbol{v}_\lambda \dx + \intG G_\lambda^{\prime\prime}(\psi_\lambda)\delt\psi_\lambda\Gradg\psi_\lambda\cdot\boldsymbol{w}_\lambda\dG \\ 
         &\quad = \bigscp{(\delt\phi_\lambda, \delt\psi_\lambda)}{(\Grad\phi_\lambda\cdot\boldsymbol{v}_\lambda, \Gradg\psi_\lambda\cdot\boldsymbol{w}_\lambda)}_{K,\alpha} \\
         &\qquad + \intO \Grad(F_\lambda^\prime(\phi_\lambda))\cdot\boldsymbol{v}_\lambda\delt\phi_\lambda\dx + \intG \Gradg(G_\lambda^\prime(\psi_\lambda))\cdot\boldsymbol{w}_\lambda\delt\psi_\lambda\dG \\
         &\quad = \bigscp{(\delt\phi_\lambda, \delt\psi_\lambda)}{(\Grad\phi_\lambda\cdot\boldsymbol{v}_\lambda, \Gradg\psi_\lambda\cdot\boldsymbol{w}_\lambda)}_{K,\alpha} \\
         &\qquad - \intO F_\lambda^\prime(\phi_\lambda)\Grad\delt\phi_\lambda\cdot\boldsymbol{v}_\lambda\dx - \intG G_\lambda^\prime(\psi_\lambda)\Gradg\delt\psi_\lambda\cdot\boldsymbol{w}_\lambda\dG.
     \end{split}
\end{align}
Here we used integration by parts along with the facts that $\Div\;\boldsymbol{v}_\lambda = 0$ a.e.~in $Q$, $\Divg\;\boldsymbol{w}_\lambda = 0$ and $\boldsymbol{v}_\lambda\cdot\n = \boldsymbol{w}_\lambda\cdot\n = 0$ a.e.~on $\Sigma$. Thus, collecting \eqref{HigherOrder1}-\eqref{HigherOrder3}, we find that
\begin{align}\label{EST:1}
	\begin{split}
		&\ddt\frac12\norm{(\mu_\lambda, \theta_\lambda)}_{L,\beta}^2 + \norm{(\delt\phi_\lambda, \delt\psi_\lambda)}_{K,\alpha}^2 + \intO F_{1,\lambda}^{\prime\prime}(\phi_\lambda)\abs{\delt\phi_\lambda}^2\dx + \intG G_{1,\lambda}^{\prime\prime}(\psi_\lambda)\abs{\delt\psi_\lambda}^2\dG 
        \\
		&\quad = - \bigscp{(\delt\phi_\lambda, \delt\psi_\lambda)}{(\Grad\phi_\lambda\cdot\boldsymbol{v}_\lambda, \Gradg\psi_\lambda\cdot\boldsymbol{w}_\lambda)}_{K,\alpha} + \intO F_\lambda^\prime(\phi_\lambda)\Grad\delt\phi_\lambda\cdot\boldsymbol{v}_\lambda\dx 
        \\
        &\qquad + \intG G_\lambda^\prime(\psi_\lambda)\Gradg\delt\psi_\lambda\cdot\boldsymbol{w}_\lambda\dG - \intO F_2^{\prime\prime}(\phi_\lambda)\abs{\delt\phi_\lambda}^2\dx - \intG G_2^{\prime\prime}(\psi_\lambda)\abs{\delt\psi_\lambda}^2\dG. 
	\end{split}
 \end{align}
We immediately notice that the last two terms on the left-hand side of \eqref{EST:1} are non-negative due to \ref{Yosida:Convexity}. Moreover, applying the bulk-surface Poincar\'{e} inequality (Lemma~\ref{Prelim:Poincare}) immediately yields the inequalities
\begin{align*}
	\norm{(\delt\phi_\lambda, \delt\psi_\lambda)}_{L,\beta} \leq C\norm{(\delt\phi_\lambda, \delt\psi_\lambda)}_{\mathcal{H}^1} \leq C\norm{(\delt\phi_\lambda, \delt\psi_\lambda)}_{K,\alpha}.
\end{align*}
Then, using this inequality, the Lipschitz continuity of $F_2^{\prime}$ and $G_2^\prime$, the weak formulation for $(\delt\phi_\lambda, \delt\psi_\lambda)$ as well as the Sobolev embedding $\mathcal{W}^{2,6}\emb\mathcal{L}^\infty$, the last two terms on the right-hand side of \eqref{EST:1} can be bounded as follows:
\begin{align}\label{Est:delt:pp:L^2}
    \begin{split}
    	&- \intO F_2^{\prime\prime}(\phi_\lambda)\abs{\delt\phi_\lambda}^2\dx - \intG G_2^{\prime\prime}(\psi_\lambda)\abs{\delt\psi_\lambda}^2\dG \leq C\norm{(\delt\phi_\lambda, \delt\psi_\lambda)}_{\mathcal{L}^2}^2 
        \\
        &\quad = C \intO \phi_\lambda\boldsymbol{v}_\lambda\cdot\Grad\delt\phi_\lambda\dx 
            + C\intG \psi_\lambda\boldsymbol{w}_\lambda\cdot\Gradg\delt\psi_\lambda\dG 
            \\
    	&\qquad - C \bigscp{(\mu_\lambda, \theta_\lambda)}{(\delt\phi_\lambda, \delt\psi_\lambda)}_{L,\beta} 
        \\
    	&\quad\leq C \norm{(\phi_\lambda, \psi_\lambda)}_{\mathcal{L}^\infty}
            \norm{(\delt\phi_\lambda, \delt\psi_\lambda)}_{K,\alpha}
            \norm{(\boldsymbol{v}_\lambda, \boldsymbol{w}_\lambda)}_{\mathcal{L}^2} 
            + C\norm{(\mu_\lambda, \theta_\lambda)}_{L,\beta}
            \norm{(\delt\phi_\lambda, \delt\psi_\lambda)}_{L,\beta} 
        \\
    	&\quad\leq \frac16\norm{(\delt\phi_\lambda, \delt\psi_\lambda)}_{K,\alpha}^2 
            + C\norm{(\phi_\lambda, \psi_\lambda)}_{\mathcal{W}^{2,6}}^2
            \norm{(\boldsymbol{v}_\lambda, \boldsymbol{w}_\lambda)}_{\mathcal{L}^2}^2 
            + C\norm{(\mu_\lambda, \theta_\lambda)}_{L,\beta}^2 
        \\
        &\quad\leq \frac16\norm{(\delt\phi_\lambda,\delt\psi_\lambda)}_{K,\alpha}^2 
            + C\norm{(\boldsymbol{v}_\lambda,\boldsymbol{w}_\lambda)}_{\mathcal{L}^2}^2 
            + C\big(1 + \norm{(\boldsymbol{v}_\lambda,\boldsymbol{w}_\lambda)}_{\mathcal{L}^2}^2\big)
            \norm{(\mu_\lambda,\theta_\lambda)}_{L,\beta}^2.
    \end{split}
\end{align}
Our next step is to bound the first term on the right-hand side of \eqref{EST:1}. 
In view of \eqref{Est:ApproxProblem:PP:W^26}, we first use the Sobolev embedding $W^{1,6}(\Om)\emb L^\infty(\Om)$ to deduce that
\begin{align}\label{Est:delt:pp:Ka:phi}
	\begin{split}
		&\abs{\intO \Grad\delt\phi_\lambda\cdot\Grad(\Grad\phi_\lambda\cdot\boldsymbol{v}_\lambda)\dx} \\
		&\quad\leq \Big(\norm{\phi_\lambda}_{W^{2,6}(\Om)}\norm{\boldsymbol{v}_\lambda}_{L^3(\Om)} + \norm{\Grad\phi_\lambda}_{L^\infty(\Om)}\norm{\boldsymbol{v}_\lambda}_{H^1(\Om)}\Big)\norm{\Grad\delt\phi_\lambda}_{L^2(\Om)} \\
		&\quad \leq \frac16 \norm{\Grad\delt\phi_\lambda}_{L^2(\Om)}^2 + C\norm{\phi_\lambda}_{W^{2,6}(\Om)}^2\norm{\boldsymbol{v}_\lambda}_{H^1(\Om)}^2 \\
		&\quad\leq \frac16 \norm{\Grad\delt\phi_\lambda}_{L^2(\Om)}^2 +  C\norm{\boldsymbol{v}_\lambda}_{H^1(\Om)}^2
		\big(1 + \norm{(\mu_\lambda, \theta_\lambda)}_{L,\beta}^2\big).
	\end{split}
\end{align}
Analogously, exploiting again \eqref{Est:ApproxProblem:PP:W^26} and the embedding $W^{1,6}(\Ga)\emb L^\infty(\Ga)$, we obtain
\begin{align}\label{Est:delt:pp:Ka:psi}
	\begin{split}
		&\abs{\intG \Gradg\delt\psi_\lambda\cdot\Gradg(\Gradg\psi_\lambda\cdot\boldsymbol{w}_\lambda)\dG} \\
		&\quad\leq \frac16\norm{\Gradg\delt\psi_\lambda}_{L^2(\Ga)}^2 + C\norm{\boldsymbol{w}_\lambda}_{H^1(\Ga)}^2
		\big(1 + \norm{(\mu_\lambda, \theta_\lambda)}_{L,\beta}^2\big).
	\end{split}
\end{align}
Next, we use \eqref{Est:ApproxProblem:PP:W^26} in combination with the trace theorem $H^1(\Om)\emb L^4(\Ga)$ and infer that
\begin{align}\label{Est:delt:pp:Ka:boundary}
	\begin{split}
		&\abs{\sigma(K) \intG (\alpha\delt\psi_\lambda - \delt\phi_\lambda)(\alpha\Gradg\psi_\lambda\cdot\boldsymbol{w}_\lambda - \Grad\phi_\lambda\cdot\boldsymbol{v}_\lambda)\dG} \\
		&\quad\leq \sigma(K)\Big(\norm{\alpha\Gradg\psi_\lambda}_{L^4(\Ga)}\norm{\boldsymbol{w}_\lambda}_{L^4(\Ga)} + \norm{\Grad\phi_\lambda}_{L^4(\Ga)}\norm{\boldsymbol{v}_\lambda}_{L^4(\Ga)}\Big)\norm{\alpha\delt\psi_\lambda - \delt\phi_\lambda}_{L^2(\Ga)}\\
		&\quad\leq \frac16\sigma(K)\norm{\alpha\delt\psi_\lambda - \delt\phi_\lambda}_{L^2(\Ga)}^2 + C\sigma(K)\norm{(\boldsymbol{v}_\lambda, \boldsymbol{w}_\lambda)}_{\mathcal{H}^1}^2\norm{(\phi_\lambda, \psi_\lambda)}_{\mathcal{H}^2}^2 \\
		&\quad\leq \frac16\sigma(K)\norm{\alpha\delt\psi_\lambda - \delt\phi_\lambda}_{L^2(\Ga)}^2 + C\sigma(K)\norm{(\boldsymbol{v}_\lambda, \boldsymbol{w}_\lambda)}_{\mathcal{H}^1}^2
		\big(1 + \norm{(\mu_\lambda, \theta_\lambda)}_{L,\beta}^2\big).
	\end{split}
\end{align}
Recalling the definition of the inner product $(\cdot,\cdot)_{K,\alpha}$, we combine \eqref{Est:delt:pp:Ka:phi}-\eqref{Est:delt:pp:Ka:boundary} to conclude that
\begin{align}\label{Est:delt:pp:Ka}
	\begin{split}
		&\abs{\bigscp{(\delt\phi_\lambda, \delt\psi_\lambda)}{(\Grad\phi_\lambda\cdot\boldsymbol{v}_\lambda, \Gradg\psi_\lambda\cdot\boldsymbol{w}_\lambda)}_{K,\alpha} }
        \\
        &\quad \le \frac16\norm{(\delt\phi_\lambda,\delt\psi_\lambda)}_{K,\alpha}^2 
            + C\norm{(\boldsymbol{v}_\lambda, \boldsymbol{w}_\lambda)}_{\mathcal{H}^1}^2
		      \big(1 + \norm{(\mu_\lambda, \theta_\lambda)}_{L,\beta}^2\big)
	\end{split}
\end{align}
Furthermore, by means of \eqref{Est:ApproxProblem:MT:POT:H^1:L^6}, the second and the third term on the right-hand side of \eqref{EST:1} can be bounded as
\begin{align}
    \label{Est:pot:term}
	\begin{split}
		&\abs{\intO F_\lambda^\prime(\phi_\lambda)\Grad\delt\phi_\lambda\cdot\boldsymbol{v}_\lambda\dx + \intG G_\lambda^\prime(\psi_\lambda)\Gradg\delt\psi_\lambda\cdot\boldsymbol{w}_\lambda\dG} \\
		&\quad\leq \frac16 \norm{(\delt\phi_\lambda, \delt\psi_\lambda)}_{K,\alpha}^2 + C\norm{(F_\lambda^\prime(\phi_\lambda), G_\lambda^\prime(\psi_\lambda))}_{\mathcal{L}^6}^2\norm{(\boldsymbol{v}_\lambda, \boldsymbol{w}_\lambda)}_{\mathcal{H}^1}^2 \\
        &\quad\leq \frac16\norm{(\delt\phi_\lambda,\delt\psi_\lambda)}_{K,\alpha}^2 + C\big(1 + \norm{(\mu_\lambda,\theta_\lambda)}_{L,\beta}^2\big)\norm{(\boldsymbol{v}_\lambda,\boldsymbol{w}_\lambda)}_{\mathcal{H}^1}^2.
	\end{split}
\end{align}
Finally, collecting the estimates \eqref{Est:delt:pp:L^2}, \eqref{Est:delt:pp:Ka} and \eqref{Est:pot:term} we end up with the differential inequality
\begin{align}
    \label{Est:ApproxProblem:Gronwall2}
	\begin{split}
		&\ddt\frac12\norm{(\mu_\lambda, \theta_\lambda)}_{L,\beta}^2 + \frac12\norm{(\delt\phi_\lambda, \delt\psi_\lambda)}_{K,\alpha}^2 \\
		&\quad\leq C\norm{(\boldsymbol{v}_\lambda, \boldsymbol{w}_\lambda)}_{\mathcal{H}^1}^2 + C\big(1 + \norm{(\boldsymbol{v}_\lambda, \boldsymbol{w}_\lambda)}_{\mathcal{H}^1}^2\big)\norm{(\mu_\lambda, \theta_\lambda)}_{L,\beta}^2.
	\end{split}
\end{align}
We are now in the position to apply Gronwall's lemma to \eqref{Est:ApproxProblem:Gronwall2}, which yields
\begin{align}
    \label{Est:ApproxProblem:Gronwall3}
    \begin{split}
        \norm{(\mu_\lambda(t), \theta_\lambda(t))}_{L,\beta}^2 &\leq \norm{(\mu_\lambda(0), \theta_\lambda(0))}_{L,\beta}^2\exp\Big(C\int_0^t \big(1 + \norm{(\boldsymbol{v}_\lambda, \boldsymbol{w}_\lambda)}_{\mathcal{H}^1}^2\big)\dtau\Big) \\
        &\quad + C\int_0^t \norm{(\boldsymbol{v}_\lambda, \boldsymbol{w}_\lambda)}_{\mathcal{H}^1}^2\exp\Big(C\int_s^t \big(1 + \norm{(\boldsymbol{v}_\lambda, \boldsymbol{w}_\lambda)}_{\mathcal{H}^1}^2\big)\dtau\Big)\ds
    \end{split}
\end{align}
for almost all $t\in[0,T]$. From the convergence in \eqref{Conv:Init:MT} we deduce that
\begin{align}\label{Est:ApproxIC:MT:H^1}
    \norm{(\mu_\lambda(0), \theta_\lambda(0))}_{L,\beta} \leq C\big( 1 + \norm{(\mu_0,\theta_0)}_{L,\beta}\big).
\end{align}
Combining \eqref{Est:ApproxProblem:Gronwall2}-\eqref{Est:ApproxIC:MT:H^1} with \eqref{Est:vw:ep}, we obtain
\begin{align}\label{Est:MT:Lb:uniform}
    \sup_{t\in[0,T]}\norm{(\mu_\lambda(t),\theta_\lambda(t))}_{L,\beta}^2 \leq C
\end{align}
for all $\lambda\in(0,\lambda_\ast)$. Then, integrating \eqref{Est:ApproxProblem:Gronwall3} in time from $0$ to $T$ and using the estimates \eqref{Est:vw:ep}, \eqref{Est:ApproxIC:MT:H^1}-\eqref{Est:MT:Lb:uniform}, we have
\begin{align}\label{Est:PP:Ka}
    \int_0^T \norm{(\delt\phi_\lambda,\delt\psi_\lambda)}_{K,\alpha}^2\ds \leq C.
\end{align}
As $K\in[0,\infty)$, the bulk-surface Poincar\'{e} inequality (Lemma~\ref{Prelim:Poincare}) directly yields
\begin{align}\label{Est:delt:pp:H^1:uniform}
    \norm{(\delt\phi_\lambda,\delt\psi_\lambda)}_{L^2(0,T;\mathcal{H}^1)} \leq C.
\end{align}
Using \eqref{Est:ApproxProblem:PP:W^26} and \eqref{Est:ApproxProblem:MT:POT:H^1:L^6} in combination with \eqref{Est:MT:Lb:uniform}, we further deduce that
\begin{align}\label{Est:pp:pot:W^26:inf:uniform}
    \norm{(\phi_\lambda,\psi_\lambda)}_{L^\infty(0,T;\mathcal{W}^{2,6})} 
    + \norm{(F_{1,\lambda}^\prime(\phi_\lambda),G_{1,\lambda}^\prime(\psi_\lambda))}_{L^\infty(0,T;\mathcal{L}^6)}
    \leq C.
\end{align}
Next, employing elliptic regularity theory for systems with bulk-surface coupling (see, e.g., \cite[Theorem~3.3]{Knopf2021}) in combination with the Sobolev embedding $\mathcal{W}^{2,6}\emb\mathcal{W}^{1,\infty}$, \eqref{Est:ApproxProblem:mean:MT:K}, \eqref{Est:MT:Lb:uniform} and \eqref{Est:pp:pot:W^26:inf:uniform}, we obtain
\begin{align}
    \label{EST:MT:H2}
    \begin{split}
    &\norm{(\mu_\lambda(t),\theta_\lambda(t))}_{\mathcal{H}^2} \\
    &\quad\leq C\norm{(-\delt\phi_\lambda(t) - \Grad\phi_\lambda(t)\cdot\boldsymbol{v}_\lambda(t),-\delt\psi_\lambda(t) - \Gradg\psi_\lambda(t)\cdot\boldsymbol{w}_\lambda)(t)}_{\mathcal{L}^2} + C\abs{\mean{\mu_\lambda}{\theta_\lambda}}\\
    &\quad\leq C\norm{(\delt\phi_\lambda(t),\delt\psi_\lambda(t))}_{\mathcal{L}^2} + C\norm{(\Grad\phi_\lambda(t),\Gradg\psi_\lambda(t))}_{\mathcal{L}^\infty}\norm{(\boldsymbol{v}_\lambda(t),\boldsymbol{w}_\lambda(t))}_{\mathcal{L}^2} \\
    &\qquad + C\big(1 + \norm{(\mu_\lambda,\theta_\lambda)}_{L,\beta}\big) \\
    &\quad\leq C\big(1 + \norm{(\delt\phi_\lambda(t),\delt\psi_\lambda(t))}_{\mathcal{L}^2} + \norm{(\phi_\lambda(t),\psi_\lambda(t))}_{\mathcal{W}^{2,6}}\norm{(\boldsymbol{v}_\lambda(t),\boldsymbol{w}_\lambda(t))}_{\mathcal{L}^2}\big) \\
    &\quad\leq C\big(1 + \norm{(\delt\phi_\lambda(t),\delt\psi_\lambda(t))}_{\mathcal{L}^2} + \norm{(\boldsymbol{v}_\lambda(t),\boldsymbol{w}_\lambda(t))}_{\mathcal{L}^2}\big)
    \end{split}
\end{align}
for almost all $t\in[0,T]$. Integrating the previous inequality, we readily conclude that
\begin{align}\label{Est:mt:H^2:uniform}
    \norm{(\mu_\lambda,\theta_\lambda)}_{L^2(0,T;\mathcal{H}^2)} \leq C
\end{align}
by means of \eqref{Est:vw:ep} and \eqref{Est:delt:pp:H^1:uniform}.
Invoking \eqref{Est:ApproxProblem:MT:L^2:K}, this further entails
\begin{align}\label{Est:mt:H^1:inf:uniform}
    \norm{(\mu_\lambda,\theta_\lambda)}_{L^\infty(0,T;\mathcal{H}^1)} \leq C.
\end{align}
Next, as a consequence of \eqref{Est:pp:pot:W^26:inf:uniform}, we see that
\begin{align}\label{Est:pp:Linfty:uniform}
    \norm{(\phi_\lambda,\psi_\lambda)}_{L^\infty(0,T;\mathcal{L}^\infty)} \leq C.
\end{align}
Thus, by \eqref{Est:vw:ep}, \eqref{Est:MT:Lb:uniform} and \eqref{Est:pp:Linfty:uniform}, we use a comparison argument to conclude that
\begin{align}\label{Est:delt:pp:inf:uniform}
    \norm{(\delt\phi_\lambda,\delt\psi_\lambda)}_{L^\infty(0,T;(\mathcal{H}^1)^\prime)} \leq C.
\end{align}
Applying once more elliptic regularity theory for systems with bulk-surface coupling (see, e.g., \cite[Theorem~3.3]{Knopf2021}), we obtain
\begin{align*}
    &\norm{(\mu_\lambda,\theta_\lambda)}_{\mathcal{H}^3} \\
    &\quad\leq C\norm{(-\delt\phi_\lambda - \boldsymbol{v}_\lambda\cdot\Grad\phi_\lambda,-\delt\psi_\lambda - \boldsymbol{w}_\lambda\cdot\Gradg\psi_\lambda)}_{\mathcal{H}^1} + C\abs{\mean{\mu_\lambda}{\theta_\lambda}}
    \\
    &\quad\leq C\big(\norm{(\delt\phi_\lambda,\delt\psi_\lambda)}_{\mathcal{H}^1} + \norm{(\phi_\lambda,\psi_\lambda)}_{\mathcal{W}^{1,\infty}}\norm{(\boldsymbol{v}_\lambda,\boldsymbol{w}_\lambda)}_{\mathcal{H}^1}\big) + C\big(1 + \norm{(\mu_\lambda,\theta_\lambda)}_{L,\beta}\big)
    \\
    &\quad\leq C\big(1 + \norm{(\delt\phi_\lambda,\delt\psi_\lambda)}_{K,\alpha} + \norm{(\boldsymbol{v}_\lambda,\boldsymbol{w}_\lambda)}_{\mathcal{H}^1}\big)
\end{align*}
almost everywhere on $[0,T]$.
Integrating this estimate in time from $0$ to $T$ and combining the resulting equation with \eqref{Est:vw:ep} and \eqref{Est:PP:Ka}, we conclude that 
\begin{align}\label{Est:PP:KaMT}
    \int_0^T \norm{(\mu_\lambda,\theta_\lambda)}_{\mathcal{H}^3}^2\ds \leq C.
\end{align}

\textbf{Step 7: Passage to the limit and the existence of a strong solution.} 
Based on the uniform estimates \eqref{Est:delt:pp:H^1:uniform}, \eqref{Est:pp:pot:W^26:inf:uniform}, \eqref{Est:mt:H^1:inf:uniform}, \eqref{Est:delt:pp:inf:uniform} and \eqref{Est:PP:KaMT},
we proceed similarly as in the proof of Theorem~\ref{THEOREM:EOWS} and use standard weak and strong compactness results to infer the existence of a quadruplet $(\phi,\psi,\mu,\theta)$ of functions satisfying
\begin{align*}
    (\delt\phi,\delt\psi) &\in L^\infty(0,T;(\mathcal{H}^1_{L,\beta})^\prime) \cap L^2(0,T;\mathcal{H}^1), \\
    (\phi,\psi)&\in L^\infty(0,T;\mathcal{W}^{2,6})\cap \Big(C(\overline{Q})\times C(\overline\Sigma)\Big), \\
    (\mu,\theta)&\in L^\infty(0,T;\mathcal{H}^1_{L,\beta})\cap L^2(0,T;\mathcal{H}^3), \\
    (F^\prime(\phi), G^\prime(\psi))&\in L^\infty(0,T;\mathcal{L}^6),
\end{align*}
such that
\begin{alignat*}{2}
    (\delt\phi_\lambda,\delt\psi_\lambda) &\to (\delt\phi,\delt\psi)
    &&\quad\text{weakly-$\ast$ in $L^\infty(0,T;(\mathcal{H}^1_{L,\beta})^\prime)$, \
        weakly in $L^2(0,T;\mathcal{H}^1)$},
    \\
    (\phi_\lambda,\psi_\lambda) &\to (\phi,\psi)
    &&\quad\text{weakly-$\ast$ in $L^\infty(0,T;\mathcal{W}^{2,6})$, \
        strongly in $C(\overline{Q})\times C(\overline\Sigma)$},
    \\
    (\mu_\lambda,\theta_\lambda) &\to (\mu,\theta)
    &&\quad\text{weakly-$\ast$ in $L^\infty(0,T;\mathcal{H}^3)$, \
    weakly in $L^2(0,T;\mathcal{H}^1_{L,\beta})$},
    \\
    \big(F_\lambda'(\phi_\lambda),G_\lambda'(\psi_\lambda)\big) &\to (F^\prime(\phi), G^\prime(\psi))
    &&\quad\text{weakly-$\ast$ in $L^\infty(0,T;\mathcal{L}^6)$}
\end{alignat*}  
up to subsequence extraction.
Recall here that $\mathcal{H}^1_{L,\beta} = \mathcal{H}^1$ as we only consider $L\in (0,\infty]$.
Together with the convergence of the approximate initial data (see \eqref{Conv:Init:PP} and \eqref{ID:INI}) and the approximate velocity fields (see \eqref{Conv:vw:ep}), it is straightforward to verify that the quadruplet $(\phi,\psi,\mu,\theta)$ is the weak solution of System~\eqref{EQ:SYSTEM} in the sense of Definition~\ref{DEF:SING:WS} to the initial data $(\phi_0,\psi_0)$ and the velocity fields $(\boldsymbol{v},\boldsymbol{w})$.
Due to the regularities we have established, it is clear that $(\phi,\psi,\mu,\theta)$
actually satisfies System~\eqref{EQ:SYSTEM} almost everywhere. 
The remaining regularity assertion $(F^\prime(\phi),G^\prime(\psi))\in L^2(0,T;\mathcal{L}^\infty)$ can be established as in \cite[Proposition~7.1 and Proposition~7.2]{Knopf2024a}.

\textbf{Step 8: Refinement of the uniform estimates.} As a last step, we show the validity of \eqref{HighReg:Est:MT:Lb:sup} and \eqref{HighReg:Est:PP:Ka:L^2} by refining the estimates from Step 4 and Step 6. In fact, so far, the constants in the estimates we have proven still depend explicitly on the time $T$ as well as the norms of the velocity fields. To this end, we first notice that
\begin{align}\label{PP:Lambda:uniform}
    \abs{\phi_\lambda} \leq 2 \quad\text{a.e.~in~}Q, \qquad 
    \abs{\psi_\lambda} \leq 2 \quad\text{a.e.~on~}\Sigma
\end{align}
for all $\lambda\in(0,\lambda_\star)$, where $\lambda_\star$ is sufficiently small. Indeed, since the limiting solution satisfies $\abs{\phi} < 1$ a.e. in $Q$ as well as $\abs{\psi} < 1$ a.e. on $\Sigma$, the claim readily follows from the uniform convergence
\begin{align*}
    (\phi_\lambda,\psi_\lambda) \rightarrow (\phi,\psi) \qquad\text{strongly in~} C(\overline{Q})\times C(\overline{\Sigma})
\end{align*}
as $\lambda\rightarrow 0$. Therefore, in the following, we assume that $\lambda\in(0,\lambda_\star)$. Having \eqref{PP:Lambda:uniform} at hand, we can go back to the approximated energy inequality \eqref{EI:ApproxProblem} and bound the convective terms as follows:
\begin{align*}
    &\abs{\int_0^t\intO \phi_\lambda\boldsymbol{v}_\lambda\cdot\Grad\mu_\lambda\dxs + \int_0^t\intG \psi_\lambda\boldsymbol{w}_\lambda\cdot\Gradg\theta_\lambda\dGs} \\
    &\quad\leq \frac12\int_0^t\intO \abs{\Grad\mu_\lambda}^2\dxs + \frac12 \int_0^t\intG \abs{\Gradg\theta_\lambda}^2\dGs + \int_0^t\intO \abs{\boldsymbol{v}_\lambda}^2\dxs + \int\intO \abs{\boldsymbol{w}_\lambda}^2\dGs.
\end{align*}
Estimating the approximated initial data as before, we thus have
\begin{align}\label{Est:Gronwall:Refined}
    \begin{split}
        &\norm{(\phi_\lambda,\psi_\lambda)}_{L^\infty(0,T;\mathcal{H}^1)}^2 + \norm{(F_\lambda(\phi_\lambda),G_\lambda(\psi_\lambda))}_{L^\infty(0,T;\mathcal{L}^1)} + \int_0^T \norm{(\mu_\lambda,\theta_\lambda)}_{L,\beta}^2\ds\\
        &\quad \leq C_0\big(1 + \norm{(\boldsymbol{v}_\lambda,\boldsymbol{w}_\lambda)}_{L^2(0,T;\mathcal{L}^2)}^2\big).
    \end{split}
\end{align}
Then, the computations presented in \eqref{Est:ApproxProblem:MZ}-\eqref{Est:Approx:H^2+trace:Phi} can be repeated, and we now obtain, using \eqref{Est:Gronwall:Refined} instead of \eqref{Est:ApproxProblem:uniform:1},
\begin{align}\label{Est:ApproxProblem:MT:POT:H^1:L^6:refined}
    \begin{split}
        &\norm{(\phi_\lambda,\theta_\lambda)}_{\mathcal{W}^{2,6}} + \norm{(F_{1,\lambda}^\prime(\phi_\lambda), G_{1,\lambda}^\prime(\psi_\lambda))}_{\mathcal{L}^6} + \norm{(\mu_\lambda,\theta_\lambda)}_{\mathcal{H}^1} \\
        &\quad\leq C_0\big(1 + \norm{(\boldsymbol{v}_\lambda,\boldsymbol{w}_\lambda)}_{L^2(0,T;\mathcal{L}^2)} + \norm{(\mu_\lambda,\theta_\lambda)}_{L,\beta}\big).
    \end{split}
\end{align}
Then, as before, we can prove the higher order regularity as in Step 5, and eventually arrive at
\begin{align}\label{EST:1:refined}
	\begin{split}
		&\ddt\frac12\norm{(\mu_\lambda, \theta_\lambda)}_{L,\beta}^2 + \norm{(\delt\phi_\lambda, \delt\psi_\lambda)}_{K,\alpha}^2
        \\
		&\quad \leq - \bigscp{(\delt\phi_\lambda, \delt\psi_\lambda)}{(\Grad\phi_\lambda\cdot\boldsymbol{v}_\lambda, \Gradg\psi_\lambda\cdot\boldsymbol{w}_\lambda)}_{K,\alpha} + \intO F_\lambda^\prime(\phi_\lambda)\Grad\delt\phi_\lambda\cdot\boldsymbol{v}_\lambda\dx 
        \\
        &\qquad + \intG G_\lambda^\prime(\psi_\lambda)\Gradg\delt\psi_\lambda\cdot\boldsymbol{w}_\lambda\dG - \intO F_2^{\prime\prime}(\phi_\lambda)\abs{\delt\phi_\lambda}^2\dx - \intG G_2^{\prime\prime}(\psi_\lambda)\abs{\delt\psi_\lambda}^2\dG. 
	\end{split}
 \end{align}
 The first term on the right-hand side of \eqref{EST:1:refined} can be estimated as in \eqref{Est:delt:pp:Ka}, while the second and third term are bounded as in \eqref{Est:pot:term}.
Then, for the last two terms we use again the weak formulation for $(\delt\phi_\lambda,\delt\psi_\lambda)$ to infer that
\begin{align}\label{Est:delt:pp:L^2_refined}
        &\abs{- \intO F_2^{\prime\prime}(\phi_\lambda)\abs{\delt\phi_\lambda}^2\dx - \intG G_2^{\prime\prime}(\psi_\lambda)\abs{\delt\psi_\lambda}^2\dG} \nonumber \\
        &\quad\leq C_0\norm{(\delt\phi_\lambda,\delt\psi_\lambda)}_{\mathcal{L}^2} \nonumber \\
        &\quad = C_0\intO \phi_\lambda\boldsymbol{v}_\lambda\cdot\Gradg\delt\phi_\lambda\dx + C_0\intG \psi_\lambda\boldsymbol{w}_\lambda\cdot\Gradg\delt\psi_\lambda - C_0\big((\mu_\lambda˛\theta_\lambda),(\delt\phi_\lambda,\delt\psi_\lambda)\big)_{L,\beta} \\
        &\quad\leq C_0\norm{(\phi_\lambda,\psi_\lambda)}_{\mathcal{L}^\infty}\norm{(\delt\phi_\lambda,\delt\psi_\lambda)}_{K,\alpha}\norm{(\boldsymbol{v}_\lambda,\boldsymbol{w}_\lambda)}_{\mathcal{L}^2} + C_0\norm{(\mu_\lambda,\theta_\lambda)}_{L,\beta}\norm{(\delt\phi_\lambda,\delt\psi_\lambda)}_{K,\alpha} \nonumber \\
        &\quad\leq \frac16\norm{(\delt\phi_\lambda,\delt\psi_\lambda)}_{K,\alpha}^2 + C_0\big(\norm{(\boldsymbol{v}_\lambda,\boldsymbol{w}_\lambda)}_{\mathcal{L}^2}^2 + \norm{(\mu_\lambda,\theta_\lambda)}_{L,\beta}^2\big) + C_0\norm{(\boldsymbol{v}_\lambda,\boldsymbol{w}_\lambda)}_{\mathcal{L}^2}^2\norm{(\mu_\lambda,\theta_\lambda)}_{L,\beta}^2, \nonumber
\end{align}
where we made use of the bounds \eqref{PP:Lambda:uniform}. Thus, combining the estimates \eqref{Est:delt:pp:Ka} and \eqref{Est:pot:term} with the now refined estimate \eqref{Est:delt:pp:L^2_refined}, we find the differential inequality
\begin{align}\label{PreGronwall:refined}
    \begin{split}
        &\ddt\frac12\norm{(\mu_\lambda,\theta_\lambda)}_{L,\beta}^2 + \frac12\norm{(\delt\phi_\lambda,\delt\psi_\lambda)}_{K,\alpha}^2 \\
        &\quad\leq C_0\big(\norm{(\boldsymbol{v}_\lambda,\boldsymbol{w}_\lambda)}_{\mathcal{H}^1}^2 + \norm{(\mu_\lambda,\theta_\lambda)}_{L,\beta}^2\big) + C_0\norm{(\boldsymbol{v}_\lambda,\boldsymbol{w}_\lambda)}_{\mathcal{H}^1}^2\norm{(\mu_\lambda,\theta_\lambda)}_{L,\beta}^2.
    \end{split}
\end{align}
Applying Gronwall's lemma yields
\begin{align*}
    \norm{(\mu_\lambda(t),\theta_\lambda(t))}_{L,\beta}^2 &\leq C_0\norm{(\mu_\lambda(0),\theta_\lambda(0))}_{L,\beta}^2\exp\Big(C_0\int_0^t\norm{(\mu_\lambda,\theta_\lambda)}_{L,\beta}^2 + \norm{(\boldsymbol{v}_\lambda,\boldsymbol{w}_\lambda)}_{\mathcal{H}^1}^2\dtau\Big) \\
    &\qquad + C_0\int_0^t\norm{(\boldsymbol{v}_\lambda,\boldsymbol{w}_\lambda)}_{\mathcal{H}^1}^2\exp\Big(C_0\int_s^t\norm{(\mu_\lambda,\theta_\lambda)}_{L,\beta}^2 + \norm{(\boldsymbol{v}_\lambda,\boldsymbol{w}_\lambda)}_{\mathcal{H}^1}^2\dtau\Big)\ds.
\end{align*}
By using \eqref{Est:vw:ep}, \eqref{Est:ApproxIC:MT:H^1} and \eqref{Est:Gronwall:Refined}, we find
\begin{align}\label{Est:Mt:Lb:infty:refined}
    \begin{split}
        \sup_{t\in[0,T]}\norm{(\mu_\lambda(t),\theta_\lambda(t))}_{L,\beta}^2 &\leq C_0\Big(1 + \norm{(\mu_0,\theta_0)}_{L,\beta}^2 + \int_0^T\norm{(\boldsymbol{v},\boldsymbol{w})}_{\mathcal{H}^1}^2\ds\Big)\\
    &\quad\times\exp\Big(C_0\int_0^T\norm{(\boldsymbol{v},\boldsymbol{w})}_{\mathcal{H}^1}^2\ds\Big).
    \end{split}
\end{align}
Thus, integrating \eqref{PreGronwall:refined} in time from $0$ to $T$, and exploiting \eqref{Est:vw:ep}, \eqref{Est:Gronwall:Refined} and \eqref{Est:Mt:Lb:infty:refined}, we deduce that
\begin{align}\label{Est:PP:delt:Ka:refined}
    \begin{split}
        \int_0^T\norm{(\delt\phi_\lambda,\delt\psi_\lambda)}_{K,\alpha}^2\ds &\leq C_0\Big(1 + \norm{(\mu_0,\theta_0)}_{L,\beta} + \int_0^T \norm{(\boldsymbol{v},\boldsymbol{w})}_{\mathcal{H}^1}^2\ds\Big)\\
        &\quad\times\Big(1 + \int_0^T \norm{(\boldsymbol{v},\boldsymbol{w})}_{\mathcal{H}^1}^2\ds\Big)\exp\Big(C_0\int_0^T \norm{(\boldsymbol{v},\boldsymbol{w})}_{\mathcal{H}^1}^2\ds\Big).
    \end{split}
\end{align}
Finally, using elliptic regularity theory for systems with bulk-surface coupling once more, we get
\begin{align*}
    \norm{(\mu_\lambda - \beta\mean{\mu_\lambda}{\theta_\lambda},\theta_\lambda - \mean{\mu_\lambda}{\theta_\lambda})}_{\mathcal{H}^3} \leq C_0\big(\norm{(\delt\phi_\lambda,\delt\psi_\lambda)}_{K,\alpha} + \norm{(\boldsymbol{v}_\lambda,\boldsymbol{w}_\lambda)}_{\mathcal{H}^1}\big)
\end{align*}
from which we readily conclude that
\begin{align}\label{Est:MT:Grad:H^2:Refined}
    \begin{split}
        \int_0^T \norm{(\Grad\mu_\lambda,\Gradg\theta_\lambda)}_{\mathcal{H}^2}^2\ds &\leq C_0\Big(1 + \norm{(\mu_0,\theta_0)}_{L,\beta} + \int_0^T \norm{(\boldsymbol{v},\boldsymbol{w})}_{\mathcal{H}^1}^2\ds\Big)\\
        &\quad\times\Big(1 + \int_0^T \norm{(\boldsymbol{v},\boldsymbol{w})}_{\mathcal{H}^1}^2\ds\Big)\exp\Big(C_0\int_0^T \norm{(\boldsymbol{v},\boldsymbol{w})}_{\mathcal{H}^1}^2\ds\Big).
    \end{split}
\end{align}
Using the fact that norms are weakly lower semi-continuous, the estimates \eqref{HighReg:Est:MT:Lb:sup} and \eqref{HighReg:Est:PP:Ka:L^2} follow from the convergences shown in Step 7 as well as the estimates \eqref{Est:Mt:Lb:infty:refined}, \eqref{Est:PP:delt:Ka:refined} and \eqref{Est:MT:Grad:H^2:Refined}. This means that all claims are established, and therefore, the proof of Theorem~\ref{Theorem:HighReg} is complete.
\end{proof}

\section*{Acknowledgement}
\noindent
AG was supported by MUR grant Dipartimento di Eccellenza
2023-2027 of Dipartimento di Matematica, Politecnico di Milano, and by ``INdAM-GNAMPA Project", CUP E5324001950001. \\
PK and JS were supported by the Deutsche Forschungsgemeinschaft (DFG, German Research Foundation) – Project 52469428, and partially supported by the Deutsche Forschungsgemeinschaft (DFG, German Research Foundation) – RTG 2339. The support is gratefully acknowledged.

\section*{Conflict of Interests and Data Availability Statement}

There are no conflicts of interest.

There are no data associated with the manuscript.

\bibliographystyle{abbrv}
\bibliography{GKS.bib}
\end{document}